\documentclass[12pt,letterpaper]{amsart}
\usepackage{latexsym} 

\usepackage{amsmath,amsthm,amsfonts,amssymb, mathrsfs,mathabx}

\usepackage[utf8]{inputenc}
\usepackage{url}
\usepackage{enumerate}
\usepackage{mdwlist} 
\usepackage{rotating} 
\usepackage{color} 

\usepackage[letterpaper,textwidth=6.5in,textheight=8in]{geometry} 
\usepackage{todonotes}
\usepackage{hyperref}
\hypersetup{colorlinks, citecolor=blue, linkcolor=black}

\usepackage{graphicx}
\usepackage{transparent}

\usepackage{tikz}
\usetikzlibrary{calc,arrows,decorations.pathmorphing,decorations.pathreplacing}

\newtheorem{thm}{Theorem}[section]
\newtheorem{theo}[thm]{Theorem}

\newtheorem*{thm*}{Theorem}
\newtheorem{cor}[thm]{Corollary}
\newtheorem{coro}[thm]{Corollary}
\newtheorem*{cor*}{Corollary}

\newtheorem{prop}[thm]{Proposition} 
\newtheorem*{prop*}{Proposition} 
\newtheorem*{properties*}{Properties} 
 
\newtheorem{lem}[thm]{Lemma} 
\newtheorem{lemma}[thm]{Lemma} 
\newtheorem*{lem*}{Lemma}

\newtheorem*{claim*}{Claim} 
 
\newtheorem*{fact*}{Fact}

\newtheorem*{qst*}{Question}

\newtheorem*{pb*}{Problem}

\theoremstyle{definition}

\newtheorem{defi}[thm]{Definition} 
\newtheorem*{dfn*}{Definition}

\newtheorem{defn}[thm]{Definition}
\newtheorem{question}[thm]{Question}

\theoremstyle{remark}
 
\newtheorem*{algo*}{Algorithm} 
\newtheorem*{rem*}{Remark}
\newtheorem{rem}[thm]{Remark}

\newtheorem{rmk}[thm]{Remark}
\newtheorem*{example*}{Example}

\newtheorem{conv}[thm]{Convention}

\newcommand{\semidirect}{\ltimes}

\newcommand{\actson}{\curvearrowright}
\newcommand{\ad}{{\rm ad}}
\newcommand{\diam}{\mathop{\mathrm{diam}\;}}

\newcommand {\calA} {{\mathcal {A}}}   
\newcommand {\calB} {{\mathcal {B}}}   
\newcommand {\calC} {{\mathcal {C}}}   
   
\newcommand {\calE} {{\mathcal {E}}}   
\newcommand {\calF} {{\mathcal {F}}}   
\newcommand {\calG} {{\mathcal {G}}}   
\newcommand {\calH} {{\mathcal {H}}}

\newcommand {\calL} {{\mathcal {L}}}   
\newcommand {\calM} {{\mathcal {M}}}   
\newcommand {\calN} {{\mathcal {N}}}   
\newcommand {\calO} {{\mathcal {O}}}   
\newcommand {\calP} {{\mathcal {P}}}   
\newcommand {\calQ} {{\mathcal {Q}}}

\newcommand {\calT} {{\mathcal {T}}}

\newcommand {\calW} {{\mathcal {W}}}

\newcommand {\bbC} {{\mathbb {C}}}   
\newcommand {\bbD} {{\mathbb {D}}}   
   
\newcommand {\bbF} {{\mathbb {F}}}

\newcommand {\bbJ} {{\mathbb {J}}}

\newcommand {\bbN} {{\mathbb {N}}}

\newcommand {\bbR} {{\mathbb {R}}}

\newcommand {\bbX} {{\mathbb {X}}}   
   
\newcommand {\bbZ} {{\mathbb {Z}}}

\newcommand {\lk}{{\rm lk}}\newcommand {\G}{\Gamma}
\newcommand*{\longhookrightarrow}{\ensuremath{\lhook\joinrel\relbar\joinrel\rightarrow}}

\newcommand {\tto}{ \twoheadrightarrow }

\renewcommand{\phi}{\varphi}

\newcommand {\onto} {\twoheadrightarrow}
\newcommand {\into} {\hookrightarrow}

\newcommand{\ol}[1]{\overline{#1}}

\newcommand{\normal} {\vartriangleleft}

\newcommand{\Isom} {{\mathrm{Isom}}}

\newcommand{\Out} {{\mathrm{Out}}}
\newcommand{\Hom} {{\mathrm{Hom}}}
\newcommand{\Aut} {{\mathrm{Aut}}}

\newcommand{\define}[1]{\textit{#1}}

\newcommand{\mo}{{-1}}

\renewcommand{\tilde}{\widetilde}
\newcommand{\Z}{\mathbb{Z}}

\newcommand{\C}{\mathbb{C}}

\newcommand{\Q}{\mathbb{Q}}

\newcommand{\bk}[1]{{\langle #1 \rangle}}
\newcommand{\out}[1]{{\mathrm{Out}\left(#1\right)}}
\newcommand{\outo}{\mathrm{Out}_0}
\newcommand{\nsgp}{\triangleleft}
\newcommand{\charleq}{\triangleleft_c}
\newcommand{\chargeq}{\triangleright_c}
\newcommand{\lucs}{l_{\mathrm{u.c.s}}}
\newcommand{\tffgn}{\ensuremath{\mathfrak{T}}}
\newcommand{\fgn}{\ensuremath{\mathfrak{N}}}
\newcommand{\alg}{\mathscr}
\newcommand{\aut}[1]{\ensuremath{\mathrm{Aut}\left(#1\right)}}

\newcommand{\tron}[1]{\ensuremath{\mathrm{Tr}_1(n,#1)}}
\newcommand{\trzn}[1]{\ensuremath{\mathrm{Tr}_0(n,#1)}}

\newcommand{\malcev}{\sqrt}
\newcommand{\tuple}[1]{\MakeUppercase{#1}}

\newcommand{\diag}[1]{\ensuremath{\mathrm{diag}(#1)}}

\newcommand{\commentOut}[1]{\iffalse #1 \fi \begin{center} ***COMMENTED OUT CONTENT***\end{center}}

\newcommand{\QH}{\ensuremath{{(QH)}}}
\newcommand{\orbfun}[1]{\ensuremath{{\pi_1^{\mathrm{orb}}\left(#1\right)}}}
\newcommand{\orbbdy}[1]{\ensuremath{{\partial^{\mathrm{orb}}#1}}}
\newcommand{\topbdy}[1]{\ensuremath{{\partial^{\mathrm{top}}#1}}}
\newcommand{\dihedral}{\ensuremath{\bbZ_2*\bbZ_2}}
\newcommand{\RQH}{\ensuremath{{(rQH)}}}
\newcommand{\tp}[2]{\ensuremath{\mathrm{tp}^{\mathrm{orb}}_{#1}{\left(#2\right)}}}
\newcommand{\ncl}[1]{\bk{\bk{#1}}}
\newcommand{\tfi}[1]{{\mathbf{i}_{\mathrm{t.f.}}\left({#1}\right)}}

\tikzstyle{bvertex}=[circle, fill, inner sep=0pt, minimum size=6pt]
\tikzstyle{wvertex}=[circle, draw, inner sep=0pt, minimum size=6pt]

\begin{document}

\title{Deciding Isomorphy using Dehn fillings, the  splitting case}
\author{François Dahmani} 
\address{François Dahmani\\ Univ. Grenoble Alpes,  CNRS, Institut Fourier  \\
    F-38000 Grenoble
  (France)} \email{francois.dahmani@univ-grenoble-alpes.fr}
\author{Nicholas  Touikan} 
\address{Nicholas Touikan\\
  Department of Mathematics and Statistics\\
  University of New Brunswick\\
  P.O. Box 4400, Fredericton, New Brunswick\\
  E3B 5A3 (Canada)\\
} \email{nicholas.touikan@unb.ca}
\thanks{During the preparation of this article, the first author was
  supported by the ANR grant   2011-BS01-013-02, and the Institut
  Universitaire de France, and the second
  author was supported by an NSERC PDF, ANR-2010-BLAN-116-01 GGAA, and a Fields postdoctoral
  fellowship.}

\begin{abstract}  We solve Dehn's isomorphism problem for virtually
  torsion-free relatively hyperbolic groups with nilpotent parabolic
  subgroups.

  We do so by reducing the isomorphism problem 
  to three algorithmic problems in the parabolic subgroups, namely the
  isomorphism problem, separation of torsion (in their outer
  automorphism groups) by congruences, and the mixed Whitehead
  problem, an automorphism group orbit problem.  The first step of the
  reduction is to compute canonical JSJ decompositions. Dehn
  fillings and the given solutions of the algorithmic problems in the
  parabolic groups are then used to decide if the graphs of groups
  have isomorphic vertex groups and, if so, whether a global
  isomorphism can be assembled.

  For the class of finitely generated nilpotent groups, we give
  solutions to these algorithmic problems by using the arithmetic
  nature of these groups and of their automorphism groups.
\end{abstract}
\maketitle

\vspace{.3cm}

\begin{flushright}

{\small {\it Le dual reste loin, solitaire et plaintif, \\ Cherchant
    l'isomorphie et la trouvant rebelle.}}

\vspace{0.2cm}

{\small {André Weil}}

\end{flushright}
\vspace{0.2cm}
\tableofcontents

\section{Introduction}


\subsection{The isomorphism problem in geometric group theory, and algebra}

The isomorphism problem for finitely presented groups 
 asks for an algorithm that, given
two finite presentations, decides whether they define isomorphic
groups. It is one of the main classical algorithmic problems in group theory, and
understanding which classes have solvable isomorphism problem is 
subtle,  difficult, and important as a theoretical classification and structuration device. To this date,  the
approach of geometric group theory  has revealed itself to be one of the few
fruitful strategies.

In a series of works, the isomorphism problem was studied for certain
classes of groups that exhibit  a negatively curved
geometry. Sela
\cite{Sela-1995} introduced an elegant strategy for solving it for
a crucial class of torsion-free  hyperbolic groups, based on the solvability of
equations in these groups, and on a rigidity
criterion relating solutions of equations and splittings of the
group. 
 Bumagin, Kharlampovich and Miasnikov
\cite{BKM-Isomorphism} solved the problem for  limit groups,
a class of relatively hyperbolic groups. Groves and the first author
\cite{DGr_ihes}, following and simplifying Sela's strategy, completed
his original solution  into a
solution for all torsion-free hyperbolic groups.
Subsequently,  Guirardel and the first author \cite{DG_gafa}, also following Sela's
strategy, overcame the difficulties posed by torsion and gave a solution to the isomorphism problem for {all}
hyperbolic groups.

In a more algebraic coloration, a solution of the isomorphism problem
for virtually nilpotent groups, was found by Grunewald and Segal
\cite{GS2}, and for virtually polycyclic groups, by Segal
\cite{Seg}. The first of these studies relies on the fascinating fact
that automorphism groups of finitely generated nilpotent groups are
always arithmetic groups \cite{Auslander-1969}, and the solution is
a reduction to an orbit problem for a rational action of such an
arithmetic group on an algebraic variety. The solution to the
isomorphism problem for polycyclic groups is similar in method, but
significantly more involved.

\subsection{On relatively hyperbolic groups}

The case of  relatively hyperbolic groups is  natural and
appealing due to its many examples.  Relatively hyperbolic groups are
algebraic (or coarse) analogues of fundamental groups of finite volume negatively
curved manifolds, and their peripheral subgroups are the analogues of
the groups of the cusps of these manifolds. In the case of hyperbolic
manifolds, these peripheral subgroups are virtually abelian. In the case of pinched
negative curvature, these peripheral subgroups are, by the
Margulis Lemma, virtually nilpotent. In the generality of relatively
hyperbolic groups, they can be anything.

In the simplest example of relatively hyperbolic groups, the free
products of arbitrary peripheral groups, it appears that one will need
to restrict the class of peripheral subgroups in order to proceed on
the isomorphism problem. The most natural and tame class of peripheral
subgroups is certainly the class of abelian groups, and indeed, Groves
and the first author have proved the isomorphism problem to be
solvable for torsion-free groups that are hyperbolic relative to
abelian subgroups \cite{DGr_ihes} (the so-called toral relatively
hyperbolic groups.)  This class is already somewhat satisfying in
several aspects because it contains limit groups as well as the
fundamental groups of finite volume hyperbolic manifolds. Toral
relatively hyperbolic groups also correspond to the rare class with
peripheral subgroups in which the problem of equations is known to be
decidable. It is far from being completely satisfying though, as it
doesn't cover the case of finite volume manifolds with pinched
(non-constant) negative curvature.

For these later groups, there are fundamental difficulties, sometimes
proven impossibilities, to carrying out the Sela's strategy.
Nevertheless, we manage to prove the following result.

\begin{theo}\label{theo;ip-nilpotent}
  There is an algorithm which, given two presentations $\bk{X\mid R}$,
  $\bk{Y \mid S}$ of virtually   
  torsion-free groups that are hyperbolic relative
  to a family  of finitely generated nilpotent groups, decides if the
  presentations yield isomorphic groups.
\end{theo}

As a particular case, our approach provides a new, logically
independent proof of an earlier result of Groves and the first named
author, which was the main result of \cite{DGr_ihes}.

\begin{theo}[{\cite[Theorem A]{DGr_ihes}}] \label{theo;ihes}
  The isomorphism problem is solvable for toral relatively hyperbolic
  groups.
\end{theo}

\subsection{A change of strategy}

Sela's solution  involves pursuing isomorphisms in sets of solutions of systems
of equations and inequations. As soon as nilpotent groups are introduced this approach fails due to
the undecidability of their universal theories
\cite{Romankov-undecidable}. This is a critical obstruction;
furthermore this phenomenon is far from exceptional, in fact it
appears to be generic within the class of finitely generated nilpotent
groups \cite{duchin2015equations,garreta2016properties}. 

 Given this theoretical obstruction   Guirardel and the first author proposed a change in 
strategy, and
showed in \cite{DG_charac}  that the isomorphism problem is decidable
in the class of
relatively hyperbolic groups with residually finite peripheral
subgroups, that are \emph{rigid} (in the sense that they do not have
elementary decomposition as amalgamation or HNN extension).
This is achieved through     an application of  the
Dehn filling construction performed in \cite{DGO} (see originally \cite{Osi_DF},
and \cite{GrM_DF} for the Dehn Filling theorem), that enables  
 a reduction of the problem to an already solved case (namely
the case of hyperbolic groups with torsion). This technique amounts to
rigorously solving the isomorphism problem by using algebraic and geometric approximations.

We follow the strategy proposed by  \cite{DG_charac}.
 However, in  \cite{DG_charac}  only
rigid relatively hyperbolic groups are shown to be characterized by
Dehn fillings, and counterexamples are shown in the non-rigid case. A
particular attention must be granted to the splitting case: when the
groups have decompositions (or splittings) as amalgamated free products or HNN
extension over elementary subgroups. 

\subsection{Canonical decompositions}

 In some  cases, canonical spaces of elementary splittings exist. For instance when the groups
are torsion-free, they have  Grushko decompositions into free
products, whose isomorphism type is uniquely defined.  
If the groups are
freely indecomposable, Guirardel and Levitt developed the theory of
 JSJ decompositions over virtually cyclic or parabolic subgroups, and this allows to
 obtain a canonical decomposition as a finite graph of groups.
 Computing such group
invariants is a crucial step in our solution to the isomorphism problem.

Such canonical decompositions must be maximal in some sense and
computing them is difficult in general.
 In all previous results, such computations were done using
{equationnal} techniques: the guarantee that a decomposition is maximal
is given by the lack of solutions to certain equations and inequations.
 Again, we encounter the difficulty of undecidability of universal
 theory of some nilpotent groups.

Instead we compute the JSJ decomposition using work of the second
author \cite{Tou}, which avoids having to solve equations, and the
Guirardel Levitt characterizations of JSJ splittings
\cite{2016arXiv160205139G}; the resulting algorithm has a certain
structural simplicity. The following is the combination of Theorem
\ref{thm;compute_JSJ} below, and of Corollary
\ref{cor:ctf-compute-can-splittings}, which is a consequence of
\cite[Theorem C]{Tou}, which is recalled here as Theorem
\ref{theo;TouC}.

We first need a technical, but natural definition.

\begin{defi}[Heriditarily algorithmically tractable and
	algorithmically bounded torsion class]\label{def;tract} Let $\calC$
	be a
	class of groups.
	\begin{enumerate}
		\item\label{it:AT} We say $\calC$ is \define{algorithmically
			tractable} if the presentations of the groups in $\calC$ are
		recursively enumerable and there is a uniform solution to the 
		conjugacy problem and the generation problem, i.e. to decide if a
		tuple of elements generate the group, over these presentations.
		\item\label{it:HAT} We say $\calC$ is \define{heriditarily
			algorithmically tractable} if it is algorithmically tractable and
		closed under taking subgroups in the following way:
                given a finite
		generating set $S$ of a subgroup $\bk{S} \leq H,$ where $H \in \calC$,
		then $\bk{S} \in \calC$ and there is a uniform procedure to
		construct a presentation for $\bk{S}$ with generators in $S$. This
		property is called \emph{effective
			coherence}. 
                    \item Finally we say $\calC$ has
                      \define{algorithmically bounded torsion} it is
                      algorithmically tractable and there is a uniform
                      algorithm which, for every presentation
                      $\langle S \mid R\rangle$ of a group in $\calC$,
                      produces $\calF_{\langle S \mid R\rangle}$, a
                      finite list of words that contains a conjugacy
                      representative of each torsion element of
                      $\langle S\mid R\rangle$.
	\end{enumerate}
\end{defi} 

Algorithmically tractable was defined in \cite{Tou} and is required to
determine if a relatively hyperbolic group is rigid. It follows from
\cite{BCRS-PF} that the class of virtually polycyclic groups satisfy
all three items of this definition.

\begin{thm}[Theorem \ref{thm;compute_JSJ}, 
  Corollary \ref{cor:ctf-compute-can-splittings}]
Let $\calC$ be a  hereditarily algorithmically tractable class of groups  with algorithmically bounded torsion. 
%
%
There is an algorithm to find the canonical JSJ decomposition
  of any one-ended virtually torsion free relatively hyperbolic group $(G,\calP)$, with $\calP$ in $\calC$.

\end{thm}


Even once one has computed canonical JSJ splittings for two groups
$G_1, G_2$, in the prospect of deciding whether they are isomorphic,
and even if one has computed that vertex groups of one decomposition
are isomorphic to the vertex groups of the other, one still doesn't
know whether $G_1$ and $G_2$ are isomorphic.  The problem of
assembling isomorphisms between vertex groups into global isomorphisms
of graphs of groups introduces new and interesting subtleties. We
illustrate this with an example.

\subsection{An example}

Let $G_a, G_b$ be groups that are hyperbolic relative to some subgroup  $P_a, P_b$
respectively,  non-isomorphic, rigid,  and with trivial outer
automorphism group.  Now let $P$ be a group in which $P_a$ and $P_b$ each
embed in two different ways via $i_k:P_a \to P$, $j_k: P_b\to 
P$, $k=1,2$. Let $\Gamma_1,\Gamma_2$ be the amalgams \[\Gamma_k = G_a
*_{P_a} P *_{P_b} G_b\] for the attaching maps $i_k, j_k;
k=1,2$. These groups are relatively hyperbolic, by the Combination
Theorem \cite{Dah_CoC}, and the splitting of $\Gamma_k$ that is exhibited, is the
canonical JSJ splitting of $(\Gamma_k,P)$. It follows that $(\Gamma_1,
P)$ is isomorphic to $(\Gamma_2, P)$ if and only if there is an
automorphism $\alpha$ of $P$ such that the maps $\alpha\circ i_1$ and
$\alpha\circ j_1$ are respectively conjugated in $P$ to $i_2$ and to
$j_2$.   
 If we fix a tuple of generators for the edge groups,
deciding the existence of such an automorphism $\alpha$ is an instance of the
{mixed Whitehead problem} in $P$ (see Definition \ref{def:mwhp}.) We will in fact solve the mixed
Whitehead problem for nilpotent groups, using the powerful theory of
\cite{GS1} and the arithmeticity  of the automorphism  group of  a finitely
generated nilpotent group  and of its holomorph.

If we allow $(G_a, P_a)$ and $(G_b, P_b)$ to have
 non-trivial outer
automorphisms, the situation is further complicated: $(\Gamma_1,P)$ is
isomorphic to $(\Gamma_2,P)$ if and only if there are automorphisms
$\beta_a \in \aut{G_a, P_a},\beta_b \in \aut{G_b, P_b}$ and $\alpha \in \aut{P}$
such that the following hold:
\begin{eqnarray} \alpha \circ i_1 \circ \beta_a|_{P_a} &\sim_P& i_2\label{e:conj1}\\
  \alpha \circ j_1 \circ \beta_b|_{P_b} &\sim_P& j_2\label{e:conj2}
\end{eqnarray} where $\sim_P$ denotes conjugacy in $P$. Note that a
choice of $\alpha$, to satisfy (\ref{e:conj1}), depends on $\beta_a$,
and a choice of $\beta_b$, to satisfy (\ref{e:conj2}), depends on
$\alpha$; all three automorphisms are {interdependent}.

There is some comfort to be found in the fact the outer automorphism
groups of rigid relatively hyperbolic groups are {finite}, thus
leaving finitely many possibilities for $\beta_a,\beta_b$ up to conjugacy. But
unfortunately, we still do not know how to effectively compute this
finite group in general. 

A closer look at the previous example reveals
that knowing $\beta_a$ and $\beta_b$ is not necessary. What matters is
their restriction to $P_a$ and $P_b$ respectively, up to conjugacy in
$P_a$. 
 In other words, we
are interested in the image of the natural morphism $\out{G_a, [P_a]} \to
\out{P_a}$, which, as we will see,  is well defined in our case, and which is a finite
subgroup of $\out{P_a}$.     In order to identify that image (which is
done in Proposition \ref{prop;orbit_computation}, see Proposition \ref{prop;intro}
in the discussion below for a first account),  we will need a
property that we call having congruences separating the torsion.
In Theorem \ref{thm:find-deep-enough},   we ensure that, in the cases in which we are
interested, this property is satisfied, and for that,  we take advantage of the fact
that, if $P_a$ is  
finitely generated nilpotent, its outer automorphism    
  group $\out{P_a}$ is arithmetic (a condition
that actually holds in the much broader class of polycyclic-by-finite
groups, by \cite{Baues_Grunewald}.)

\subsection{Dehn fillings, and congruences separating the torsion}

Given a relatively hyperbolic group $(G, P)$, and a subgroup $N$ of
$P$, normal in $P$, the operation of quotienting $G$ by the normal
closure of $N$ (in $G$) is a Dehn Filling. If $N$ avoids a certain
finite set of non-trivial elements, then Osin's hyperbolic Dehn
filling theorem ensures that the quotient
$G/\langle\langle N \rangle\rangle$ is hyperbolic relative to $P/N$
\cite{Osi_DF}\cite{GrM_DF}. The geometry of this construction was
studied in \cite{DGO}, and in \cite{DG_charac} it was used to show
that, if $P$ is residually finite, and if $(G, P)$ is rigid, then a
certain characteristic sequence of Dehn fillings of $(G, P)$
characterizes $(G, P)$ up to isomorphism.  In this paper, we further
show that, if $P$ has finite characteristic quotients for which the
congruence in $\out{P}$ has torsion-free kernel, then the sequence of
Dehn fillings characterizes the image of $\out{G, [P]} \to
\out{P}$. The arithmeticity of outer automorphisms groups of nilpotent
groups will permit us to show that these groups satisfy this property,
that we call {having congruences subgroups separating the
  torsion}. For instance this property is a classical theorem of
Minkowski for $\bbZ^n$, namely $GL(n,\bbZ)\to GL(n,\bbZ/3\bbZ)$ has
torsion-free kernel.  For the purpose of the isomorphism problem, we
will need an effective version of this property, and we will prove
that one can effectively compute a finite congruence separating the
torsion for any torsion-free nilpotent group.  A statement that gives
a flavour of what we must accomplish is the following:

\begin{prop}\label{prop;intro}
  Let $G$ be a finitely generated group hyperbolic relative to
  $P<G$. Assume that $P$ is residually finite, and that congruences
  separate the torsion in $\Out(P)$. Let $S$ be a generating
  tuple     
  of $P$. Then, there exists a finite index
  characteristic subgroup $N \charleq P$, such that the quotient map
  to the Dehn filling $G \to \bar G = G/\langle\langle N\rangle
  \rangle$ induces an isomorphism
  \[  \Out(G, [P])/\Out(G, [S]) \to \Out(\bar G, [\bar P])/ \Out(\bar G,
  [\bar S]).\]
\end{prop}

\subsection{Main results}

 We now state our main result.


\begin{theo}\label{theo;main} Let $\calC$ be a hereditarily   algorithmically
  tractable   class   with algorithmically bounded torsion, 
   satisfying the following properties:
  \begin{itemize}
  \item all groups in $\calC$ are residually finite,
  \item the isomorphism problem is 
   solvable in $\calC$,
  \item in $\calC$, congruences effectively separate the
    torsion (see Definition \ref{def;CST}),
  \item the mixed Whitehead problem is effectively solvable in $\calC$
    (see Definition \ref{def:mwhp}.)
  \end{itemize}
  There is an algorithm which decides if two explicitly given
  virtually torsion-free relatively hyperbolic groups $(G,\calP)$, $(H,\calQ)$
  whose peripheral subgroups belongs to $\calC$, are isomorphic as
  groups with unmarked peripheral structure.
\end{theo}

It is sometimes helpful to allow some peripheral subgroups to be
virtually cyclic, even if this could cause them to fall outside the class
$\calC$. This is harmless, since all the properties above
are classically true for virtually cyclic groups, and we will allow
extending the peripheral structure to include some maximal virtually
cyclic groups.

Theorem  \ref{theo;main} can be applied to the class
$\calC=\calC_{ab}$ of finitely generated 
abelian groups, and implies in particular the result of Theorem \ref{theo;ihes}.  The fact that
congruences subgroups effectively separate torsion is, as we
mentioned, a theorem of Minkowski. Instances of the mixed Whitehead
problem (i.e. how to bring one tuple of vectors to another tuple) can
be solved using the standard methods for modules over principal ideal
domains. The other properties required for $\calC$ in Theorem
\ref{theo;main} are classical.

But Theorem  \ref{theo;main} can also be applied for the class $\calC$
of  finitely generated  nilpotent groups, arguably the  next most natural class
$\calC$ of parabolic subgroups
of relatively hyperbolic groups.   This class
of groups (in fact, the class of virtually polycyclic groups) is
already known to be  heriditarily algorithmically tractable with algorithmically bounded torsion 
\cite{BCRS-PF}.

\begin{theo}[{Theorems
    \ref{thm:find-deep-enough}}, and \ref{thm:mwhp}] \label{theo;nilp}
  Finitely generated nilpotent groups have congruences that
  effectively separate the torsion in their outer automorphism group,
  and have solvable mixed Whitehead problem.
\end{theo}

From this theorem, and from \cite{GS2}, the five assumptions of  Theorem \ref{theo;main} on peripheral
subgroups are therefore satisfied for finitely generated nilpotent
groups. Theorem \ref{theo;main} and  Corollary
\ref{cor;canonicity} then imply Theorem \ref{theo;ip-nilpotent}.

Theorem \ref{theo;main}, as well as the methods to prove it, are
general and it is likely that in the future they will be applicable to
even larger classes of relatively hyperbolic groups.  The limitations
that constrain us to remain in the class of virtually torsion-free
relatively hyperbolic groups with nilpotent parabolics motivate the
following questions:

\begin{question}\label{qu:cong-sep+mwhp}
  Are there other classes of groups that are residually finite, hereditarily algorithmically tractable
  with algorithmically bounded torsion,
  where Theorem \ref{theo;nilp} holds? Is this the case
  for virtually polycyclic groups?
\end{question}
Any positive answer to Question \ref{qu:cong-sep+mwhp} will enlarge
the class of peripheral subgroups of relatively hyperbolic in which we
can solve the isomorphism problem.
\begin{question}\label{qu:detect-rigid-torsion}
  Is there an algorithm to detect the rigidity of relatively
  hyperbolic groups that works in the presence of torsion?
\end{question}
A positive answer to Question \ref{qu:detect-rigid-torsion}, or a
positive answer to the long standing question of whether hyperbolic
groups are residually finite, would enable us to remove the
``virtually torsion-free'' clause from the statement of Theorem
\ref{theo;main}.

\subsection{Algorithmic problems in the automorphism groups of
  nilpotent groups}

The proof of Theorem \ref{theo;nilp} has a definitely more algebraic
flavor. Let us briefly discuss what we do.

Given $N$ a finitely generated nilpotent group, the mere existence of
a congruence $K\normal N$ separating the torsion in ${\rm Out }(N)$ is
non-trivial. A property of profinite topology for virtually polycyclic
groups due to Ribes, Zalesski\u{\i}, and Segal \cite{RSZ} states that
in these groups, the closure of the centralizer of a subgroup (in the
profinite completion) is equal to the centralizer of the subgroup in
the closure. For any finite order outer-automorphism $[\alpha]$ of
$N$, the application of this property to the semi-direct product
$\mathbb{Z} \semidirect N $ with structural automorphism $\alpha$
shows that in deep enough characteristic finite quotient of $N$,
$\alpha$ is not inner (Proposition \ref{prop:segal}.)  The
arithmeticity of ${\rm Out }(N)$ implies that there are only finitely
many conjugacy classes of finite subgroups in it. Applying the above
argument to each conjugacy class of finite order outer-automorphisms
produces a finite family of characteristic finite index subgroups of
$N$, and their intersection is a congruence separating the torsion.

 However, as we hinted, 
there is still one difficulty left, and that is the computability of
such a suitable congruence subgroup. This amounts to the computability
of a complete set of conjugacy classes of finite subgroups in ${\rm
  Out }(N)$. Using the upper central series, we can reveal some
classes by
induction on the nilpotency rank, but there might be automorphisms of
finite order in ${\rm Out} (N)$,  that are not visible through
this induction. We call them elusive. Their analysis and computation
is done by means of commutative algebra in Proposition
\ref{prop:coset-list}, and Corollary \ref{cor:construct-elusive}. 

The mixed Whitehead problem, which is the object of the second half of
Theorem \ref{theo;nilp}, has an already interesting particular case:
given two tuples $(s_1, \dots, s_k)$ and $(t_1,\dots, t_k)$ of
elements in $N$, one needs to find an automorphism $\alpha$ of $N$
such that $\alpha(s_i) = t_i$ for each $i$ (or prove that there is not
such $\alpha$). In the general case, one has several tuples
$S_1, \dots, S_r$, $T_1, \dots T_r$ on each side, and one accepts that
$\alpha$ sends each tuple $S_i$ to a conjugate of $T_i$, the
conjugator (in $N$) being possibly different for different tuples
$S_j$, see Definition \ref{def:mwhp}. In the simpler case, one can use
the orbit decidability of Grunewald and Segal \cite{GS1}. This is
possible by virtue of the arithmeticity of the automorphism group of
$N$, a result that dates back to \cite{Auslander-1969}, and due to the
fact that its action on tuples of elements is a rational action on an
algebraic variety.  The situation in the general case is more
complicated because it is a two quantifier problem (one quantifier
being on the automorphism group of $N$, the other being on elements of
the group $N$.)  We will remedy this by re-expressing the mixed
Whitehead problem as an orbit problem for a certain semidirect product
(see Proposition \ref{prop:orbit-reduction}) which turns out, in the
nilpotent case, to be explicitly arithmetic (Proposition
\ref{prop:explicit-semidirect}.)

 We remain tantalized, with legitimate reason (i.e. Question \ref{qu:cong-sep+mwhp}), by the case of virtually nilpotent
groups, or even polycyclic-by-finite groups.  
Although our arguments for Theorem \ref{theo;nilp}  do not generalize to these
groups, it is difficult to ignore the fact that polycyclic-by-finite groups have congruences that
  separate the torsion in their outer automorphism group (Proposition
  \ref{prop:segal}) as well as arithmetic outer automorphism groups \cite{Baues_Grunewald}.

\subsection{Acknowledgement}
Both authors wish to warmly thank Dan Segal, who kindly explained how
to use a key feature of polycyclic groups, proved in \cite{RSZ}, and
Vincent Guirardel, who, among other discussions, showed us how to
simplify our original argument for section \ref{sec;MWP_orbit_pb}. The
authors are also extremely grateful for the anonymous referee's
numerous and insightful comments, suggestions, corrections, and
warnings. The majority of these have lead to improvements to the
paper.

\subsection{Organization}
In Section \ref{sec:setting} we gather much of the setting, give
terminology, as well as some basic results. In Section
\ref{sec;canonical-splittings}, we discuss canonical splittings of
relatively hyperbolic groups, following Guirardel and Levitt's general
treatment of the JSJ theory. We explain how to compute, using the
machinery of \cite{Tou}, and arguments developed in \cite{DG_gafa}, a
Grushko decomposition for some torsion-free relatively hyperbolic
groups, and a Dunwoody-Stallings decomposition, as well as a canonical
JSJ splitting for some virtually torsion-free relatively hyperbolic
groups.  In Section \ref{sec;gog-isos}, we give a general discussion
of the isomorphism problem for graphs of groups, its relations with
orbit problems in the automorphism groups of vertex groups, and
formulate a useful reduction for the relatively hyperbolic setting
(this reduction actually differs from that in \cite{DG_gafa}.)  In the
case of one-ended relatively hyperbolic groups, we single out the
orbit problem on marked parabolic tuples in rigid vertex groups and
the mixed Whitehead problem in parabolic vertex groups as the two
natural algorithmic problems needed to decide if we can assemble a
collection of ``local'' isomorphisms of vertex groups to a global
isomorphism of graphs of groups.  In Section \ref{sec;orbits}, we
explain a key feature of this paper, namely how to use congruence
subgroups and Dehn fillings to solve some orbit problems given in the
previous section. We provide the proof of Theorem \ref{theo;main} in
section \ref{sec;Solving}. This is done in two steps, first the case
of one-ended relatively hyperbolic groups, and second the reduction of
the general case to the one-ended case. This second reduction is
immediate in the torsion-free case. In the presence of torsion,
however, the argument is more involved: we must work with
Dunwoody-Stallings decompositions and some new results about
normalizers of finite subgroups need to be proved. In Section
\ref{sec;nilp}, we prove Theorem \ref{theo;nilp}. This last section is
somewhat different from the rest of the paper because it uses
algebraic techniques for nilpotent groups, profinite methods, and the
arithmeticity of nilpotent groups and their automorphism groups.

\section{Setting}\label{sec:setting}
\subsection{Peripheral structures, markings, and
  automorphisms}\label{sec:periph-struct}

Let $G$ be a group.  
An \define{unmarked peripheral structure} on $G$ is a
finite disjoint union of conjugacy classes of subgroups:\[ \calP =
[P_1]\sqcup \ldots \sqcup [P_n]
\] where $P_i$ is a subgroup and $[P_i] = \{gP_ig^\mo | g \in G\}$.
  It is also convenient to specify an unmarked
peripheral structure as follows $\calP = \{[P_1], \ldots,
[P_n]\}$.
  A subgroup $H \leq G$ is called
\define{peripheral} if $H \in \calP$.

An \emph{unmarked ordered peripheral structure} $\calP_{uo}$ on $G$ is
a tuple $\calP_{uo}=([P_1], \dots, [P_n])$, where $P_i$ and $P_j$ are
not conjugate in $G$ if $i\neq j$. Any unmarked ordered peripheral
structure produces an unmarked peripheral structure via: \[ ([P_1],
\dots, [P_n]) \mapsto [P_1]\sqcup \ldots \sqcup [P_n].
\] Conversely, an unmarked peripheral structure $[P_1]\sqcup \ldots
\sqcup [P_n]$ produces a choice of $n!$ unmarked ordered peripheral
structures.

A \emph{marked peripheral structure} $\calP_m = (S_1, \dots,
S_n)$   
 is a tuple of tuples of $G$ (each
$S_i$ is a tuple of $G$.) The underlying unmarked ordered peripheral
structure $\calP_{uo}$ is \[\calP_{uo} =([\langle S_1\rangle], \dots,
[\langle S_n\rangle] ).\] We also say that $\calP_m$ is a
\define{marking} of $\calP_{uo}$.

Given a group $G$ with an unmarked ordered peripheral structure
$\calP_{uo}$, the group $\Aut(G, \calP_{uo})$ is the group of
automorphisms of $G$ that preserves the tuple $\calP_{uo}$. In other
words, it is the group of automorphisms that send each $P_i$ to a
conjugate of itself in $G$.  
Let us emphasize that the conjugators can
be different for each $P_i$. Inner automorphisms of $G$ form obviously
a normal subgroup of $\Aut(G, \calP_{uo})$, and $\Out(G, \calP_{uo})$ is defined
as usual by quotienting $\Aut(G, \calP_{uo})$ by this subgroup.

Given a marked peripheral structure $\calP_m= ( S_1, \dots, S_n)$ in
$G$, the group $\Aut(G, \calP_m)$ consists of the automorphisms of $G$
sending the marking tuple $S_i$ to a \emph{conjugate} of itself (again
conjugators can be different for different indices $i$). The group
$\Aut(G, \calP_m)$  contains all inner automorphisms, and  $\Out(G,
\calP_m)$ is defined as the quotient of  $\Aut(G, \calP_m)$ by the
group of inner automorphisms. If $\calP_u$ is the unmarked ordered
structure underlying $\calP_m$, there is a natural inclusion $\Out(G,
\calP_m) \hookrightarrow \Out(G, \calP_{uo})$ as a subgroup.

\subsection{Relatively hyperbolic groups}

Our definition of relative hyperbolicity will be in terms of the
hyperbolicity of a space $X$ constructed from a Cayley graph.

\begin{defn}[Horoballs and depth. See
  {\cite{GrM_DF}}]\label{defn:comb-horoball}
  Let $(Y,d_y)$ be a discrete metric space. A \emph{combinatorial
    horoball over $Y$} is a 1-complex obtained by taking
  $Y\times \bbR_{\geq 0}$ and for each integer $n$ and $y_1,y_2 \in Y$
  adding an edge between $(y_1,n),(y_2,n)$ if and only if
  $d_Y(y_1,y_2)\leq 2^n$. Such an edge is at \define{depth} $n$. A
  \define{horoball of depth $m$} is the union of  all edges of depth
  at least $m$ and points $(y,r), r\geq m$.
\end{defn}

Such horoballs exponentially distort the metric on $Y$, while keeping
it proper. A common variation on definitions by Gromov and Bowditch of
relatively hyperbolic groups is the following (see for instance {\cite{GrM_DF}}).

\begin{defn}[Relatively hyperbolic groups]\label{def;RHG}
  Let $(G,\calP)$ be a group equipped with a peripheral structure and
  let $(C_{G,S},d)$ be a Cayley graph for $G$ associated to some
  finite generating set $S$ of $G$ equipped with the path metric
  $d$. A \define{Groves-Manning space with combinatorial horoballs}
  $X$ is obtained from $C_{G,S}$ by choosing conjugacy representatives
  of the subgroups in $\calP$, and attaching a combinatorial horoball
  to each left coset of these peripheral subgroup, equipped with the
  subspace metric. The pair $(G,\calP)$ is relatively hyperbolic if
  $X$ is $\delta$-hyperbolic for some $\delta$.
\end{defn}

 We remark that the action of $G$ on $X$, in this case,  is free and proper, but not cocompact in general. 
 

\begin{defn}
  If $(G,\calP)$ is relatively hyperbolic any subgroup $H \leq G$ that
  is contained in a peripheral subgroup is called \define{parabolic}.
\end{defn}

In particular \define{maximal parabolic} subgroups are the peripheral
subgroups for the structure $\calP$, and correspond exactly to the
stabilizers of horoballs.  It is classical that, in a relatively
hyperbolic group, there are at most finitely many conjugacy classes of
peripheral subgroups. We will also often consider that the peripheral
structure given is {\it unmarked ordered}, thus making an implicit
choice.

The following is standard, (and an easy exercise on horoball stabilizers, in our context).   See for instance
\cite[Theorem 1.4]{Osi}.  

\begin{prop} \label{prop;parab_own_norm} Let $(G, \calP)$ be a
  relatively hyperbolic group. Then, each  group in $\calP$ is
  almost malnormal in the sense that  if $P \in \calP$, and if $P\cap gPg^{-1}$ is  infinite, then $g\in P$.
\end{prop}


\subsection{Finding parabolic subgroups}

\begin{defn}\label{defn;given}
  A relatively hyperbolic group $(G,\calP)$ is \define{barely given}
  if we are given a finite presentation of $G$, it is
  \define{explicitly given} if we are given a finite presentation of
  $G$ and a marking of its peripheral structure tuple, and it is
  \define{strongly explicitly given} if it is explicitly given and a
  finite presentation over the marking is given for each
  representative of peripheral subgroup.
\end{defn}

A class of groups is said to be recursive if there exists an algorithm
enumerating precisely the finite presentations of the groups in this
class. The following result is useful for all theoretical algorithmic
purposes that have to be solved from a mere presentation of a
relatively hyperbolic group and a solution to its word problem.

\begin{theo}[{\cite[Theorem 3]{DG_PPS}}] \label{theo;presentation}
  Let $\calC$ be a recursive class of finitely presented groups.

  There is an algorithm that, given a presentation of a group, and a
  solution to its word problem, terminates if and only if the group is
  relatively hyperbolic group with respect to subgroups in the class
  $\calC$, and outputs a finite generating set and a finite
  presentation of each representative of maximal parabolic group, and
  an isoperimetry constant for the relative isoperimetric inequality.
\end{theo}

Since the parabolic subgroups will be assumed to be residually finite
we can always solve the word problem:
 
\begin{lem}[{ \cite[Corollary
    6.6]{DG_charac}}]\label{lem;word-problem}
  The word problem is uniformly solvable in the class of finitely
  presented relatively
  hyperbolic groups with residually finite peripheral subgroups. 

\end{lem}



In other words, for all theoretical purposes (under the suitable
assumptions), from a barely given relatively hyperbolic group we can
get a strongly explicit presentation. That being said, these
explicitly given presentations may not be \emph{canonical}. We
therefore need the following extra results:

\begin{lem}[Corollary of {\cite[Lemma 5.4]{Osi}}] 
  \label{lem;parabolics-undistorted}
  Let $(G,\calP)$ be relatively hyperbolic. Then every peripheral subgroup $P \in
  \calP$ is undistorted in $G$.
\end{lem}

\begin{thm}[{\cite[Theorem 1.8]{Drutu-Sapir}}]\label{thm;induced-parabolics}
  Let $G = \bk{S}$ be a finitely generated group that is hyperbolic
  relative to subgroups $H_1,\ldots,H_n$. Let $G'$ be an undistorted
  finitely generated  subgroup of $G$. Then $G'$ is relatively
  hyperbolic with respect to subgroups $H_1',\ldots,H'_m$, where each
  $H_i'$ is one of the intersections $G' \cap gH_jg^\mo$.
\end{thm}

\begin{cor}\label{cor;canonicity-i}
  Let $G$ be equipped with two relatively hyperbolic structures
  $\calP, \calQ$ that both consist of non-virtually cyclic nilpotent
  groups. Then $\calP = \calQ$.
\end{cor}
\begin{proof}
  By Lemma \ref{lem;parabolics-undistorted} and Theorem
  \ref{thm;induced-parabolics} any $Q \in \calQ$ must be hyperbolic
  with respect to its intersection with the groups in $\calP$. Since
  nilpotent groups can only have the trivial relatively hyperbolic
  structure, $Q$ must be contained in a group in $\calP$. By symmetry
  we get $\calQ = \calP$. 
\end{proof}

In fact the same argument goes through for any class of so-called
\emph{not relatively hyperbolic} groups
(c.f. \cite{Drutu-Sapir}). Since we can recognize virtually cyclic
groups, and nilpotent groups are residually finite and recursively
enumerable, we immediately get:

\begin{cor}\label{cor;canonicity}
  Let $\bk{X\mid R}$ be a presentation of a group $G$ that is known to
  have a relatively hyperbolic structure $\calP$ consisting of
  non-virtually cyclic nilpotent groups. Then $\calP$ can be found
  effectively and explicitly.
\end{cor}

\subsection{Finding finite subgroups}

\begin{lem}\label{lem:make-balls}
  Let $(G,\calP)$ be relatively hyperbolic where the groups in $\calP$
  are finitely presented and in which we can solve the word problem.
  Then, given a presentation of $G$, generating sets and presentations
  of representatives of conjugacy classes of $\calP$, it is possible
  to construct arbitrarily large balls in the Groves-Manning space $X$
  (Definition \ref{def;RHG}) and compute a hyperbolicity constant
  $\delta$.
\end{lem}

In order to compute a hyperbolicity constant we will need the
following mesoscopic Cartan-Hadamard theorem by Delzant and Gromov
(although we will use a formulation due to Coulon).  We state a
consequence of the theorem in the special case where $X$ is the
Groves-Manning space.

\begin{thm}[See {\cite[Théorème 4.3.1]{delzant2008courbure}} or
  {\cite[Theorem A.1]{coulon2013small}}]\label{thm:mesoscopic}
  Let $\rho >0$,let $\sigma > 10^7\rho$ and suppose $\sigma$ is
  greater than the longest relation in a finite presentation of
  $G$. If every ball of radius $\sigma$ in $X$ is $\rho$-hyperbolic,
  then $X$ is $\delta=300\rho$-hyperbolic.
\end{thm}

\begin{proof}[Proof of Lemma \ref{lem:make-balls}] 
  First, by \cite{DG_PPS}, one can compute an explicit bound for the
  linear relative isoperimetric inequality, and this allows, given
  $R>0$, to construct the ball of radius $2^R$ of the Cayley graph. By
  the solution (\cite[Theorem 5.6]{Osi}) of the special parabolicity
  problem and the word problem for each $P \in \calP$ we can then 
  partition the vertices of this ball into left $P$ cosets.
  One can then construct the combinatorial horoballs attached to each
  of these cosets. Finally, any element of the group that is at
  distance at most $R$ from the identity for the metric of the
  Groves-Manning space, is in the ball of radius $2^{R/2}$ for the
  word metric; thus we can construct the ball of radius $R$, centered
  at the identity, in $X$. Note that, because $G$ is finitely
  presented we are given a $\sigma$ as in the statement of Theorem
  \ref{thm:mesoscopic}.

  On the one hand any metric ball that is contained inside a
  horoball is known to be $20$-hyperbolic (see \cite[Remark
  3.9]{GrM_DF}) on the other hand we can construct arbitrarily large
  balls about the identity. Since $X$ is Gromov hyperbolic it follows
  by Theorem \ref{thm:mesoscopic} that we will eventually construct a
  sufficiently large ball around the identity to certify a
  hyperbolicity constant $\delta$.
\end{proof}

\begin{lem}\label{lem:ball-of-finites}
  Let $(G,\calP)$ be relatively hyperbolic where the groups in $\calP$
  are finitely presented and in which we can solve the word
  problem. Then we can construct a finite subset $P_2 \subset G$ which
  contains a conjugacy class of every non-parabolic finite subgroup of
  $G$.
\end{lem}
\begin{proof}
  We recall a classical fact, that any finite subgroup of a relatively
  hyperbolic group either is in a peripheral subgroup, or is
  conjugated to a subgroup close to the identity (for a word metric).
  Let $X$ be a Groves-Manning cusped space for $(G,\calP)$ and suppose it
  is $\delta$-thin. By \cite{bogopolskii1996finite} any finite group
  $F$ of isometries of $X$ has a minimally displaced vertex $x_F$ that
  is moved at most $3\delta+1$ by elements of $F$.

  If $x_F$ is at depth more than $3\delta+1$ inside a horoball then
  $F$ is parabolic. It follows that for any non-parabolic subgroup of
  $(G,\calP)$, the minimally displaced vertex is at depth at most
  $3\delta+1$. Let $P_0$ be the set of vertices at most distance at most
  $3\delta+1$ from the identity. Let $P_1$ be the set of vertices a
  distance at most $3\delta+1$ from $P_0$. And let $P_2$ be the set of
  vertices in the Cayley graph of $G$ (contained in $X$) that are
  connected by paths consisting of vertical edges to points in
  $P_1$. By Lemma \ref{lem:make-balls} the set $P_2$ can be
  constructed algorithmically, and in $P_2$ we can find a copy of a
  conjugate of every non-parabolic finite subgroup of $G$.  
\end{proof}

\begin{cor}\label{cor:decide-non-par-finite}
  Let $(G,\calP)$ be relatively hyperbolic where the groups in $\calP$
  are finitely presented and in which we can solve the word
  problem. Then given a generating set $S$ of a non-parabolic subgroup
  $\bk S = H \leq G$ it is possible to decide if $H$ is
  finite.
\end{cor}
\begin{proof}
  Given $S$ and the solution to the word problem, we enumerate the
  elements of $H$ and construct a multiplication table. If $H$ is
  finite then it must have cardinality less than $\left|P_2\right|$,
  so a complete multiplication table for $H$ will be constructed. On
  the other hand, as we enumerate, if we find that $\left| H \right| >
  \left|P_2\right|$, then we can conclude that $H$ is infinite.
\end{proof}

\subsection{Graphs of groups and trees}
\subsubsection{Splittings} 
\begin{defn}\label{defn;gog}
  A graph $X$ consists of a set $V(X)$ of vertices and a set $E(X)$ of
  oriented edges, with a fix-point free involution $\bar{ }: E(X)\to
  E(X)$ and incidence maps $t:E(X) \to V(X)$, $o:E(X) \to V(X)$ that
  satisfy $o(e) = t(\bar{e})$. If $e$ is an edge, $t(e)$ is called its
  terminal vertex, and $o(e)$ is called its origin vertex. Given $v$ the
  set of edges $e$ such that $t(e)= v$ is denoted by $\lk(v)$.
  \begin{itemize}
  \item A \define{graph of groups} structure $\mathbb{X}$ on a graph $X$ is
    the data of a group $\Gamma_v$ for each vertex $v$, and a group
    $\Gamma_e$ for each edge $e$ such that $\Gamma_e=\Gamma_{\bar e}$,
    and of a monomorphism $i_e : \Gamma_e \to \Gamma_{t(e)}$ for each
    edge $e$.   
    We write
    $\mathbb{X} =(X, \{\Gamma_v, v\in V(X)\}, \{\Gamma_e, e\in E(X)\},
    \{i_e, e\in E(X)\})$.
  \item The \define{Bass group} is the free product of the vertex
    groups and of the free groups on $E(X)$ subject to relations $\bar
    e = e^{-1}, \forall e\in E(X)$ and $ei_e(g)e^{-1} =i_{\bar e}(g)$.
  \item The \define{fundamental group $\pi_1(\bbX, \tau)$ of a graph
      of groups $\bbX$ with maximal subtree $\tau \subset X$}, is the
    quotient of the Bass group by the normal subgroup generated by the
    set of edges in $\tau$.
\end{itemize}
\end{defn}

Classically, the groups $\G_v, v\in V(X)$ are called vertex groups,
the groups $\Gamma_e, e\in E(X)$ are edge groups, and the maps $i_e,
e\in E(X)$ are the attaching maps.  For convenience, for every vertex
group $\G_v$, we introduce the (unmarked) adjacency peripheral
structure $(\G_v, \calA_u)$, consisting of the conjugacy classes of
the groups $i_e(\G_e), e\in \lk(v)$.  For a choice of a generating
tuple $S_e=S_{\bar e}$ of $\G_e = \G_{\bar e}$, we can define the
marked peripheral structure $(\G_v, \calA_m)$ induced by $i_e$.

We say that a group is \define{explicitly given} as a fundamental
group of a graph of groups if a graph of groups is given, and a
presentation of all vertex groups, a generating set of all edge
groups, the image of the attaching maps on these generating sets, for
which the group is isomorphic to the fundamental group of this graph
of groups (once chosen a maximal subtree.)

There is a universal covering $T_{\bbX}$  
   of $\bbX$, a tree, called \define{the
  Bass-Serre tree}, on which the fundamental group $\pi_1(\bbX,
\tau)$ 
acts, and $X \simeq \pi_1(\bbX,\tau)\backslash T_\bbX$. Conversely,
given only the abstract group $\pi_1(\bbX,\tau)$ and the action on
$T_\bbX$ we can recover the graph of groups $\bbX$. We refer to
\cite{Serre} for further details, but we assume the reader is fluent
in Bass-Serre theory.

\subsubsection{Vocabulary on trees}\label{sec;vocab_trees}
 
We now define or recall a few adjectives and operations related to
trees. A $G$-tree is \define{reduced} 
 if any two adjacent vertices in different $G$-orbits have
different stabilizers. A vertex of a $G$-tree is \emph{inessential} if
it has valence 1.

Given a group with a peripheral structure $(G,\calP)$, a
\define{$(G,\calP)$-tree} is a tree with an action of $G$ such that
each group in $\calP$ is elliptic. A \emph{splitting of $(G,\calP)$}
is a splitting of $G$ such that the Bass-Serre tree is a
$(G,\calP)$-tree. A $(G,\calP)$-tree is \define{minimal} if there is
no proper invariant subtree. $(G, \calP)$ is said to be \define{freely
  indecomposable }if all $(G,\calP)$-tree with trivial edge
stabilizers have a $G$-fixed point. We sometimes redundantly say
\define{relatively freely indecomposable} to emphasize the role of the
peripheral structure and avoid ambiguity.

Let $T$ be a $(G,\calP)$-tree. A \define{collapse} of $T$ is a
$(G,\calP)$-tree $S$ with a $G$-equivariant map $T\tto S$ such that
each edge of $T$ is sent on an edge, or on a vertex of $S$, and no
pair of edges of $T$ in different orbits are sent in the same edge
(they may be sent on the same vertex though.)  A \define{refinement}
of $T$ is a $(G,\calP)$-tree $S$ such that $T$ is a collapse of $S$.

The following is well known and useful.
\begin{lemma} \label{lem;comp_ell} If $T$ is a $G$-tree and $\hat{T}$
  is a refinement of $G$, then any edge stabilizer in $T$ stabilizes
  an edge in $\hat{T}$.
\end{lemma}
\begin{proof} Consider $e$ an edge in $T$, $v,w$ its vertices, and
  $G_e$ its stabilizer. Let $T'_v, T'_w$ the preimages of $v,w$ by the
  collapse map $\hat{T} \to T$. These are trees, obviously disjoint,
  hence, let $\sigma$ be the segment joining them in $\hat{T}$. The
  collapse map sends $\sigma$ on $e$, hence, $G_e$ stabilizes $\sigma$
  setwise, and also its ends points, since  their $G_e$-orbit
  are respectively in $T'_v\cap \sigma$ and $T'_w\cap \sigma$ (which
  are singletons.) Since $\sigma$ is a segment, all its edges are
  fixed by $G_e$. 
\end{proof}

Two $(G,\calP)$-trees are \emph{compatible} if they are collapses of a
single $(G,\calP)$-tree (or equivalently if they have a common
refinement.)   
 We say that a
$(G,\calP)$-tree $T$ \emph{dominates} a $(G,\calP)$-tree $T'$ it there
is an equivariant map $T\to T'$.  An equivalence class for the
relation of mutual domination is a \emph{deformation space}.

Let $\calE$ be a family of subgroups of $G$; we say that a
$(G,\calP)$-tree is \define{over} $\calE$ if all edge stabilizers are
in $\calE$. A $(G, \calP)$-tree over $\calE$ is \define{universally
  compatible} over $\calE$ relative to $\calP$ if it is compatible
with any other $(G, \calP)$-tree over $\calE$.  In other words, such
a tree $T$ is universally compatible if for any other $(G,
\calP)$-tree $S$ over $\calE$, there exists a $(G, \calP)$-tree over
$\calE$, denoted by $\hat T$, and (equivariant) collapse maps $\hat T
\to T$ and $\hat{T} \to S$.
  
\subsubsection{Splittings of relatively hyperbolic groups}\label{sec;rhg-splittings}
We are now interested in splittings of relatively hyperbolic
groups. Let $(G,\calP)$ be relatively hyperbolic, the class $\calE$ of
\define{elementary subgroups} of $(G,\calP)$ is the collection of
virtually cyclic and parabolic subgroups. An \define{elementary}
$(G,\calP)$-tree is a $(G,\calP)$ tree over  $\calE$. A relatively
hyperbolic group is \define{rigid} if there is no non-trivial, reduced
elementary $(G,\calP)$-splitting.

Following \cite{Bow_periph}, we say that a $(G,\calP)$-tree (or a
$(G,\calP)$-splitting) is \define{peripheral} if it is a bipartite
with white and black vertices, such that the collection of black
vertex stabilizers is $\calP$. 

Every peripheral tree $S$ is a $(G,\calP)$-tree, but the converse is
not true in general.  A necessary condition is that all edge
stabilizers are subgroups of groups in $\calP$.  Another necessary
condition, possibly less important, is that each group of $\calP$
coincides with the stabilizer of a vertex.  In particular, it is
possible to modify a bipartite $(G,\calP)$-tree in which black vertex
stabilizers contain $\calP$ by equivariantly adding valence 1
vertices fixed by groups in $\calP$, attached by an edge to a vertex
in the their fixed-point set. This motivates the following definition.

A  $(G,\calP)$-tree $T$  is \define{essentially
  peripheral}   if it is obtained by
collapsing edges adjacent to inessential (i.e. valence 1) black vertices of
a peripheral tree $S$.
Then, $S$ is said to be a \define{peripheral refinement}
of $T$; they both lie in the same deformation space. Peripheral
splittings of relatively hyperbolic groups are elementary, and we have
the following characterization.

\begin{lemma} \label{lem;carac_ess_periph} A splitting $\bbX$ of
  $(G,\calP)$ is essentially peripheral if and only if it is
  bipartite, all black vertex groups being in the collection $\calP$,
  and no white vertex groups are parabolic.
\end{lemma}

\begin{defn}\label{defn;induced-structure}
  Let $T$ be an essentially peripheral $(G,\calP)$-tree, with the
  convention that black vertex stabilizers are the groups of
  $\calP$. Let $w$ be a white vertex and $G_w$ its stabilizer in $G$.
  Then the \emph{induced peripheral structure on $G_w$}
  is \[\calP_T(G_w) = \{ P\cap G_w \mid P \in \calP\}.\]
\end{defn}

Note that this provides a well defined induced peripheral structure on
a vertex group of a splitting. It may help to be more explicit, as in
the following.

\begin{lemma} \label{lem;describe_induced} Let $(G,\calP)$ be
  relatively hyperbolic.  Let $w$ a white vertex of a peripheral
  $(G,\calP)$-tree, and ${\rm lk}(w)$ its link (the set of edges
  adjacent to $w$.)  The induced peripheral structure of the vertex
  stabilizer $G_w$ is $\calP_T(G_w) = \{ P\in \calP, P<G_w\} \cup \{
  G_e, e\in {\rm lk}(w) \}$. 
\end{lemma}

\begin{proof} All these groups are clearly in $\calP_T(G_w)$, since
  the vertex at the other end of each $e\in {\rm lk}(w)$ is black. If
  $P\in \calP$, then it fixes a black vertex $b$, and $P\cap G_w$
  stabilizes the first edge $e$ of the segment $[w,b]$. Let $b'$ be
  the other vertex of $e$. Then $G_{b'} \cap G_b$ is infinite, which,
  in a relatively hyperbolic group implies $G_b=G_{b'}$ (both groups
  are peripheral.)
\end{proof}

Bowditch proved:

\begin{thm}[First clause of {\cite[Theorem
    1.3]{Bow_periph}}]\label{thm;induced-struct}
  Suppose that $(G,\calP)$ is a relatively hyperbolic group. Suppose
  $\Gamma_v$ is some non-peripheral vertex group of some peripheral
  splitting $\bbX$ of $G$. Then $\Gamma_v$ is hyperbolic relative the
  induced peripheral structure $\calP_\bbX(\Gamma_v)$.
\end{thm}

The reader can check that the same conclusion holds if $\bbX$ is
essentially peripheral instead of peripheral, by looking at the
peripheral refinement. An elementary splitting of $(G,\calP)$ is not
always essentially peripheral: an edge group can be
non-parabolic virtually cyclic. We therefore find it convenient to
define the following augmented peripheral structure.

\begin{defn}\label{defn;augmented-periph}
  Let $T$ be a bipartite $(G,\calP)$-tree (with black and white
  vertices), and let $\calB_T$ denote the set of stabilizers in $G$ of
  the black vertices of $T$. We say that the \define{augmented
    peripheral structure} of $(G,\calP)$ induced by $T$ is $\calP_T =
  \calP \cup \calB_T$. Let $\bbX$ be a bipartite splitting (with black
  and white vertices) of $(G,\calP)$, the \define{augmented peripheral
    structure} $\calP_\bbX$ of $(G,\calP)$ induced by $\bbX$ is that
  induced by the Bass-Serre tree of $T_\bbX$.
\end{defn}

If we are only increasing our peripheral structure to contain
virtually cyclic groups, then the augmented peripheral structure is,
yet again, relatively hyperbolic:

\begin{prop}\label{prop;RH_passing_to_a_vertex} Let  $(G,\calP)$ be a
  relatively hyperbolic group.  Let $T$ be a bipartite elementary
  $(G,\calP)$-tree with maximal elementary black
  vertices. 
  Let $G_v$ be a
  vertex stabilizer. Then $G_v$ is hyperbolic relative to
  $\calP_T(G_v)$. 
\end{prop}
\begin{proof}
  First, note that $G$ itself is hyperbolic relative to $\calP_T$, the
  collection of maximal elementary subgroups stabilizing some edge,
  since it is obtained from $\calP$ by adding maximal virtually cyclic
  edge groups by\cite[Corollary 1.7]{Osin_ESBG}.  The result now
  follows from Theorem \ref{thm;induced-struct}.
\end{proof}

\section{Canonical splittings and deformation spaces of relatively hyperbolic
  groups}\label{sec;canonical-splittings}

In the first two parts of this section  we recall definitions
and key properties of Dunwoody-Stallings and JSJ decompositions, and
we show how to compute such canonical splittings of relatively
hyperbolic groups, provided we have a procedure to certify
rigidity. In Section \ref{sec:detect-rigid} we will show how how to
use the second author's results to decide rigidity; thus giving
algorithms to compute these canonical splittings.

\subsection{ Dunwoody-Stallings decompositions}

\begin{defn}[relative Dunwoody-Stallings decomposition]
  Let $(G,\calP)$ be a group with a peripheral structure. $G$ is
  one-ended relative to $\calP$ if it has a fixed point on any
  $(G,\calP)$-tree of which the edge stabilizers are finite.
  \define{A Dunwoody-Stallings decomposition} of $(G,\calP)$ is a
  reduced splitting $\bbD$ of $(G,\calP)$ such that every edge group
  is finite and every vertex group is one-ended relative to $\calP$.
\end{defn}

A Dunwoody-Stallings decomposition can be obtained by successively 
refining splittings of $(G,\calP)$ and terminating when all the vertex
groups are relatively one-ended. It exists for finitely presented groups by
\cite{Dunwoody-1985}. Although such splittings, strictly speaking, are
not canonical, the corresponding \emph{deformation space} is.
Specifically if $\bbD_1$ and $\bbD_2$ are two Dunwoody-Stallings
decompositions then there are $G$-equivariant maps:
\begin{eqnarray*}
  T_{\bbD_1}&\onto&T_{\bbD_2}\\
  T_{\bbD_2}&\onto&T_{\bbD_1}\\
\end{eqnarray*}
between the dual Bass-Serre trees whose compositions are homotopic to
the identity. The following easy fact tells us how to compute such a
decomposition.

\begin{prop}\label{prop:compute-DS}
  If there is a procedure to decide if each vertex group that occurs
  in a sequence of splittings of $(G,\calP)$ with finite edge groups
  \[ \bbX_1,\ldots,\bbX_d
  \] obtained from a sequence of refinements of the vertex groups is
  one ended relative to $\calP$, then there is a procedure to produce
  a Dunwoody-Stallings decomposition relative to $\calP.$
\end{prop}

If $(G,\calP)$ is torsion-free then the 
Dunwoody-Stallings decompositions are in fact   \define{Grushko
  decompositions relative to $\calP$:}
\begin{equation}
  \label{eqn:grushko}
  G = H_1 * \ldots * H_p*\bbF_q,
\end{equation}
where each subgroup in $\calP$ is conjugate into some $H_i$ and each
$H_i$ is freely indecomposable relative to $\calP$. The $H_i$ are
canonical up to permutation of the indices and conjugacy and the pair
$(p,q)$ known as Scott complexity is well defined for $(G,\calP)$.

\subsection{JSJ decompositions}

We now turn our attention to one-ended relatively hyperbolic
groups. Group theoretical JSJ theory (as initiated by Rips and Sela
\cite{R-S-JSJ}) gives a complete description of their elementary
splittings. We start by presenting some key points of this theory,
following Guirardel and Levitt's treatment. For the experts, we may
anticipate by saying that  
the canonical JSJ decomposition will
be the (collapsed) tree of cylinders for co-elementarity of the
compatibility JSJ deformation space, as defined in \cite{2016arXiv160205139G}. We
will provide a new characterization, Theorem \ref{thm:JSJ-charac-II},
of this splitting that is particularly well-suited for computations.

We follow \cite{2016arXiv160205139G} (see also \cite{GL_trees}); recall the terminology introduced
in Sections \ref{sec;vocab_trees} and \ref{sec;rhg-splittings}. Let
$(G, \calP)$ be a relatively freely indecomposable, relatively
hyperbolic group, and let $\calE$ be the class of its elementary
subgroups. The \emph{tree of cylinders for co-elementarity} $T_c$ of
an elementary $(G,\calP)$-tree $T$ is constructed as follows
(\cite[Definition 4.3]{GL_trees}, \cite[Definition 7.2]{2016arXiv160205139G}.)

Define an equivalence relation on the set of edges of $T$: two edges
are equivalent if their stabilizers are subgroups of the same maximal
elementary subgroup. The equivalence classes are the \emph{cylinders}
$Cyl(T)$ of $T$. The set of vertices of $T_c$ is the union of
$V_1=Cyl(T)$ and $V_2\subset V(T)$ where $V_1$ is the set of
cylinders, and $V_2$ is the set of vertices of $T$ in at least two
cylinders. The set of edges of $T_c$ is given by the membership
relation if a vertex in $w \in V_2$ is contained in a cylinder in $c
\in V_1$, then we connect the vertices $w$ and $c$ of $T_c$ by an
edge; thus $T_c$ is bipartite. The group $G$ acts naturally on
$T_c$. The stabilizer of a vertex in $V_1$ is the global stabilizer of
the cylinder. By \cite[Proposition 6.1]{GL_trees} the \emph{collapsed
  tree of cylinder} for an elementary $(G,\calP)$-tree coincides with
the tree of cylinders $T_c$ (we mention this tree because  it appears in some
results we quote.)

\begin{lemma}\label{lem;group_of_cyl} If $(G,\calP)$ is a
   relatively
  hyperbolic group, and $T$ an elementary $(G,\calP)$-tree with
  infinite edge stabilizers, then, in the tree of cylinders,  
  the stabilizer of a vertex in $V_1$  is the unique
  maximal elementary group of $G$  containing the stabilizers of the edges
  contained in the  cylinder.
\end{lemma} 
\begin{proof}  
Let $C\in  Cyl(T)$ be a cylinder. Assume that this equivalence class
correspond to the inclusion of edge stabilizers in $E$, some
maximal elementary subgroup of $G$. Let $g\in G$ stabilizing
$C$. For an edge $e\in C$, $gc\in C$, hence both their stabilizers $G_e$ and
$g G_e  g^{-1}$  are in $E$. Hence  $G_e \subset E\cap g^{-1}Eg$,  and
$G_e$      being infinite,  
by Proposition \ref{prop;parab_own_norm}, if $E$ is parabolic,   $g\in
E$, and  it is also true if $E$ is virtually cyclic, non-parabolic. Conversely, if $g\in E$ it
clearly preserves $C$ as a set. 

\end{proof}
  
Recall that, by \cite[Theorem 9 and 10]{2016arXiv160205139G} there exists a
$(G,\calP)$-tree that is invariant by automorphisms of $(G,\calP)$, universally compatible over the class $\calE$ of elementary subgroups, relative
to $\calP$, and maximal for domination.  We consider $T_\bbJ$ its tree
of cylinders for co-elementarity. We call this tree (and the
associated splitting) the \emph{canonical elementary JSJ} tree (and
splitting, or decomposition) of $(G, \calP)$ (but we often drop the
word elementary.) 

Before continuing we must define the necessary notation and
terminology for $\QH$ subgroups. We refer the reader to
\cite[\S5]{2016arXiv160205139G} for a comprehensive treatment of $\QH$
 subgroups and a good treatment of  hyperbolic 2-orbifolds. In the
spirit of \cite{Dunwoody-Sageev} we have the following:

\begin{defn}[Finite-by-orbifold group]\label{defn:finite-by-orbifold}
  A \define{finite-by-orbifold group} is a pair $(Q,Q \actson P)$
  where:
  \begin{enumerate}[(i)]
  \item $Q$ is a group,
  \item $P$ is a complete, infinite, simply connected, planar, Gromov
    hyperbolic 2-complex, and
  \item the action $Q \actson P$ is properly discontinuous and
    cocompact.
  \end{enumerate}
  In particular $Q$ sits inside a short exact sequence
  \begin{equation}
    \label{eq:finite-by-orbifold}
    1 \to F \into Q \onto \orbfun{\calO}\to 1    
  \end{equation}
  where $F$ is the finite kernel of the action of $Q$ on $P$ and
  $\orbfun{\calO}$ is the orbifold fundamental group of the orbifold
  $\calO$ dual to the action of $Q/F$ on $P$. $\calO$ is called the
  \define{base of the finite-by-orbifold group}. $Q$ naturally
  inherits a peripheral structure $\calP_{Q,P}$ called the
  \define{peripheral boundary structure} which consists of the
  stabilizers of the connected components of $\partial P$.
\end{defn}





An orbifold $\calO$ has an orbifold fundamental group $\orbfun\calO$,
and has an orbifold boundary $\orbbdy\calO \subset \topbdy\calO$
consisting of regular points whose connected components are either
simple closed curves or arcs with endpoints in mirrors. If a connected
component $C \subset \orbbdy\calO$ is a circle, then we say it is
\define{$\bbZ$-type} and it's inclusion in $\calO$ corresponds to a
conjugacy class of cyclic groups in $\orbfun\calO$. If $C$ is an arc
between mirrors, then we say that it is \define{infinite dihedral or
  $\dihedral$-type} and its inclusion in $\calO$ corresponds to a
conjugacy class of an infinite dihedral subgroup. We call these
subgroups the \define{peripheral boundary subgroups} and we denote by
$\calP_\calO$ the corresponding peripheral structure on
$\orbfun{\calO}$. If $\calO$ is the base of a finite-by-orbifold
group $(Q,\calP_Q)$ then it follows from Definition
\ref{defn:finite-by-orbifold} that $\calP_Q$ consists of the preimages
of the groups in $\calP_\calO$ via the surjection given in
(\ref{eq:finite-by-orbifold}).

\begin{defn}[$\QH$  subgroup. See {\cite[Definition
    5.13]{2016arXiv160205139G}}]
  \label{defn:fbo} Let $(G,\calP)$ be a one-ended relatively
  hyperbolic group and let $\bbX$ be an elementary splitting. A
  \define{$\QH$ subgroup} is a vertex group $G_q$ of $\bbX$ such that
  equipped with $(G_q,\calP_\bbX(G_q))$, the peripheral structure
  induced by $\bbX$ (Definition
  \ref{defn;induced-structure}), satisfies the following:
  \begin{enumerate}
  \item $G_q$ can be made into a finite-by-orbifold group
    $(G_q,G_q\actson P)$, and
  \item the induced peripheral structure coincides with the peripheral
    boundary structure (Definition \ref{defn:finite-by-orbifold}),
    i.e. \[ \calP_\bbX(G_q) = \calP_{G_q,P}.
    \]
  \end{enumerate}
\end{defn}

The equality above implies that the  
 infinite intersections of $\QH$ subgroups
with parabolic subgroups of $(G,\calP)$ are  peripheral
boundary subgroups.

\begin{prop}[c.f. {\cite[Theorem 9.18 and Corollary 9.20]{2016arXiv160205139G}}]\label{prop:jsj-charac-1} 
  Let $(G,\calP)$ be a one-ended relatively hyperbolic group. The
  canonical JSJ decomposition $\bbJ$ of $(G,\calP)$ is an elementary
  splitting satisfying the following properties:
  \begin{enumerate}
  \item\label{it:JSJ-part1} The underlying graph $J$  
    is bipartite with
    black and white vertices.
    \begin{enumerate}[(i)]
    \item\label{it:max-elem-vert-gp} The black vertex groups $G_b$ are
      maximal elementary vertex groups.
    \item\label{it:white-gps} The white vertex groups $G_w$ are
      non-elementary and fall into one of the following two
      categories:
      \begin{enumerate}[(a)]
      \item \label{it:white-qh}$(G_w,\calP_\bbJ)$ is isomorphic to a $\QH$  group
        $(Q,\calP_Q)$. 
    \item\label{it:rigids} $(G_w,\calP_\bbJ)$ is rigid (in the sense
      previously defined that there is no non-trivial, reduced
      elementary $(G_w,\calP_\bbJ)$-splitting) and not isomorphic to a
      $\QH$ subgroup. Also for any two different adjacent edges, the
      image of the edge groups in $G_w$ are not conjugated in $G_w$
      into the same maximal elementary subgroup of $G_w$.
      \end{enumerate}
      
    \end{enumerate}
  \item\label{it:univ-compat} $\bbJ$ is universally compatible over the class $\calE$ of elementary subgroups,
   relative to $\calP$.
  \item\label{it:ue-edge-gps} Every edge stabilizer in $T_\bbJ$ is
    elliptic in any $(G,\calP)$-tree over $\calE$.
    
  \item\label{it:essential} $(G_w,\calP_\bbJ)$ is essential,
    i.e. there are no valence-1 vertices in $T_\bbJ$.
  \item\label{it:auto-invariant} $\bbJ$ is canonical, in the sense that it is invariant by automorphisms of $(G,\calP)$. 
  \end{enumerate}
\end{prop}

\begin{proof}
  By definition, $T_\bbJ$ is the collapsed tree of cylinders from
  \cite[Theorem 9.18]{2016arXiv160205139G}, thus is bipartite,
  invariant by automorphisms, and by Lemma \ref{lem;group_of_cyl}, up
  to choice of colouring, the stabilizers of black vertex groups are
  maximal elementary.  This proves (\ref{it:max-elem-vert-gp}) and
  (\ref{it:auto-invariant}). Next, (\ref{it:univ-compat}) follows from
  \cite[Corollary 9.20]{2016arXiv160205139G}.  By construction of
  cylinders, the white vertex groups are non-elementary.  The rigidity
  of non-$\QH$ white vertex groups is a consequence of the maximality
  for domination (for the universally compatible trees) of $T_\bbJ$,
  thus the groups given by (\ref{it:white-qh}) are the only
  non-rigid white vertex groups.  The claim (\ref{it:rigids}) is a
  consequence that black vertex groups are groups of cylinders for
  coelementarity.


  By Lemma \ref{lem;comp_ell}, every edge group in $\bbJ$ will remain
  elliptic after a refinement and, of course, also after a collapse.
  So, by universal compatibility, (\ref{it:ue-edge-gps})
  follows. (\ref{it:essential}) is an immediate consequence of
  requiring the action of $G$ on $T_\bbJ$ to be minimal.
  
\end{proof}

\begin{defn}\label{defn:inessential}
  An edge $e$ of a $G$-tree $T$ is called \define{inessential} if its
  stabilizer equals that of one of its vertices.  
\end{defn}

\subsubsection{Structure, algorithmic recognition, and presentation
  of virtually infinite cyclic subgroups.}

First we recall some generalities about these groups.

\begin{lem}\label{lem:barely-infinite}
  Every proper quotient of $\bbZ$ and $\bbZ_2*\bbZ_2$ is finite.
\end{lem}

\begin{lem}[See {\cite[Lemma 11.4]{hempel20043}}]\label{lem:vz-struct}
  If $G$ is virtually infinite cyclic then it sits in an exact
  sequence\[ 1 \to K \to G \to Q \to 1
  \] where $Q$ is either $\bbZ$ or $\bbZ_2*\bbZ_2$, and $K$ is a
  maximal normal finite subgroup. We call $Q$ the
  \define{$\Isom(\bbR)$-quotient} and $K$ the
  \define{$\Isom(\bbR)$-kernel.}
\end{lem}

The next result tells us how to detect and present virtually
infinite cyclic groups given certain algorithmic
properties.

\begin{defn}\label{defn:tf-index}
  If a group $H$ is virtually torsion-free, then a \define{torsion-free index}
  is a number $\tfi H \in \bbZ_{\geq 0}$ such that the intersection\[
    \bigcap_{[H:K]\leq \tfi H} K
  \] is torsion free.
\end{defn}

\begin{lem}[Detecting virtually infinite cyclic from presentations]\label{lem:vz-parabolic}
  Suppose we are given a finite group presentation $P= \bk{X\mid R}$,
  a solution to the word problem for this presentation, and an upper
  bound $\tfi P \leq L$. Then there is an algorithm that decides if
  $P$ is virtually infinite
  cyclic.
\end{lem}

\begin{proof}
  We first give a characterization of virtually infinite cyclic
  groups. Let $L$ be the upper bound for $\tfi P$ and let
  \[ Q = \bigcap_{[P:K]\leq L}K.
  \]
  \emph{Claim: $Q\approx \bbZ$ if and only if $P$ is virtually
    infinite cyclic.} By definition, $Q$ is torsion-free and of finite
  index in $P$. If $P$ is virtually infinite cyclic  then, by
  Lemma \ref{lem:vz-struct}, $Q \approx \bbZ$.
  The converse is obvious. The claim is proved.


  Suppose now that we are given a finite group presentation $P$ that
  satisfies the statement of the lemma. We can find presentations of
  all its finite index subgroups up to index $M$ as well as a
  presentation for their intersection; thus we can find a presentation
  for $Q$. On the other hand we can decide, using our solution to the
  word problem in $P$, whether the generators of $Q$ commute, and if
  so, using $\bbZ$-module calculations decide if $Q \approx \bbZ$.

\end{proof}

We leave the proof of the following to the reader. 

\begin{lem}\label{lem:tfi-monotonic}
  If $P \leq H$ then $\tfi P \leq \tfi H$.
\end{lem}

\begin{cor}\label{cor:present-par-vz}
  If $P$ virtually torsion free, hereditarily algorithmically tractable,
  and with algorithmically bounded torsion, then we can decide whether
  a collection of elements $S \subset P$ generates a virtually
  infinite cyclic subgroup, and find a presentation for the subgroup.
\end{cor}
\begin{proof}
  By hereditary algorithmic tractability we can find a presentation
  for $\bk{S}$. Since $P$ is residually finite and torsion is
  algorithmically bounded, we can find a finite index torsion free
  subgroup and therefore give an upper bound for $\tfi P$. By Lemma
  \ref{lem:tfi-monotonic} this gives an upper bound for $\tfi{\bk S}$,
  and the result follows from Lemma \ref{lem:vz-parabolic}.
\end{proof}

We now show how to detect and present non-parabolic virtually infinite
cyclic subgroups of relatively hyperbolic groups.

\begin{lem}\label{lem:enum-vz}
  There is an algorithm that enumerates presentations of all virtually
  infinite cyclic groups, finds their $\Isom(\bbR)$-kernels, and gives
  an algorithm to solve the word problem for these groups.
\end{lem}

This lemma has an elementary approach, based on computations with
group presentations and basic algebra of finite groups, as well as a
cohomological approach, based on the classification of group
extensions (see \cite[\S IV.6]{brown_cohomology_1994}.) Although these
approaches are certainly more elegant, in the interest of providing
the shortest possible complete argument we have the following.
\begin{proof}
  
  Enumerate all group presentations. In parallel, for each group
  presentation, run Papasoglu's algorithm \cite{Pap} which will
  terminate if a presentation defines a hyperbolic group. For each
  presentation $P$ of a hyperbolic group $H$, knowing the
  hyperbolicity constant $\delta$, we can find a maximal finite normal
  subgroup $K$ as well as solve the word problem for $H$. Enumerate
  all Tietze transformations on the presentation of $H/K$. If via
  Tietze transformations we either obtain $\bk{a \mid }$ or
  $\bk{a,b \mid a^2,b^2}$, then stop that branch and output the
  presentation $P$ and the normal subgroup $K$. It is obvious that
  this algorithm will output exactly the presentations of the
  virtually infinite cyclic groups.
\end{proof}

\begin{lem}[Injectivity for virtually infinite
  cyclic groups]\label{lem:vz-inj}
  Let $G$ be virtually infinite cyclic and let it sit in the short
  exact sequence \[ 1 \to K \to G \to Q \to 1 \] given in Lemma
  \ref{lem:vz-struct}. Let $\varphi: G \to H$ be a homomorphism such
  that
  \begin{enumerate}
  \item\label{it:inf-img} $\varphi(G)$ is infinite, and 
  \item the intersection with $K$, the $\Isom(\bbR)$-kernel, and
    $\ker(\varphi)$ is trivial:\[ K \cap \ker(\varphi) = \{1\}.
      \]
  \end{enumerate}
  Then $\varphi$ is injective.
\end{lem}
\begin{proof}
  Suppose towards a contradiction that there was some element
  $a \in G \setminus K$ such that $\varphi(a)=1$. Then its image
  $\bar a \in Q$ is non-trivial, so by Lemma \ref{lem:barely-infinite}
  $Q/\ncl{\bar a}$ is finite. On the one hand we have the commutative
  diagram with exact rows:\[
    \begin{tikzpicture}[xscale=3,yscale=1.5]
      \node (tl1) at (0,0) {$1$};
      \node (K) at (1,0) {$K$};
      \node (G) at (2,0) {$G$};
      \node (Q) at (3,0) {$Q$};
      \node (tr1) at (4,0) {$1$};
      \node (bl1) at (0,-1) {$1$};
      \node (K/a) at (1,-1) {$K/(K\cap\ncl{a})$};
      \node (G/a) at (2,-1) {$G/\ncl{a}$};
      \node (Q/a) at (3,-1) {$Q/\ncl{\bar{a}}$};
      \node (br1) at (4,-1) {$1$};
      \draw[->] (tl1) -- (K);
      \draw[->] (K)-- (G);
      \draw[->] (G)-- (Q);
      \draw[->] (Q) -- (tr1);
      \draw[->] (bl1) -- (K/a);
      \draw[->] (K/a)-- (G/a);
      \draw[->] (G/a)-- (Q/a);
      \draw[->] (Q/a) -- (br1);
      \draw[->>] (K)--(K/a);
      \draw[->>] (G)--(G/a);
      \draw[->>] (Q)--(Q/a);
    \end{tikzpicture}
    \] which implies that $G/\ncl{a}$ is finite. On the other hand
    $\varphi$ factors through $G/\ncl{a}$, which contradicts that
    $\varphi$ has infinite image.
\end{proof}

We now have an algorithm to produce presentations of virtually
infinite cyclic groups.

\begin{lem}[Presenting virtually infinite cyclic groups]\label{lem:present-vz}
  Suppose we are given a group $H$, a finite generating set $\bk{S}=H$
  as well as a solution to the word problem for words in
  $\bk{S}$. Then there is procedure that terminates if and only if $H$
  is virtually infinite cyclic and outputs a presentation
  $\bk{S\mid R_S}$ for $H$, with the generating set $S$.
\end{lem}
\begin{proof}
  Use the algorithm of Lemma \ref{lem:enum-vz} to enumerate
  presentations $P_i=\bk{X_i\mid R_i}$ of virtually cyclic groups. Let
  us also denote by $P_i$ the group defined by the presentation. In
  parallel for each $P_i$, enumerate the mappings
  $X_i \to \bk{S}$. Because we can solve the word problem in $\bk{S}$
  and $P_i$, we can determine which of these mappings extend to a
  homomorphism. In parallel, for each homomorphism that is found,
  enumerate its image and check if $S$ is contained in the image. If
  such a homomorphism is found, then we have a surjection onto $H$.
  In parallel, for each found epimorphism $\varphi:P_i \onto H$,
  verify that $\varphi$ is injective on $K_i$, the
  $\Isom(\bbR)$-kernel of $P_i$. If this is the case then
  $P_i \approx H$ by Lemma \ref{lem:vz-inj}, which implies that $H$ is
  virtually cyclic and using $\varphi$, we can get the desired
  presentation. Conversely, if $H$ is virtually infinite cyclic, then
  one branch of this process will terminate with a desired presentation for
  $H$.
\end{proof}


%
%
%
%

\begin{prop}\label{prop;present-np-vz}
  Let $(G,\calP)$ be relatively hyperbolic, and suppose we are given
  a uniform solution to the conjugacy problem for the groups in
  $\calP$. Given a generating set $S$ of a subgroup
  $H=\bk S \leq (G,\calP)$, there is an algorithm which
  \begin{enumerate}
  \item\label{it:decide-virt-Z} decides if $\bk S$ is non-parabolic
    and virtually infinite cyclic, and, if so, 
  \item\label{it:present-virt-Z} computes a presentation of $\bk S = \bk{S \mid R_S}$.
  \end{enumerate}
\end{prop}

\begin{proof}
  We first show that (\ref{it:decide-virt-Z}) is decidable. Because we
  are assuming that the we have a solution to the conjugacy problem in
  the parabolic subgroups by \cite[Theorem 5.6]{Osi} we can decide if
  the elements of $S$ are all conjugate into some parabolic subgroup.

  If this isn't the case, then we must determine if $H$ is virtually
  cyclic. In this case $H$ is either $\bbZ$-type, i.e. it maps
  surjectively onto $\bbZ$ and the commutator subgroup $[H,H]$ is
  finite, or $H$ is $(\bbZ_2*\bbZ_2)$-type, i.e. it maps surjectively
  onto $\bbZ_2*\bbZ_2$ and the commutator subgroup $[H,H]$ infinite. Note
  however that in the $(\bbZ_2*\bbZ_2)$-case, the commutator subgroup
  $[H,H]$ itself is $\bbZ$-type. It follows that if $H$ is virtually
  cyclic, then its second derived subgroup $H^{(2)}=[[H,H],[H,H]]$ is
  finite. To see that this condition is sufficient, note that
  any non-parabolic subgroup of $(G,\calP)$ that is not
  virtually cyclic contains a free group of rank 2 and therefore has
  infinite second derived subgroup.

  In particular either $H^{(2)}$ has cardinality less than
  $\left|P_2\right|$, where $P_2 \subset G$ is the subset computed in
  Lemma \ref{lem:ball-of-finites}, or $H^{(2)}$ is infinite. We now
  start two parallel processes. One of them is to enumerate
  $H^{(2)}$. If at any point we find more than $|P_2|$ distinct
  elements of $H^{(2)}$ we stop and correctly declare that $H$ is not
  virtually infinite cyclic. The other process is to run the procedure
  given by Lemma \ref{lem:present-vz}, which will terminate if and
  only if $H$ is virtually infinite cyclic and will produce a
  presentation for $H$ with the generators $S$, as desired.

\end{proof}

\begin{cor}\label{cor;elem-cor}
  If $(G,\calP)$ is a virtually torsion free relatively hyperbolic group,
  where $\calP$ lies in a hereditarily algorithmically tractable class
  of virtually torsion free groups with algorithmically bounded
  torsion, then there is an algorithm which takes a finite generating
  set $S$ of a subgroup $H \leq G$ and outputs the following:
  \begin{enumerate}
  \item whether $H$ is parabolic,
  \item whether $H$ is virtually cyclic, and 
  \item a presentation $H=\bk{S|R(S,H)}$.
  \end{enumerate}
	Finally, there is an algorithm deciding whether $H$ is maximal elementary.
\end{cor}

\begin{proof}
	The first item is an application of the solution to the special parabolicity problem \cite[Theorem 5.6]{Osi}. The second and third items are consequences of Corollary \ref{cor:present-par-vz} if $H$ is found to be parabolic, and of Proposition \ref{prop;present-np-vz} otherwise.  
	For the last claim, if $H$ is parabolic, one can find a
        peripheral subgroup in which it is contained, and one uses the
        solution to the generation problem given by assumption to
        check whether these two are equal. If it is virtually cyclic
        non-parabolic, we refer the reader to the proof of \cite[Lemma
        2.8(3)]{DG_gafa}, since the proof is identical, up to
        replacing the result of Lys\"enok (to verify whether a loxodromic element is a proper power) by \cite[Theorem 1.16(3)]{Osi}.  
\end{proof}

\subsubsection{Structure and algorithmic recognition of $\QH$ subgroups}
The following is obvious if the $\QH$ subgroups are fundamental groups
of compact surfaces with boundary, by gluing the said surfaces along
their boundaries we obtain larger surfaces. The general case is the
same but the proof is more subtle, as it involves more than
cut-and-paste topological arguments. In particular, given two
finite-by-orbifold groups whose fibers have non-trivial automorphism
groups, even once a continuous identification map has been defined
between two boundary components there will be many ways in which the
finite-by-orbifold groups can be amalgamated.

\begin{lem}[$\QH$  subgroups amalgamate nicely]\label{lem:nice-amalgam}
  Let $\bbX$ be an elementary splitting of a relatively hyperbolic
  group $(G,\calP)$ with a bipartite underlying graph where black
  vertex groups are maximal elementary and white vertex groups are
  non-elementary. Suppose there is a valence 2 non-parabolic black
  vertex $b$ both of whose edges are inessential and adjacent to
  vertices  carrying $\QH$  subgroups. Then the collapse
 of  the two edges adjacent to $b$ produces a vertex group that is a
  larger $\QH$  subgroup.
  
  
  
  
\end{lem}
\begin{proof}
	We only need to analyze the fundamental group of the subgraph whose two edges are the adjacent edges of $b$. Let $a, c$ be their end vertex (possibly $a=c$). 
	First, after collapse of one of the edges, which is inessential at $b$,  the situation reduces to the case of an amalgamation of two $\QH$  subgroups, that we may write $A*_B C$ if $a\neq c$, or, if $a=c$,   of the HNN extension $A*_C$ of  a $\QH$ subgroup over two  peripheral subgroups. Note
	that in the HNN extension case, these peripheral subgroups must be
	non-conjugate in $A$, since otherwise we will obtain a non parabolic $\bbZ\oplus\bbZ$
	subgroup contradicting relative hyperbolicity. We produce the argument in the case of the amalgamation, the case of the HNN extension being similar.

	Recall that being $\QH$ subgroups, $A$ and $C$ split as short
        exact sequences $ 1\to F\to A\to A_q\to 1$ and
        $ 1\to F'\to B\to B_q\to 1$, in which $F, F'$ are finite, and
        $A_q, B_q$ are fundamental groups of $2$-orbifolds with
        boundary components corresponding to the peripheral
        structure. Since the edge group $B$ is in the peripheral
        structure of $A$, which corresponds to  boundary components, it
        splits as the short exact sequence $1\to F \to B \to Q \to 1$,
        in which $F$ is the finite group appearing in the exact
        sequence of $A$, and $Q$ is either $\bbZ$ or $\bbZ_2*\bbZ_2$.
        The same consideration for the attaching map of $B$ in $C$
        thus reveals that $F'=F$, that this group is normal in
        $A*_B C$, and is in the kernel of the action on the Bass-Serre
        tree of the amalgam.  It follows that the quotient of $A*_B C$
        by $F$ is acting on this Bass-Serre tree, which reveals
        its structure as an amalgam $A_0 *_Q C_0$. This new group is
        obviously a $2$-orbifold group, whose boundary subgroups are
        the images of the peripheral structures of $A$ and $C$,
        without the conjugates of $B$ (in the HNN extension case, this
        observation uses that the two images in $A$ of the edge group
        $B$ are non-conjugate in $A$, and thus correspond to distinct
        boundary components). Thus $A*_B C$ is a $\QH$ group as
        expected.

\end{proof}

\begin{defn}\label{defn:RQH}
  A $\QH$ group $(G_q,\calP_\bbX(G_q))$ is called \define{rigid
    quadratically hanging or $\RQH$} if it does not admit any
  essential elementary splittings (relative to $\calP_\bbX(G_q)$.)
\end{defn}

These groups are classified in \cite[\S 5.1.3]{2016arXiv160205139G},
but there they are called \define{small orbifolds}. The most famous
example is the pair of pants. We note that if we amalgamate any two
$\RQH$ groups in the manner of Lemma \ref{lem:nice-amalgam}, and that
the resulting orbifold still has boundary, then the resulting $\QH$ 
subgroup is no longer rigid.

\begin{lem}[$\QH$  groups are maximal]\label{lem:qhf-maximal}
  Assume that, in $\bbJ$, $b$ is a non-parabolic black vertex of
  valence $2$, whose neighbours 
  both carry $\QH$  groups. Then at least one edge is essential.
\end{lem}
\begin{proof}
  Suppose towards a contradiction that both are inessential, and let
  us call $v$ and $w$ the white neighbours of $b$.  By Lemma
  \ref{lem:nice-amalgam} the collapse of the edges $e,f$ between $b,v$
  and $b,w$, which identifies $v,b,w$ to the vertex $\bar q$, yields a
  new splitting $\bar\bbJ$ with a $\QH$ subgroup
  $(G_{\bar q},\calP_{\bar\bbJ}(G_{\bar q}))$. Let $\calO$ be a base
  orbifold for $(G_{\bar q},\calP_{\bar\bbJ}(G_{\bar q}))$.  The
  subgroup $G_b = G_e=G_f \leq G_{\bar q}$ is the edge group of an
  elementary splitting of $(G_{\bar q},\calP_{\bar\bbJ}(G_{\bar q}))$
  and is either carried by a simple closed curve in $\calO$, or an arc
  terminating in the mirrors. In both cases call this path
  $\gamma \subset \calO$.

  In both cases, there is a simple closed curve or arc connecting
  mirrors in $\calO$, representing an essential elementary splitting
  of $\orbfun\calO$, that crosses $\gamma$. It follows that there is
  an elementary splitting of
  $(G_{\bar q},\calP_{\bar\bbJ}(G_{\bar q}))$ in which $G_e$ is
  hyperbolic, and therefore a splitting of $(G,\calP)$ in which $G_e$
  is not elliptic contradicting (\ref{it:ue-edge-gps}) of Proposition
  \ref{prop:jsj-charac-1}.
\end{proof}

\begin{thm}\label{thm:JSJ-charac-II}
  Let $(G,\calP)$ be a relatively freely indecomposable, relatively
  hyperbolic group.  The canonical JSJ splitting $\bbJ$ of $(G,\calP)$
  is the unique essential splitting that satisfies
  (\ref{it:JSJ-part1}) of Proposition \ref{prop:jsj-charac-1} and for
  which, whenever $b$ is a non-parabolic black vertex of valence $2$,
  whose neighbours both carry $\QH$ groups, at least one edge is
  essential.
\end{thm}

The proof of this Theorem follows from the next two lemmas.

   \begin{lemma}\label{lem;same-DS}
     If $T$ is a $(G,\calP)$-tree satisfying (\ref{it:JSJ-part1}) of
     Proposition \ref{prop:jsj-charac-1} and is such that whenever $b$
     is a non-parabolic black vertex of valence $2$, whose
     neighbours 
     both carry $\QH$  groups, at least one edge is essential, then
     $T$ is in the same deformation space as $T_\bbJ$.
  \end{lemma}

  \begin{proof}
    By \cite{Forester_2002} it suffices to show that any vertex
    stabilizer of $T$ is elliptic in $T_\bbJ$, and that any vertex
    stabilizer of $T_\bbJ$ is elliptic in $T$. Actually, we will make
    a symmetric argument. We first use the universal compatibility of
    $T_\bbJ$: there exists a $(G,\calP)$-tree $\hat{T}$ collapsing on
    $T_\bbJ$ and on $T$. Now we are in the following setting: given
    two $(G,\calP)$-trees $T_1, T_2$ satisfying the assumption of the
    Lemma, and refined by a same tree $\hat{T}$, we must prove that
    vertex stabilizers in $T_1$ are elliptic in $T_2$.

    Let $v$ be a vertex of $T_1$, and $G_v$ its stabilizer.  Assume
    that $v$ is a white vertex of rigid type (in the sense
    of(\ref{it:JSJ-part1} (ii) (b)) of Proposition
    \ref{prop:jsj-charac-1}.)  In $\hat{T}$, the augmented peripheral
    structure of $G_v$ is elliptic by Lemma \ref{lem;comp_ell}. By
    rigidity, $G_v$ is thus elliptic in $\hat{T}$.  It is therefore
    elliptic in $T_2$ as well, since $T_2$ is a collapse of $\hat{T}$.

    Assume that $v$ is a white vertex of $\QH$  type. The action of
    $G_v$ on its minimal subtree $\hat{T_v}$ in $\hat{T}$ is dual to a
    collection of separate, non-boundary parallel simple curves in the
    base orbifold $\calO_v$. Consider the collapse $\hat{T} \to T_2$
    restricted to $\hat{T_v}$. Assume two distinct vertices with
    non-elementary stabilizer are mapped on distinct vertices
    $v_-,v_+$ (which are necessarily white.) In $[v_-, v_+]$, we
    choose two white vertices at distance $2$ apart, $v'_-,v'_+$, and
    denote by $T'_-, T'_+$ their preimages (obviously disjoint) in
    $\hat{T}$. The segment $\sigma$ between $T'_-$ and $T'_+$ is
    mapped to the segment $[v'_-,v'_+]$ in $T_2$, thus ensuring that
    it contains only vertices with elementary stabilizers. Since
    $\hat{T_v}$ corresponds to the pull back of a splitting of the
    base orbifold $\calO_v$ of $G_v$, this ensures that all edges of
    $\sigma$ are inessential, and therefore, both edges of
    $[v_-, v_+]$ are inessential. This contradicts our assumption on
    $T$.  We have proved that all white vertex stabilizers in $T_1$
    are elliptic in $T_2$.
    
    Let $v$ be a black vertex. If its stabilizer is
    in $\calP$, by definition, it is elliptic in the $(G, \calP)$-tree
    $T_2$. If it is cyclic, it is virtually a subgroup of a white
    vertex stabilizer, hence elliptic in $T_2$ by our previous
    study. We have proved our claim, that all vertex stabilizers in
    $T_1$ are elliptic in $T_2$, hence the lemma.
  \end{proof}

  \begin{lemma}\label{lem;carac_toc}
    Under the assumptions of Lemma \ref{lem;same-DS}, the tree $T$ is
    its own tree of cylinders for co-elementarity. More precisely, let
    $(G, \calP)$ be a relatively hyperbolic group, and $T$ be a
    $(G,\calP)$ tree with infinite edge stabilizers. The following are
    equivalent.
    \begin{enumerate}[(i)]
    \item The tree $T$ is isomorphic to its tree of cylinders
      for co-elementarity.
    \item $T$ is bipartite with all black vertex stabilizer
      being maximal elementary, black vertices in different orbits
      have non-conjugated stabilizers, all white vertex stabilizer
      being non-elementary.
    \item\label{carac_toc3} $(G,\calP_T)$ is relatively
      hyperbolic, and $T$ is essentially peripheral for $\calP_T$,
      with non-elementary white vertex stabilizers, and black vertices
      in different orbits have non-conjugated stabilizers.
    \end{enumerate}
  \end{lemma}

  \begin{proof}
    Since $T$ satisfies $(ii)$, by the assumption (\ref{it:rigids}) of
    Proposition \ref{prop:jsj-charac-1}, we only need to prove the
    stated equivalence.  The vertices in $V_2$ in the tree of
    cylinders are the vertices of $T$ belonging to at least two
    different cylinders, so their stabilizers are not elementary.
    
    Lemma \ref{lem;group_of_cyl} ensures that vertices in $V_1$ have
    maximal elementary stabilizers. If two vertices in $V_1$ have the
    same stabilizer, they correspond to the same cylinder, hence are
    equal.  This shows that $(i)\implies (ii)$.  For $(ii)\implies
    (iii)$, note that $\calP_T$ is obtained from $\calP$ by adding
    maximal cyclic groups that are not parabolic. The relative
    hyperbolicity is thus preserved (see \cite[Corollary
    1.7]{Osin_ESBG}, or \cite[Lemma 4.4]{Dah_CoC}.) The tree is
    essentially peripheral for $\calP_T$ by Lemma
    \ref{lem;carac_ess_periph}.
    
    Finally, let us prove $(iii)\implies
    (i)$. It suffices to show that the cylinders of $T$ are exactly
    the stars of the black vertices. Consider two different edges
    adjacent to a white vertex $v_w$, and assume their stabilizers
    co-elementary. Their black vertices $v_{b_1}, v_{b_2}$ have
    co-elementary stabilizers $E_1, E_2$.  These groups are in
    $\calP_T$ because $T$ is essentially peripheral. But since
    $(G,\calP_T)$ is relatively hyperbolic, these groups must be
    equal. The two edges are therefore in the same orbit and images of one
    another by an element $g$ normalizing $E=E_1=E_2$, which is therefore in
    $E$ by almost malnormality of peripheral subgroups (Proposition \ref{prop;parab_own_norm}). We get that $g$ fixes $v_{b_1}$, hence
    $v_{b_1}=v_{b_2}$ and both edges are equal. The claim follows.

  \end{proof}

  \begin{proof}[Proof of Theorem \ref{thm:JSJ-charac-II}] 
    $T_\bbJ$ was defined to be the tree of cylinders of a certain
    universally compatible tree over $\calE$ relative to
    $\calP$. Since the tree of cylinders of a deformation space is
    unique (\cite[Corollary 4.10]{GL_trees}, \cite[Lemma 7.3(3)]{2016arXiv160205139G}), it follows that $T_\bbJ$
    is isomorphic to any $(G,\calP)$-tree that is in the same
    deformation space and that is its own tree of cylinder. By our two
    previous lemmas, it is the case of any tree $T$ satisfying the
    stated assumption.
\end{proof}

\begin{lem}\label{lem:recognize-rqh}
  Let $(H,\calP_H)$ be a virtually torsion free relatively hyperbolic group for which
  all groups in $\calP_H$ are groups in a heriditarily algorithmically tractable class with algorithmically bounded torsion. Then there
  is an algorithm which decides if $(H,\calP_H)$ is an $\RQH$ group.
\end{lem}
\begin{proof}
  
  By Corollary \ref{cor;elem-cor} we can first decide whether
  $\calP_H$ is a collection of virtually cyclic groups. If it isn't we
  answer no. So we  can assume that  $\calP_H$ indeed consists  of virtually cyclic groups. The group 
  $H$ is therefore a hyperbolic group. We can use
  Papasoglu's algorithm \cite{Pap} to find a hyperbolicity constant
  $\delta$, and  a Dehn presentation $\bk{Z\mid T}$. We can also use
  standard algorithms for hyperbolic groups in order to find the list
  of conjugacy classes of its finite subgroups (e.g. using that by 
  \cite{bogopolskii1996finite} any finite subgroup is conjugate to a
  subgroup contained in  the ball of radius  $3\delta+1$ centered at
  the identity.) 

  We can also find $N$, the maximal finite normal subgroup of $G$, and
  pass to $(H/N,\calP_{H/N})$. If $(H,\calP_H)$ is $\RQH$ then we now
  have the fundamental group of a rigid (or small) orbifold. The set
  of rigid orbifolds is given in \cite[Proposition
  5.12]{2016arXiv160205139G}. There are ten classes and each class is
  governed by at most three positive integer parameters that can be
  inferred from the collection of maximal finite subgroups of $H/N$.
  This therefore gives a finite list $\{\calO_j\}$ of possible
  orbifolds, and $(H,\calP_H)$ is $\RQH$ if and only if
  $(H/N,\calP_{H/N})$ is isomorphic to one of these finitely many (at
  most 3, in fact) $\orbfun{\calO_j}$. This can be decided by
  \cite[Theorem 1]{DG_gafa}.
  
\end{proof}

Given a group $G$ we say that a presentation $\langle S | R \rangle$ of $G$ has the form of a graph-of-groups presentation if there exists a finite graph $X=(V,E)$, with a spanning tree $T_{span}$,  a partition of $R$ and $S$ as  
 \begin{flalign*}  & S= \left( \bigsqcup_{v\in V} S_v \right) \sqcup \left( \bigsqcup_{e\in E} S_e \right) \sqcup E \;  \hbox{ and }  \\  & R= \left( \bigsqcup_{v\in V} R_v \right) \sqcup \left( \bigsqcup_{e\in E} R_e \right) \sqcup R_{att} \sqcup \{ e=1, e\in T_{span}\}\sqcup \{ \bar e= e^{-1}, e\in E\}
\end{flalign*} 
 where:
\begin{itemize} 
	\item  for all $x\in E\cup V$,  $R_x$ involves only generators in $S_x$,  
	\item for all $e\in E$,  $S_{\bar{e}} \simeq S_e$, and this bijection induces a bijection $R_{\bar{e}} \simeq R_e$, 
	\item  for all $e$, there is a map $a_e: S_e\to S_{t(e)}$ such that  $R_{att}$ consists of all the  relations of the form    $\bar e s_e   e=a(s_e)$  for all $e\in E, s_e \in S_e$, 
	\item the above map $a_e$ induces a map $R_e \to R_{t(e)} $. 
\end{itemize}

Note that one can routinely detect whether a given presentation has the form of a graph-of-groups presentation. Note also  that a usual presentation coming from a graph of groups decomposition is (or can be easily turned into, depending on the conventions) of the form of a graph-of groups presentation. 

However, given a presentation of $G$ that has the form of a graph-of
groups presentation, considering the group defined by
$\langle S_v| R_v \rangle$ for a vertex $v$, in general it does not
embed as a subgroup of $G$, due to missing relations. Furthermore,
because of these missing relation, the maps
$a_e:X_e \to \langle S_{t(e)}| R_{t(e)} \rangle$ may not even extend
to to homomorphisms
$\langle S_e| R_e \rangle \to \langle S_{t(e)}| R_{t(e)} \rangle$. If,
however, all these maps actually extend to monomorphisms then, by the
definition of graphs of groups, the presentation fully defines a graph
of groups whose vertex groups and edge groups have the proposed
presentations $\langle S_x| R_x \rangle, x\in V\cup E$.

\begin{thm}\label{thm;compute_JSJ}

  Let $\calC$ be a hereditarily algorithmically tractable class of
  groups with algorithmically bounded
  torsion. 
  Assume that there exists an algorithm to decide whether vertex
  groups of an elementary splitting of a virtually torsion free
  relatively hyperbolic $(G,\calP)$ with peripheral subgroups in
  $\calP$ are rigid.  Then there is an algorithm that, given
  $(G,\calP)$ a one-ended virtually torsion free relatively hyperbolic
  group with $\calP$ in $\calC$, computes its canonical JSJ
  decomposition.
   
\end{thm}

\begin{proof} 
	Observe first that one can assume that $\calC$ contains all the virtually cyclic subgroups of $G$.
	The algorithm for the theorem consists of running the following
  subprocesses in parallel.
\begin{enumerate}
\item Enumerate all the presentations of $G$.
\item For each presentation, check whether it has the form of a graph-of-groups
  presentation (by enumerating the partitions of the set of
  generators, and checking whether the presentation corresponds to
  presenting vertex groups, edge groups, attaching maps, and Bass
  relations, with respect to this partition).
\item For each of these presentations of the form of a  graph of groups decomposition that is found, try to
  certify that edge groups are  elementary, and that it is bipartite with elementary black vertex groups,  using Corollary \ref{cor;elem-cor}, and that black vertex groups are maximal elementary, using the solution to the generation problem for groups in $\calC$,   given in assumption.   This provides genuine presentations of the edge groups, and black vertex groups.  
\item For each presentation thus found, use the attaching maps to add the new relations of the elementary edge groups to the presentations associated to the white vertices. 
%
%
  Modifying the presentation of $G$ in this way does not change the
  underlying group, but now all the attaching maps extend to
  monomorphisms from the edge groups to the vertex groups. Furthermore
  the presentations of the white vertex groups are now genuine.

\item For every bipartite splitting with maximal elementary black
  vertex groups, decide whether or not the white vertex groups are
  rigid, using the algorithm in the assumptions.  Using Lemma
  \ref{lem:recognize-rqh} for the peripheral structure of their
  adjacent edge groups, decide which of these rigid white vertex
  groups are $\RQH$.
  For every white vertex group that is proved non-rigid, try to
  certify if it is $\QH$, by enumerating its  presentations until
  a presentation of a $\QH$ group is found (we can enumerate the
  presentations of $\QH$ groups).

\item For every bipartite splitting thus coloured by the previous
  step, and for which every white vertex group has been proved to be
  rigid (non-$\QH$), $\RQH$, or $\QH$, we check whether adjacent edge
  groups of white rigid vertices are conjugated into the same maximal
  elementary subgroup or not.

  For that, we use the algorithm of Corollary \ref{cor;elem-cor} to
  check which edge groups are parabolic, and which are virtually
  cyclic, non-parabolic.  For all those that are not parabolic, we
  consider the maximal virtually cyclic group containing them ({\it i.e.} the
  adjacent black vertex group), and we use instances of the conjugacy
  problem to check whether they are conjugated or not.  For those that
  are parabolic, by enumeration, we may determine to which conjugacy
  class of $\calP$ they belong, and thus check whether they are all
  different.

\item For every bipartite splitting in which adjacent edge groups of
  white rigid vertices are not conjugated (in that vertex group) into
  the same maximal elementary subgroup, verify that there is no black
  vertex of valence $2$, with non-parabolic group, for which both
  adjacent edges are inessential, and with both white neighbours
  carrying $\QH$ groups (including $\RQH$ groups).
\end{enumerate}
If $(G,\calP)$ is one-ended then the algorithm will terminate with a
splitting that will be certified by Theorem \ref{thm:JSJ-charac-II} to
be the canonical JSJ.
\end{proof}

\subsection{Detecting rigidity and computing canonical splittings}\label{sec:detect-rigid}
Because we cannot even solve equations in most nilpotent groups,
we cannot certify rigidity in the same way as in previous works
\cite{DGr_ihes, DG_gafa}. We do however have the second author's
results at our disposal, which works immediately in the absence of
torsion.

\begin{prop}[Computing the Grushko
  decomposition] \label{prop;compute-Grushko} There is an algorithm
  that, provided with a finite presentation of a torsion-free
  relatively hyperbolic group $(G, \{[P_1], \dots, [P_n]\})$ with
  residually finite parabolic subgroups $P_1, \dots, P_n$, produces an
  explicit relative Grushko free-product decomposition of $G$ in which
  a factor is free (possibly trivial), and all other factors are
  freely indecomposable relative to the $P_i$'s, and relatively
  hyperbolic with respect to some explicit conjugates of some $P_i$'s.
\end{prop}

\begin{proof}
  By Lemma \ref{lem;word-problem} since $P_1,\ldots,P_n$ are
  residually finite, we have an explicit solution to the word problem
  of $G$.  The result now follows immediately by using \cite[Theorem
  B]{Tou}  with $\kappa=0$ and $\calH = \{P_1,\ldots,P_n\}$ and then
  following the proof of \cite[Theorem A]{Tou}. 
\end{proof}


\begin{theo}[{\cite[Theorem C]{Tou}}] \label{theo;TouC} Let $\calC$ be
  an algorithmically tractable class of groups.  There is an
  algorithm which takes as input an explicitly given torsion-free
  relatively hyperbolic group $(G,\calP)$, where groups in $\calP$
  belong to $\calC$, that  terminates and correctly states
  whether $(G,\calP)$ is rigid or not.
\end{theo}

Using the algorithm from Theorem \ref{theo;TouC} to decide the
rigidity of vertex groups, Theorem \ref{thm;compute_JSJ} immediately
yields the following result:

\begin{prop}\label{prop:jsj-tf}
  Let $(G,\calP)$ be a torsion-free relatively hyperbolic group where
  $\calP$ lies in a hereditarily algorithmically tractable class of
  groups, then we can compute a Grushko Decomposition and then JSJ
  decompositions for its maximal one-ended subgroups.
\end{prop}

Although the methods of \cite{Tou} do not function in the presence of
torsion, we can still use them by passing to torsion-free finite
index subgroups. The rest of this section will be mostly devoted to
showing how to pass to finite index torsion-free subgroups.

\begin{defn}\label{defn:etf}
  A group $G$ is said to be \define{effectively virtually torsion
    free} if it is possible to algorithmically find a torsion-free
  finite index subgroup $H\leq G$.
\end{defn}

\begin{lem}\label{lem:etf-exits}
  Suppose that $G$ is finitely presented, has decidable word problem,
  and suppose we can compute a complete finite list $\calF$ containing
  a  conjugacy representative of every finite order element of
  $G$. Then the following hold:
  \begin{enumerate}
  \item\label{it:rf-vtf} If $G$ is residually finite then $G$ is
    virtually torsion-free.
  \item\label{it:vtf-etf} If $G$ is virtually torsion free then it is
    effectively virtually torsion-free.
  \end{enumerate}
\end{lem}
\begin{proof}
  Any finite quotient in which every element of $\calF$ survives will
  have torsion-free kernel. (\ref{it:rf-vtf}) immediately follows.
  Since $G$ is finitely presented, it is possible to enumerate its
  finite index subgroups of a given index and it is routine to verify
  whether the elements of $\calF$ lie in a finite index subgroup $K$
  of $G$. Suppose such a finite index subgroup has order $k$. Then
  taking the intersection of all index $k$ subgroups of $G$ will give
  a finite index normal subgroup $H$ that will contain no conjugate of
  the elements of $\calF$ and will therefore be torsion-free.
\end{proof}

\begin{lem}\label{lem:find-finite}
  Let $(G,\calP)$ be relatively hyperbolic where the groups in $\calP$
  lie in an algorithmically tractable class with algorithmically
  bounded torsion.  Then we can compute a complete and finite list
  $\calF_G$ of conjugacy representatives of the finite order elements
  of the entire group $G$.
\end{lem}

\begin{proof}    
  It is sufficient to find the list containing a conjugacy
  representative of every finite subgroup of $G$. If a finite subgroup
  is non-parabolic then it has a conjugate that lies in the subset
  $P_2 \subset G$ given by Lemma \ref{lem:ball-of-finites}. Since we
  can solve the word problem we can enumerate all the finite subgroups
  contained in $P_2$. The list of conjugacy classes of parabolic
  finite subgroups of $G$ is computable by hypothesis.
\end{proof}

For this next proof, when we say that a pair $(G,\calP)$ is one-ended
or rigid, we mean that $G$ is one-ended or rigid (respectively)
relative to $\calP$.

\begin{prop}\label{prop:virt-rigid}
  Suppose that $(G,\calP)$ is relatively hyperbolic and let
  $G_0\leq G$ be a finite index torsion-free subgroup of $G$ endowed
  with the induced relative hyperbolic structure
  $(G_0,\calP_0)$. Suppose furthermore that $(G,\calP)$ is not an
  $\RQH$ group. Then
  \begin{enumerate}
  \item\label{it:virt-one-ended} $(G,\calP)$ is one-ended if and only
    if $(G_0,\calP_0)$ is one-ended,
  \item\label{it:virt-parabolic} if $(G,\calP)$ is one-ended, then it
    admits a non-trivial splitting over a parabolic subgroup if and
    only if $(G_0,\calP_0)$ admits a non-trivial splitting over a
    parabolic subgroup, and
  \item\label{it:virt-2-ended} if $(G,\calP)$ is one-ended and doesn't
    split over a parabolic subgroup, then it admits a non-trivial
    splitting over a 2-ended subgroup if and only if $(G_0,\calP_0)$
    admits a non-trivial splitting over $\bbZ$.
  \end{enumerate}
\end{prop}
\begin{proof}
  For items (\ref{it:virt-one-ended}) and (\ref{it:virt-parabolic}),
  we first that $(G,\calP)$ and $(G_0,\calP_0)$ have homeomorphic
  Bowditch boundaries (see \cite{Bow_periph}). The (dis)connectivity
  of the boundary implies (\ref{it:virt-one-ended}), by
  \cite[Proposition 1.1]{Bow_periph}. Suppose that $(G,\calP)$ is one
  ended, then the (non)existence of global cutpoints of the boundary
  implies (\ref{it:virt-parabolic}), by \cite[Theorem
  1.2]{Bow_periph}. It remains to show
  (\ref{it:virt-2-ended}). Suppose that $(G,\calP)$ is one-ended and
  doesn't split over a parabolic vertex group.
  
  If $(G,\calP)$ splits over a two-ended subgroup, then so does
  $(G_0,\calP_0)$. This is clear from the action of $G_0$ on the tree
  dual to the splitting of $G$, which must be non-trivial since
  $[G:G_0]<\infty$. Since $G_0$ is torsion-free then this splitting
  must be over $\bbZ$. We now prove the converse.

  Suppose that $(G_0,\calP_0)$ splits over $\bbZ$, but not over a
  parabolic vertex group.  Our approach will be to modify the
  parabolic subgroups to force them to act elliptically on Bass-Serre
  trees and then apply Papasoglu's quasi-isometric invariance of
  (non-relative) two-ended JSJ decompositions
  \cite{papasoglu_quasi-isometry_2005}.

  
  Let $\calP,\calP_0$ be given by a finite set of non-conjugate
  maximal parabolic subgroups and consider the amalgamated free
  products
  \begin{equation}
    \label{eq:Gstar}
     G^* = G \Asterisk_{P\in \calP}\left(P\oplus
      \bbZ^3\right), G_0^*=G_0 \Asterisk_{P\in \calP_0}\left(P\oplus
      \bbZ^3\right).
  \end{equation}
  Set $\calP^* = \{P\oplus\bbZ^3\mid P \in \calP\}$ and
  $\calP_0^* = \{P\oplus\bbZ^3\mid P \in \calP_0\}$. The groups
  $(G^*,\calP^*)$ and $(G_0^*,\calP_0^*)$ are relatively hyperbolic by
  \cite[Theorem 0.1]{Dah_CoC}.

  \emph{Claim: $G^*$ contains a subgroup $H$ isomorphic to $G^*_0$.}
  We construct a monomorphism $G_0^* \into G^*$ as follows. First we
  map $G_0$ to itself in $G \leq G^*$. Next each $Q \in \calP_0$ is
  sent to some conjugate $g^{-1}Pg=P_Q, g \in G, P \in \calP$, we
  extend this to a map $G_0^* \into G^*$ using mappings
  \begin{equation}
    \label{eq:augmentation-map}
    Q\oplus\bbZ^3 \into P_Q\oplus\bbZ^3    
  \end{equation}
  that induce surjections (via projection) onto the $\bbZ^3$
  factor. Let $H$ be the image of this map, and identify $G^*_0$ with
  its image. Consider the Bass-Serre tree $T$ dual to the splitting of
  $G^*$ given in (\ref{eq:Gstar}).  Let $v$ be the vertex stabilized
  by $G \leq G^*$ and let $u$ be some vertex adjacent to $v$. Then the
  edge $(v,u)$ is stabilized by some conjugate
  $g^{-1}Pg, P\in \calP, g\in G$ and also some
  $h^{-1}Qh, Q \in \calP_0, h\in G_0$. The vertex $u$ is stabilized by
  $g \left(P\oplus\bbZ^3\right) g^{-1}$ and since the $P$ factor fixes
  edges, the edges adjacent to $u$ are in equivariant bijective
  correspondence with the elements in the $\bbZ^3$ factor of
  $g \left(P\oplus\bbZ^3\right) g^{-1}$. Now $u$ is also stabilized by
  $h^{-1}\left(Q\oplus\bbZ^3\right)h$, and since the mappings
  (\ref{eq:augmentation-map}) induce surjections on the $\bbZ^3$
  factors, we have that the edges adjacent to $u$ are all in a single
  $\left(G^*\right)_u$ and $\left(G^*_0\right)_u$ orbit. Thus, every
  vertex at distance 2 from $v$ is in the same $G^*_0$ orbit as
  $v$. We can now show that $[G^*:H]<\infty$.

  Let $g\in G^*$, then there is some $k\in H$ such that $kg$
  fixes $v$. Suppose towards a contradiction that this wasn't the
  case. Then there is some $g\in G^*$ such that $g\cdot v = w$ and $w$
  is closest to $v$ in its $G^*_0$-orbit. Let $\alpha$ be the path
  from $v$ to $w$ in $T$ and let $v_2$ be the vertex at distance $2$
  from $v$ in $\alpha$. Then there is some $k \in G^*_0$ such that
  $k\cdot v_2 = v$, which means that $k\cdot w$ is closer to $v$ than
  $w$ -- contradiction.

  It follows that for every $g \in G^*$ there is some $k \in G^*_0$
  such that $kg \in G$. Thus right $G_0^*$ coset representatives can
  always be taken in $G$. It follows that $[G^*,G^*_0] \leq [G:G_0]$
  and the claim is proved.
  
  %

  \emph{Claim: If $(G_0,\calP_0)$ is one-ended and does not admit a
    splitting over a parabolic subgroup then $G_0^*$ is one ended.}
  Suppose that $G^*_0$ was not one-ended. We first note that that the vertex groups
  $(P\oplus\bbZ^3) \in P_0^*$, must be elliptic in this splitting as
  they are forced to be one-ended. Let $T$ be the splitting of $G^*_0$
  given in (\eqref{eq:Gstar}). By \cite[Theorem
  1.4]{touikan_one-endedness_2015} $T$ can be blown up to a tree
  $\widecheck{T}$ in which the subgroups in $P_0^*$ are elliptic, the
  edge groups are factors in a one ended splitting of the
  $P_i \in \calP_0^*$ and $G_0$ now acts non-trivially on a subtree of
  $\widecheck{T}$ where the groups in $\calP^*_0$ are elliptic and the
  edge groups lie in $\calP_0$. This contradicts the assumption that
  $(G_0,\calP_0)$ doesn't split over a parabolic subgroup. The claim
  is proved.

  So far the (\ref{eq:Gstar}) construction has enabled us to pass from
  relative one-endedness to ``absolute'' one-endedness. By hypotheses,
  the JSJ splitting of $(G_0,\calP_0)$ only has non-parabolic $\bbZ$-edge groups. First note that in any 2-ended splitting of $G_0^*$,
  every subgroup of the form $\left(P\oplus\bbZ^3\right) \in \calP_0^*$ must be
  elliptic, for otherwise it contains some hyperbolic element
  $g$ fixing an axis $\alpha$ and the elements centralizing $g$ map
  $\alpha$ to itself. This gives a map $Z\left(P\oplus\bbZ^3\right) \onto \bbZ$ whose kernel
  fixes edges of the dual tree. Since every centralizer of an element
  of $\left(P\oplus\bbZ^3\right) \in \calP_0^*$ contains $\bbZ^3$,
  this forces an edge group of the splitting to contain $\bbZ^2$,
  contradicting the assumption that it is two-ended.

  It therefore follows that the $\bbZ$-JSJ decomposition $\bbJ^*_0$
  of $G_0^*$ can be obtained by refining the $G_0$ factor in the
  splitting (\ref{eq:Gstar}) to a JSJ splitting and collapsing all the
  non-cyclic parabolic edges in (\ref{eq:Gstar}). The vertex groups
  fall into two categories. They are called \emph{thick}, if they are
  of the form $P\oplus\bbZ^3$ and \emph{relatively hyperbolic}
  otherwise. The edge groups of $\bbJ^*_0$ fall into two categories:
  either they are incident to a thick vertex group, in which case they
  are called \emph{half-thick}, otherwise they are called
  \emph{non-half-thick}.

  Since $G_0^*$ is one-ended we may apply \cite[Theorem
  7.1]{papasoglu_quasi-isometry_2005} which, because $G^*$ is
  quasi-isometric to $G_0^*$ implies that this quasi-isometry sends
  every vertex or edge group of the JSJ decomposition of $G_0^*$ to a
  bounded neighbourhood of a vertex or edge group (respectively) of
  the JSJ decomposition $\bbJ^*$ of $G^*$. 
  %
  %
  Since
  we assumed that $G_0$ did not admit any splittings over parabolic
  groups, half-thick edge groups and non-half-thick edge groups of $G_0^*$
  never remain within a bounded neighbourhood of one another. 
  %
 %
  We further note that the dichotomy of being thick or relatively
  hyperbolic (among the vertex groups of the decomposition of $G_0^*$)
  is a quasi-isometry invariant. Indeed, on the one hand, \cite[Theorem
  1.2]{Drutu_QI_invariance} implies that admitting a proper relatively
  hyperbolic structure is a quasi-isometry invariant. On the other
  hand, thick groups can not be properly relatively hyperbolic because
  they contain an infinite normal abelian subgroup.
  	It therefore follows that, via
  Papasoglu's correspondence, half-thick edge groups are sent to half-thick edge
  groups and non-half-thick edge groups are sent to non-half-thick
  edge groups. 

  Take now the splitting $\bbX^*$ of $G^*$ given by (\ref{eq:Gstar})
  obtained by collapsing the non-virtually cyclic edge groups. Denote
  by $\hat{G}$ the vertex group containing $G$. $\bbX^*$ is a
  splitting with two-ended vertex groups. As noted before the vertex
  groups $\left(P\oplus\bbZ^3\right) \in \calP^*$ must always be
  elliptic. It follows that the JSJ decomposition $\bbJ^*$ can be
  obtained by refining the splitting $\bbX^*$ in the vertex group
  $\hat{G}$.

  \emph{Case 1: $\bbJ^*$ only has half-thick edge groups.}  This is
  only possible if the JSJ decompositions of $G^*$ and $G^*_0$ are
  exactly the same as the splittings given by (\ref{eq:Gstar}) and,
  since $G^*_0$ splits over a 2-ended group, $G_0$ must be a $\QH$
  vertex group.  This also implies that all peripheral vertex groups
  are virtually cyclic. In this case $(G,\calP)$ is a finite index
  overgroup of a $\QH$ vertex group and splits over a 2-ended subgroup
  if and only if it is not an $\RQH$ group.

  \emph{Case 2: $\bbJ^*$ has a non-half-thick edge group.} Then $\hat{G}$
  inherits a non-trivial splitting with 2-ended edge groups. Let $T$
  be the tree dual to this splitting. Let $\calP' \subset \calP$ be
  the set of maximal parabolics that are not 2-ended. On the one hand
  we have a splitting\[ \hat G = G\Asterisk_{P\in \calP'}\left(P \oplus
      \bbZ^3\right). \] Now suppose towards a contradiction that $G \leq \hat{G}$ was elliptic
  in the splitting induced on $\hat G$. Let $v$ be some vertex of the
  dual Bass-Serre tree $T$ fixed by $G$, all the groups
  $(P\oplus\bbZ^3), P \in \calP'$ must also act elliptically. Suppose
  that one such group fixed a vertex $w\neq v$. Then the subgroup
  $P \in G \cap (P\oplus\bbZ^4)$ fixes the entire path from $v$ to $w$
  pointwise, which implies that $P$ stabilizes an edge, contradicting
  the assumption that $T$ had 2-ended stabilizers.
  
  Thus, we have shown that if $(G_0,\calP_0)$ is one-ended and does
  not split over a parabolic vertex group, but splits over a
  non-parabolic cyclic group, then either $(G,\calP)$ is an $\RQH$
  group or it splits over a non-parabolic 2-ended vertex group.
  
  

\end{proof}

\begin{prop}\label{prop:detect-rigid}
  Suppose that $(G,\calP)$ is relatively hyperbolic and $\calP$ lies
  in a heriditarily algorithmically tractable class with
  algorithmically bounded torsion. Suppose furthermore that $G$ is is
  virtually torsion-free. Then we can decide if $(G,\calP)$ is
  one-ended, and if so, if it is rigid.
\end{prop}
\begin{proof}
  First we apply the algorithm of Lemma \ref{lem:recognize-rqh} to
  decide if $(G,\calP)$ is $\RQH$. If this is the case then
  $(G,\calP)$ is rigid.

  Suppose now that $(G,\calP)$ is not $\RQH$. By Lemma
  \ref{lem:find-finite} and Lemma \ref{lem:etf-exits} we can compute a
  finite index torsion-free subgroup $G_0$ of $G$. By Proposition
  \ref{prop;compute-Grushko} we can decide if $(G_0,\calP_0)$ is
  freely decomposable. Thus by Proposition \ref{prop:virt-rigid}
  (\ref{it:virt-one-ended}) we can decide if $(G,\calP)$ is one-ended.

  If $(G_0,\calP_0)$ is freely indecomposable, then it follows by
  Proposition \ref{prop:virt-rigid} (\ref{it:virt-parabolic}) and
  (\ref{it:virt-2-ended}) that $(G,\calP)$ is rigid if and only if
  $(G_0,\calP_0)$ is rigid. This can be decided by Theorem
  \ref{theo;TouC}.

  

\end{proof}

Proposition \ref{prop:detect-rigid} with Theorem \ref{thm;compute_JSJ} and \cite{Dunwoody-1985} immediately
imply the following:

\begin{cor}\label{cor:ctf-compute-can-splittings}
  Let $(G,\calP)$ be relatively hyperbolic and virtually torsion-free,
  and let $\calP$ lie in a heriditarily algorithmically tractable
  class with algorithmically bounded torsion.  Then it is possible to
  compute a Dunwoody-Stallings decomposition $\bbD$ of $(G,\calP)$ and
  compute the JSJ decomposition of each one-ended vertex group of
  $\bbD$.
\end{cor}

\begin{rem} This of course includes the case where $G$ is torsion
  free, and $\calP$ lies in a heriditarily algorithmically tractable
  class.
\end{rem}

Besides passing to torsion-free finite index subgroups, there may be
other methods to detect one-endedness. This should be doable for
relatively hyperbolic groups (possibly with torsion) with parabolic
subgroups that are virtually nilpotent, but developing the detail of
it would bring us far astray of the primary goal of this paper. Let us
give a roadmap of how such a reduction could go.

First, the tool from \cite{DGr_tams} that we just used can be
generalized. The argument of \cite{DGr_tams} can be used to detect if
a relatively hyperbolic has infinitely many (relative) ends, provided
the parabolic groups satisfy the technical conclusion of \cite[Lemma
2.15]{DGr_tams}. Specifically they must have a computable generating
set and there must be a computable constant $M$ such that, for any
$r>0$, and elements $a,b$ at distance $\geq r$ from $1$, there is a
path of length $\leq Md(a,b)$ from $a$ to $b$ avoiding the ball of
radius $r-1$ centered at $1$. This is actually a sufficient condition
in order to get \cite[Lemma 2.16]{DGr_tams}, which is the only place
where some specificity of parabolic subgroups is required.  Many
groups, including finitely generated nilpotent groups (non virtually
cyclic), have linear divergence $div_2(n,1/2)$ \cite[Definition 3.3,
Proposition 1.1]{DMS10}. We found that it is possible to adapt the
proof of \cite[Lemma 2.16]{DGr_tams} to the case of parabolic groups
of linear divergence $div_2(n,1/2)$. This applies to virtually
nilpotent groups by \cite{DMS10}.

Secondly, the argument of \cite[\S 7, Lemma 7.5]{DG_gafa}, that allows
us to compute a list of representatives of the different orbits of
Dunwoody-Stallings decompositions of a relatively hyperbolic group
under its automorphism group, is applicable to a class  $\mathcal{C}$ of relatively
hyperbolic groups provided (a) the number of conjugacy classes of
finite subgroups in each group in $\mathcal{C}$ is finite, (b) there is an
effective procedure to compute the centralizer and normalizer of these finite subgroups, (c)
the isomorphism problem for the relatively one-ended groups in $\mathcal{C}$, with marked peripheral structure consisting of finite
subgroups, is solvable.  All three conditions are satisfied by
relatively hyperbolic groups with virtually nilpotent parabolic
subgroups.

Furthermore Barrett \cite{barrett2016computing} recently adapted these
methods to detect the rigidity of hyperbolic groups with torsion.

\section{Graph of groups isomorphisms}\label{sec;gog-isos}

The following definition of an isomorphism of graphs of groups follows
\cite[\S 2.3]{Bass}; this recapitulation was also recorded in \cite[\S
2.7.2, 2.7.3]{DG_gafa}. Recall Definition \ref{defn;gog} of a graph of
groups.

\begin{defi}\label{def;iso_gog}
  Given two graphs of groups on the same abstract graph $X$,
  $\mathbb{X}=(X, \{\Gamma_v, v\in V(X)\}, \{\Gamma_e, e\in E(X)\},
  \{i_e, e\in E(X)\})$, and $ \mathbb{X}'=(X, \{\Gamma'_v\},
  \{\Gamma'_e, e\in E(X) \}, \{i'_e, e\in E(X)\}) $, an
  \emph{isomorphism of graph of groups} $\Phi : \mathbb{X}' \to
  \mathbb{X}$ is a tuple $( \{\phi_v, v\in V(X) \}, \{\phi_e, e\in
  E(X)\}, \{\gamma_e, e\in E(X)\} )$ where:
\begin{enumerate}
\item for all $v\in V(X)$, $\phi_v: \G'_v \to \G_v$ is an isomorphism,
\item for all $e\in E(X)$ $\phi_e:\G'_e\to \G_e$ is an isomorphism,
  and $\phi_e = \phi_{\bar e}$
\item for all $e\in E(X)$, $\gamma_e$ is an element of $G_{t(e)}$ such
  that the diagrams commute:
  \begin{equation}\label{diag;def_isom_gog}  
    \begin{array}{ccc} \G'_{t(e)} & \stackrel{i'_e}{ \hookleftarrow} & \G'_e  \\
      \phi_{t(e)} \! \downarrow & & \downarrow \! \phi_e \\
      \G_{t(e)} &    \stackrel{\ad_{\gamma_e}}{\leftarrow} \G_{t(e)} \stackrel{i_e}{\hookleftarrow} & \G_e. \end{array}   \end{equation}
\end{enumerate}

The maps $\phi_v$ are called the vertex maps, the maps $\phi_e$ are
the edge maps, and the elements $\gamma_e$ are the attaching elements.
\end{defi}

\subsection{The extension problem}

Suppose we are given two graphs groups $\bbX$, $\bbX'$ and a set of
isomorphisms between their edge groups and their vertex groups. We are
interested in necessary and sufficient conditions to extend these to
an isomorphism $\bbX \stackrel{\sim}{\to} \bbX'$. The following is
immediate from the definition.

\begin{prop}
  Given two structures of graphs of groups $\mathbb{X}', \mathbb{X}$
  on the same graph $X$, and an isomorphism of graph of groups $\Phi:
  \mathbb{X}' \to \mathbb{X} $ as above, each vertex map of $\Phi$
  induces an isomorphism of unmarked peripheral structures $\phi_v
  :(\G'_v, \calA'_u)\to (\G_v, \calA_u)$, where $\calA'_u$ and
  $\calA_u$ are the adjacency peripheral structures induced by the
  graphs-of-groups.
\end{prop}

We now define extension adjustments, which will relate an arbitrary
collection of unmarked peripheral structures between the vertex
groups, to an isomorphism of graph of groups.

\begin{defn}\label{defn:extn-adj}
  Let $\mathbb{X}'$, and $ \mathbb{X}$ be two graphs of groups on the
  same graph $X$ let $\Psi$ be a collection of isomorphisms of groups
  $\Psi = \{ \psi_v: \G'_v \to \G_v, \, v\in V(X)\}$. An
  \emph{extension adjustment on $(\mathbb{X}', \mathbb{X})$ with
    respect to $\Psi$} is a collection $\{ \alpha_v: \G_v \to \G_v, \,
  v\in V(X) \}$ of automorphisms, and a collection $\{ g_e \in
  \G_{t(e)}, \, e\in E(X) \}$ such that:
  \begin{enumerate}
  \item\label{it:extn-adj-first-point} $g_e$ conjugates $ \alpha_{t(e)}\circ \psi_{t(e)} \circ i'_e
    (\G'_e)$ to $i_e (\G_e)$ in $\G_{t(e)}$,
  \item the following diagram (in which we abuse notation: the bottom
    line is not defined everywhere, but it is on the image of the
    vertical arrows, by the previous point), is a commutative diagram:
    \begin{equation}\label{bigdiag}\begin{array}{rcccl}
        \G'_{t(e)}  & \stackrel{i'_e}{\hookleftarrow} & \G'_e = \G'_{\bar e} & \stackrel{i'_{\bar e}}{\longhookrightarrow} & \G'_{t(\bar e)} \\
        \psi_{t(e)} \! \Big  \downarrow & & & & \hfill \Big \downarrow \! \psi_{t(\bar e)} \\
        \G_{t(e)} & & & &  \G_{t(\bar e)} \\
        \alpha_{t(e)} \! \Big \downarrow  & & & &  \hfill \Big \downarrow \! \alpha_{t(\bar e)}  \\
        \G_{t(e)} & & & &  \G_{t(\bar e)} \\
        \ad_{g_{e}}  \! \Big \downarrow  & & & &   \hfill \Big  \downarrow \! \ad_{g_{\bar e}} \\
        \G_{t(e)} &  \stackrel{(i_{e})^{-1}}{\longrightarrow} & \G_e = \G_{\bar e} &  \stackrel{(i_{\bar e})^{-1}}{\longleftarrow} & \G_{t(\bar e)}
      \end{array}
    \end{equation} 
  \end{enumerate}
\end{defn}

Note that (\ref{it:extn-adj-first-point}) above implies that the
composition $\alpha_v \circ \psi_v$ is an isomorphism of groups with
unmarked peripheral structure (their adjacency peripheral structure)
$(\G'_v, \calA'_u) \to (\G_v, \calA_u )$. Our motivation here is to
express the isomorphism of graph of groups avoiding the ``choice'' of
a good isomorphism of edge groups.

\begin{prop}\label{prop;compatible_twist}
  Let $X$ be a finite graph, and let $\mathbb{X}$, and $ \mathbb{X}'$
  be two structures of graph of groups on $X$.

  Assume that there exists a collection of isomorphisms $\Psi = \{
  \psi_v: \G'_v \to \G_v, \, v\in V(X)\}$.  Then, the graphs of groups
  $\mathbb{X}'$ and $\mathbb{X}$ are isomorphic if, and only if there
  exists an extension adjustment on $(\mathbb{X}', \mathbb{X})$ with
  respect to $\Psi$.
\end{prop}

\begin{proof}
  Assume that $\psi_v$, $g_e$ and $\alpha_v$ are given as in the
  statement.  Then we set $\phi_v= \alpha_v\circ \psi_v$, and $\phi_e$
  to be the map $\Gamma'_e\to \G_e$ given by (\ref{bigdiag}), i.e. we
  have:\begin{equation}\label{diag;mediumdiag}
\begin{array}{c}
\begin{tikzpicture}[scale=1.2]
  \node (G'e) at (4,0) {$\Gamma'_e$};
  \node (G'te) at (0,0) {$\Gamma'_{t(e)}$};
  \node (Gte1) at(0,-1) {$\Gamma_{t(e)}$};
  \node (Gte2) at(0,-2) {$\Gamma_{t(e)}$};
  \node (Gte3) at(2,-2) {$\Gamma_{t(e)}$};
  \node (Ge) at (4,-2) {$\Gamma_e$};
  \draw[left hook->] (G'e) --node[above]{$i'_e$} (G'te);
  \draw[->] (G'te) --node[right]{$\psi_{t(e)}$} (Gte1);
  \draw[->] (Gte1) -- node[right]{$\alpha_{t(e)}$} (Gte2);
  \draw[->] (Gte2) --node[below]{$\ad_{g_e}$} (Gte3);
  \draw[->](G'e) --node[right]{$\phi_e$} (Ge);
  \draw[left hook->] (Ge) --node[below]{$i_e$} (Gte3);
  \draw[->] (G'te) .. controls +(-0.4,-0.5) and +(-0.4,0.5)  .. node[left]{$\phi_v$} (Gte2); 
\end{tikzpicture}
\end{array}
\end{equation}
This way, we have $\phi_e= \phi_{\bar e}$. To get a graph of groups
isomorphism in the sense of Definition \ref{def;iso_gog}, it remains
to check that Diagram (\ref{diag;def_isom_gog}) commutes for some
$\gamma_e\in \G_v$.  It is immediate from (\ref{diag;mediumdiag}) that
$\gamma_e = g_e^{-1}$ is the required element.

In the other direction, assume that the graphs of groups are
isomorphic, and let $\Phi= ( \{\phi_v, v\in V(X) \}, \{\phi_e, e\in
E(X)\}, \{\gamma_e, e\in E(X)\} )$ be an isomorphism. Now for the
given isomorphisms $\psi_v$, the elements $\alpha_v = \phi_v \circ
\psi_v^{-1}$, and $g_e = \gamma_e^{-1}$ are immediately seen to be
extension adjustments w.r.t $\Psi$ by comparing diagrams
(\ref{diag;def_isom_gog}) and (\ref{diag;mediumdiag}.)
\end{proof} 

\subsection{Reduction to orbit problems}

Recall that a marking of an unmarked peripheral structure $\calP_u$ is
a marked peripheral structure inducing $\calP_u$.

\begin{prop}[Reduction to algorithmic problems in the vertex groups] 
  \label{prop;main_reduction}
  Let $\calW$ be a class of groups with unmarked ordered peripheral
  structures, and $\calB$ be a class of groups.Assume that
  \begin{enumerate}
  \item \label{orbit1} for all $(G,\calP_{u})$ in the class $\calW$,
    \begin{enumerate} 
    \item \label{orbit1b} the orbit of a given marking of the unmarked
      ordered peripheral structure $\calP_{u}$ under the action of
      $\Out(G, \calP_{u})$, is finite and uniformly computable;
    \item \label{assumeIP} the isomorphism problem (for groups with
      unmarked ordered peripheral structures) is effectively decidable
      in the class $\calW$;
    \end{enumerate}
  \item in the class $\calB$,
    \begin{enumerate} 
    \item\label{bl-mwhp} the orbit problem of $\Aut(G)$ on tuples of conjugacy classes
      of tuples of $G$, i.e. the mixed Whitehead problem, is
      uniformly decidable;
    \item\label{bl-iso} the isomorphism problem (for groups without peripheral
      structure) is  decidable.
    \end{enumerate}
 \end{enumerate}

    Consider a bipartite graph $X$ (with black and white vertices:
    $V(X) = BV(X) \sqcup WV(X) $), and choose an order on each
    (oriented) link of each vertex. Consider the class of
    graphs-of-groups on $X$ with black vertex groups in $\calB$ and
    white vertex groups in $\calW$, and edge groups such that the
    adjacency peripheral structure on a white vertex group is that
    given by $\calW$.

  Then, the graph-of-groups isomorphism problem is solvable for this
  class of graphs of groups.
\end{prop}

Note that in the assumption on the black vertices, we forget about the
peripheral structure.  Note also that in the case of the white, the
assumption is stronger: not only the orbit problem on such objects is
solvable, but each orbit is finite (and computable.)

\begin{proof} 
  If both graphs of groups are isomorphic, this
  will be discovered by enumeration. We need an algorithmic
  certificate that they are not isomorphic. If there is $w\in WV(X)$
  such that the groups with adjacency (unmarked ordered) peripheral
  structures $(\G_w, \calA_{w})$ and $ (\G'_w, \calA'_{w})$ are not
  isomorphic, this will be discovered by assumption (\ref{assumeIP}) on
  $\calW$. If there is $b\in BV(X)$ such that $\G_b$ and $\G'_b$ are
  not isomorphic, this will be discovered by assumption
  (\ref{bl-iso}) on $\calB$.  Hence we assume that they are
  isomorphic, and, after an enumeration chase, that we know some
  $\psi_w: (\G'_w, \calA'_w) \to (\G_w, \calA_w)$ for each $w\in
  WV(X)$ and $\psi_b: \G'_b\to \G_b$ for each $b\in BV(X)$. Let $\Psi$
  be this collection.

  By Proposition \ref{prop;compatible_twist}, the two graphs of groups
  fail to be isomorphic if and only if there is no extension
  adjustment for $\Psi$. For white vertices, this unmarked peripheral
  structure is prescribed by assumption on $\calW$, but that this is
  not the case for black vertices. Therefore, for all white vertices
  $w$, the automorphisms $\alpha_w$ must preserve the unmarked
  peripheral structure of their adjacent edge groups, in order to be
  part of an extension adjustment.
  
  By assumption (\ref{orbit1b}), for each white vertex $w$, one can
  compute a finite list of automorphisms $\{\alpha_{w,i}\}$ of the
  vertex group $\Gamma_w$ realizing the orbit of any given marking
  $\calA_{m}(w)$ of $\calA_{w}$ under $\Out(\Gamma_w, \calA_{w})$.
  For each choice of representatives, one for each $w \in WV(X)$, from
  the finite collections of automorphisms $\{\alpha_{w,i}, w \in
  WV(X)\}$, we need to decide whether there exists a collection of
  automorphisms $\{\alpha_b, b \in BV(X)\}$ that can complete an
  extension adjustment (for a certain choice of conjugating elements
  $g_e, e\in E(X)$.)

  For that we use the solution given to the orbit problem on tuples of
  tuples under automorphisms of black vertex groups given by
  assumption (\ref{bl-mwhp}.) Indeed, given a black vertex $b$, and
  $e$ an adjacent edge ($t(e) = b$) whose other end is a white vertex
  $w=t(\bar e)$, once chosen a generating tuple for $\Gamma'_e$, and
  an automorphism $\alpha_w = \alpha_{t(\bar{e})}$ (there are only
  finitely many of these to consider), one obtains a generating tuple
  of $\Gamma_e$ by chasing down the right side of Diagram
  (\ref{bigdiag}) that is well defined up to conjugacy; this is
  because $\bar{g}_e$ has not been fixed yet.

  Obtaining $\alpha_b = \alpha_{t(e)}$ so that the diagram commutes is
  exactly solving the orbit problem of the conjugacy class of this
  tuple in $\G_b$. However, the automorphism $\alpha_b$ should make
  the diagrams associated to all edges adjacent to $b$ commute; this
  is exactly solving the orbit problem of a tuple of conjugacy classes
  of tuples, or the mixed Whitehead problem, in $\G_w$. 
\end{proof}

\section{Orbits of markings in rigid relatively hyperbolic groups}
\label{sec;orbits}

In this section we show how assumption (\ref{orbit1b}) of Proposition
\ref{prop;main_reduction} is satisfied by relatively hyperbolic groups
with residually finite parabolic groups in which congruences
effectively separate the torsion in the outer automorphism
groups. This will reduce the isomorphism problem for relatively
hyperbolic groups to algorithmic problems in the parabolic subgroups
thus yielding the main result.

\subsection{Computing orbits via Dehn fillings and congruences}\label{sec;orbit_computation}

Consider the class $\calG$ of groups with an unmarked ordered
peripheral structure, $(G, \calP_{uo})$, such that
\begin{itemize}
\item $G$ is finitely presented, and $(G,\calP_{uo})$ is relatively hyperbolic, 
\item groups of $\calP_{uo}$ lie in a class of infinite, residually finite
  groups, with congruences effectively separating torsion (see
  Definition \ref{def;CST}),
\item $(G, \calP_{uo})$ admits no splitting over an elementary
  subgroup.
\end{itemize}

Our main objective here is the following. Recall the terminology
introduced in Section \ref{sec:periph-struct}.

\begin{prop}\label{prop;orbit_computation} 
  There is an algorithm that, given $(G,\calP_{uo})$ in the class
  $\calG$ and any marking $\calP_m$ of the peripheral structure
  $\calP_{uo}$, computes the orbit of $\calP_m$ under
  $\Out(G,\calP_{uo})$, or equivalently computes a set of coset
  representatives of \[\Out(G,\calP_{uo})/\Out(G,\calP_{m})\] in
  $\Out(G,\calP_{uo})$.
\end{prop}

In order to compute a system of representatives of
$\Out(G,\calP_{uo})/\Out(G,\calP_{m})$ in $\Out(G,\calP_{uo})$, we
will prove that there is a Dehn filling of $G$, cofinite in the
parabolic subgroups, that induces a {\it bijection}
$\Out(G,\calP_{uo})/\Out(G,\calP_{m}) \to \Out(\bar G,\bar
\calP_{uo})/\Out(\bar G, \bar\calP_{m})$.  This helps because we can
compute the various groups associated to automorphism of hyperbolic
groups. In particular, we can explicitly find $\Out(\bar G,\bar
\calP_{uo})/\Out(\bar G, \bar\calP_{m})$.

The stages of the proof will be as follows. First, we need to define
the map
\[\pi:\Out(G,\calP_{uo})/\Out(G,\calP_{m}) \to \Out(\bar G,\bar
  \calP_{uo})/\Out(\bar G, \bar\calP_{m}).\] Surjectivity will be
achieved by invoking \cite{DG_charac}, and the limit argument
there. Our main task is to show injectivity, and for that, we will
push the situation into the parabolic groups, see Diagram
(\ref{diag;injectivity}.)  The assumption of congruences separating
the torsion will be sufficient to ensure that for sufficiently deep
Dehn fillings, the map $\pi$ will be injective.

The next difficulty is to show that the vertical arrows of the diagram
are well defined.  The left vertical arrow is well defined because the
infinite maximal parabolic subgroups of a relatively hyperbolic group are
self-normalized. In a hyperbolic Dehn filling, however, since the
group is hyperbolic relative to a collection of finite groups, the
situation is not so clear. We will show that this is also the case for
sufficiently deep Dehn fillings by examining the geometry more
closely.

\subsubsection{The map $\pi:\Out(G,\calP_{uo})/\Out(G,\calP_{m})  \to \Out(\bar G,\bar \calP_{uo})/\Out(\bar G, \bar\calP_{m})  $}

The operation of Dehn fillings is an important tool. In the context of
relatively hyperbolic groups, it was first studied by Osin, and by
Groves and Manning.

\begin{theo}[{\cite[Dehn filling theorem]{Osi_DF,
    GrM_DF}}] \label{theo;Dehn_fill}

  Let $(G,\{[P_1],\dots,[P_k]\})$ be a relatively hyperbolic
  group. There exists a finite set $X$ in $G\setminus\{1\}$ such that
  whenever $N_i \normal P_i$ avoids $X$, the image of $P_i$ in the
  group $G/\langle\langle \bigcup_i N_i \rangle\rangle$ is canonically
  $P_i/N_i$, and $G/\langle\langle \bigcup_i N_i \rangle\rangle$ is
  hyperbolic relative to $\{ [P_1/N_1],\dots,[P_k/N_k]\}$.
\end{theo}

In this case, we say that $G/\langle\langle \bigcup_i N_i
\rangle\rangle$ is a \define{Dehn filling of $G$ by $K= \langle\langle
  \bigcup_i N_i \rangle\rangle$}. A sequence of Dehn fillings by $K_n
= \langle\langle \bigcup_i N_{i,n} \rangle\rangle $ is called
\emph{cofinal} if $\bigcap_n K_n = \{1\}$. A convenient property of
Dehn fillings is that the cofinality of $K_n$ is equivalent to the
cofinality of each $ N_{i,n}$: $\forall i, \bigcap_n N_{i,n} = \{1\}$
(see for instance \cite[Theorem 7.9]{DGO}.)

In the following, $(G, \{[P_1],\dots,[P_k]\})$ will be a relatively
hyperbolic group.  $S_i$ will be a finite generating set of $P_i$, so
that $\calP_{uo} = ([P_1],\dots,[P_k])$, and $\calP_{m} =
([S_1],\dots,[S_k])$ are respectively the unmarked ordered peripheral
structure, and the marked peripheral structure on $G$ associated to
$\calP$ (and the implicit choices.) We first set the stage and define $\pi$.

\begin{lemma}\label{lem;pi_well_definied}

  Let $N_i \normal P_i$ be a characteristic subgroup of $P_i$.  Let
  $\bar G= G/\langle\langle \bigcup_i N_i \rangle\rangle$, $\bar
  \calP_{uo}= ([P_1/N_1],\dots,[P_k/N_k])$, and $\bar \calP_{m} =
  ([\bar S_1],\dots,[\bar S_k])$. Then, the quotient map $G\to \bar G$
  induces a morphism $\pi: \Out(G, \calP_{uo}) \to \Out(\bar G, \bar
  \calP_{uo} )$, that sends $\Out(G, \calP_{m})$ inside $\Out(\bar
  G/\bar \calP_m )$.

\end{lemma}

\begin{proof}
  If an automorphism $\phi$ of $G$ preserves $[P_i]$, it has the same
  class in $\Out(G)$ than an automorphism preserving $P_i$, and
  therefore $N_i$, since the later is characteristic in $P_i$. It
  follows that $\phi$ preserves $[N_i]$. If this happens for all $i$,
  $\phi$ thus defines an element $\bar \phi$ of $\Out(\bar G)$, and it
  preserves $\bar \calP_{uo}$. The map $\pi: \Out(G, \calP_{uo}) \to
  \Out(\bar G, \bar \calP_{uo} )$ hence defined is clearly a
  homomorphism. If $\phi$ preserves $[S_i]$, $\bar \phi$ preserves $[\bar
  S_i]$.
\end{proof}

\begin{lemma}\label{lem;surjective}
  Assume that $(G, \calP_{uo})$ is a relatively hyperbolic group that
  does not split over an elementary subgroup. Then, there exists a
  finite set $X'$ in $G\setminus\{1\}$ such that whenever one chooses,
  for each $i$, $N_i \normal P_i$ avoiding $X'$, characteristic and of
  finite index in $P_i$, the homomorphism $\pi: \Out(G, \calP_{uo}) \to
  \Out(\bar G, \bar \calP_{uo} )$ is surjective.
\end{lemma}

\begin{proof}
  We argue by contradiction.  We invoke \cite[Theorem 5.2]{DG_charac} (applied to $G=G'$) which
  says that if $(G, \calP_{uo})$ is a relatively hyperbolic group that
  does not split over an elementary subgroup, then for any cofinal
  sequences of Dehn fillings by $K_n$, and any sequence of
  automorphism $\phi_n: (G/K_n,\bar
  \calP_n)\stackrel{\sim}{\to}(G/K_n,\bar\calP_n)$, there is an
  automorphism $\phi :(G,\calP)\stackrel{\sim}{\to} (G,\calP)$ that
  induces, for infinitely many $n$, a composition of $\phi_n$ with a
  conjugation.

  Let $X'_n, (n\in \bbN)$ be an exhaustion of $G\setminus\{1\}$ by
  finite sets. By assumption, there is an associated Dehn filling by
  characteristic finite index subgroups of $P_i$, such that the
  homomorphism $\pi: \Out(G, \calP_{uo}) \to \Out(\bar G_n, \bar{
    (\calP_{uo})_n })$ is not surjective, and in which $X'_n$
  survives.  
  One can
  produces a sequence $\phi_n$ of automorphisms of $(\bar G_n, \bar{
    (\calP_{uo})})_n$, by choosing an automorphism not in the image,
  for each $n$.  \cite[Theorem 5.2]{DG_charac}, that we just recalled
  provides an automorphism $\phi :(G,\calP)\stackrel{\sim}{\to}
  (G,\calP)$ that induces, for infinitely many $n$, a composition of
  $\phi_n$ with a conjugation. That is a contradiction.
\end{proof}

\subsubsection{Injectivity of  $\pi$: Congruences separating the
  torsion} \label{subsec;injpi}

The next lemma will serve to introduce the separation of torsion by
congruences. If $P_0$ is a finite index characteristic subgroup, we
call \define{$P_0$-congruence} the quotient map $P\to P/P_0$. Note
that, because $P_0$ is characteristic, this quotient map induces a
homomorphism $\Out P \to \Out P/P_0$. Its kernel is called a
\define{$P_0$-congruence subgroup of $\Out P$} (in generalization of
the case $P= \mathbb{Z}^n$.)

The desirable property is the existence of torsion-free congruence
subgroups.

\begin{defi}\label{def;CST}
  Given a group $P$, and a characteristic finite index subgroup $P_0$,
  we say that \emph{the $P_0$-congruence separates the torsion of
    $\out P$} if the $P_0$-congruence subgroup of $\out P$ is torsion
  free.  In a class of groups, \emph{congruences separate the torsion}
  if for each group $P$ in this class, there is a congruence in $P$
  that separates the torsion of $\Out(P)$. In a class of groups,
  \emph{congruences effectively separate the torsion} if there is an
  algorithm that, given a presentation of a group in the class,
  produces a congruence separating the torsion.
\end{defi}

\begin{rem}\label{rem;GLnZ3} Let us emphasize that this is a property of outer
  automorphisms groups. In the abelian group $\mathbb{Z}^n$,
  congruences separate the torsion of $GL(n,\bbZ)$, and more
  precisely, that the kernels of the quotient on
  $GL(n,\mathbb{Z}/p\mathbb{Z})$ are torsion-free for $p\geq 3$, by a
  classical theorem of Minkowski.  
\end{rem}

\begin{lemma} \label{lem;c_tfker} Let $P$ be a group in which
  congruences separate the torsion. Whenever one chooses $N \normal P$
  characteristic, of finite index, inside a congruence group
  separating the torsion, the quotient map $P\to P/N$ induces a map
  $c:\Out(P) \to \Out(P/N)$ with torsion free kernel.
\end{lemma}

\begin{proof}
  The map $c$ is well defined since $N$ is characteristic. By
  definition of congruences separating the torsion, there exists $N_0$
  characteristic, of finite index such that $c_0: \Out(P) \to
  \Out(P/N_0)$ has torsion free kernel. If $N<N_0$ is characteristic
  in $G$, the map $c_0: \Out(P) \to \Out(P/N_0)$ factorizes through
  $\Out(P) \stackrel{c}{\to} \Out(P/N) \to \Out(P/N_0)$. It follows
  that $\ker c \subset \ker c_0$ is torsion-free.
\end{proof}

\subsubsection{Injectivity of  $\pi$: the diagram.}

Our aim is to obtain a well defined commutative diagram 

\begin{equation}\label{diag;injectivity}
\begin{array}{ccc}
\Out(G , \calP_{uo})/\Out(G, \calP_m)  &  \stackrel{\pi}{\longrightarrow} &
\Out(\bar G,  \bar \calP_{uo})/\Out(\bar G, \bar \calP_m)    \\
 ? \downarrow & &   \downarrow ?  \\
\prod \Out(P_i) & \stackrel{\oplus c_i}{\longrightarrow}  &  \prod \Out (\bar P_i). 
\end{array}
\end{equation}
in which the vertical arrows are injective. 

The reader can already see how the diagram will help to prove that
$\pi$ is a bijection, under suitable assumptions. Indeed, we
established the surjectivity of $\pi$ (in deep enough Dehn fillings),
and the torsion-freeness of the kernel of $\oplus c_i$.  If the left
vertical arrow is injective, and $\Out(G,\calP_{uo})$ finite, the
bijectivity of $\pi$ will follow.

\begin{prop}\label{prop;def;r}

  Let $G$ be a group, and $P<G$. Assume that $P$ is its own normalizer
  in $G$. Then, the inclusion $P\to G$ induces a homomorphism $r: \Out (G,
  ([P])) \to \Out P$. Moreover, the kernel of this homomorphism is
  precisely $\Out (G, ([S]))$ for some (any) tuple $S$ generating $P$.

\end{prop}

\begin{proof}
  Take any automorphism $\alpha$ of $G$ preserving the conjugacy class
  of $P$.  There is $g$ such that, if $\ad_g$ denotes the conjugation
  by $g$, the automorphism $\ad_g\circ\alpha$ preserves $P$. There is
  no uniqueness of $g$, but for any other choice $g'$, the element
  $g^{-1} g'$ normalizes $P$, hence is in $P$ by assumption. It
  follows that the automorphisms $\ad_g\circ\alpha$ and
  $\ad_{g'}\circ\alpha$ coincide in $\Out(P)$. It is then
  straightforward to check that this map $\Aut(G,[P]) \to \Out (P)$ is a
  homomorphism.

  The kernel of this homomorphism is precisely the subgroup of
  $\Aut(G,[P])$ that induces some conjugation by an element of $G$ on
  $P$. The claims follow.
\end{proof}

The left vertical arrow of Diagram (\ref{diag;injectivity}) is
therefore well defined and given by:

\begin{coro}\label{coro;left_vertical}
  If $(G,\calP_{uo})$ is a relatively hyperbolic group, with
  $\calP_{uo}=([P_1], \dots, [P_k])$ and each $P_i$ is infinite, then for all $i$, the
  restriction induces a homomorphism $r_i:\Out(G,\calP_{uo}) \to
  \Out(P_i)$, such that $\ker( \oplus r_i)$ is $\Out(G,\calP_m)$. In
  particular, the restriction induces an injective homomorphism \[
  \Out(G , \calP_{uo})/\Out(G, \calP_m) \to \prod_i \Out(P_i)\]
\end{coro}

\begin{proof}
  By Proposition \ref{prop;parab_own_norm}, in a relatively hyperbolic
  group, peripheral subgroups are their own normalizers, and
  Proposition \ref{prop;def;r} can be applied.
\end{proof}

\subsubsection{Injectivity of $\pi$: the last arrow of the diagram.}

The last brick of the proof is to obtain the right vertical arrow of
Diagram (\ref{diag;injectivity}), and for that, again by Propositions
\ref{prop;def;r}, it suffices to show that, for sufficiently deep Dehn
fillings, $P_i/N_i$ is its own normalizer in $\bar G$.

For that, we need to examine the geometry of the Dehn fillings more
closely. Let us introduce the group $K_n$ that is the kernel of the
Dehn filling by the groups $N_i(n)$ that are intersections of all
index $\leq n$ subgroups of $P_i$. If the groups $P_i$ are residually
finite, the Dehn fillings by the groups $N_i(n)$ satisfy the
assumption of Theorem \ref{theo;Dehn_fill} for $n$ large enough, and
in particular, $P_i\cap K_n = N_i(n)$. We sometimes write $\{
[P_1/N_1(n)],\dots,[P_k/N_k(n)]\} = \calP/\calN(n)$ for brevity.

We will also need to manipulate a good space on which $G/K_n$ acts.
In order to comply with some conventions used in some constructions in
\cite{DGO}, we introduce $\delta_c$ a certain positive constant, whose
value can be found in \cite[\S 5]{DGO}, but which is actually
irrelevant for us.  Note that, in the definition of relative
hyperbolicity, one can, up to rescaling, assume that the space $X$ is
$\delta_c$-hyperbolic.

Thus, start with $\delta \leq \delta_c$, and with a
$\delta$-hyperbolic proper space $X$ as in Definition
\ref{def;RHG},
and endow it with a $10^5\delta$-separated
invariant family of horoballs $\mathcal{H}$.
Given
$R$, we construct the parabolic cone-off space $\dot X_R$ as in
\cite[\S 7.1, Definition 7.2]{DGO}, which is as follows.

First choose $\mathcal{H}_R$ an invariant system of $10^5\delta
+R$-separated horoballs, obtained by taking sub-horoballs of those in
$\mathcal{H}$.  For each horoball $H$ in $\mathcal{H}_R$, let
$\partial H$ be its horosphere, $\partial H = H\setminus
\mathring{H}$. Fix $r_0\geq r_U$, where $r_U$ is a constant whose
value can be found in \cite[\S 5.3]{DGO} (it is $3\times 10^9 +1$.)
Let $\dot{X}$ be the cone-off of $X$ along each $H\in\mathcal{H}_R$,
defined in \cite[\S 5.3]{DGO}, for the parameter $r_0$, and let
$\mathcal{A}$ be its set of apices.  For $a \in \calA $, let $H_a$ its
corresponding horoball in $X$, and $\dot H_a$, the cone on $H_a$, is a
subset of $\dot{X}$. For each pair of points $p,q$ in $\partial H_a$
and each geodesic $[p,q]$ of $\dot H_a$ (for the intrinsic metric)
avoiding $a$, 
consider the set $T_{[p,q]}\subset \dot{H}_a$ to be the convex hull of
$[a,p]\cup [p,q]\cup [q,a]$ in $\dot
H_a$.    
For every other pair of points, i.e. for any connecting geodesic $a\in
[p,q]$ , then $T_{[p,q]}= [p,q]$. We define $B_{H_a}$ to be the union
of all $T_{[p,q]}$, for $p,q$ in $\partial H_a$ and $[p,q]$ a
geodesic.

The parabolic cone-off, as defined in \cite[Definition 7.2]{DGO}, is
then $$\dot X_R = \left( \dot X \setminus \bigcup_{a\in \calA} \dot
  H_a \right) \cup \left( \bigcup_{a\in \calA}B_{H_a} \right)$$

The space $\dot X_R$ is geodesic and locally compact everywhere except
at the apices (see the comment before \cite[Definition 7.2]{DGO}.)  Also, by
\cite[Lemma 7.4]{DGO}, the space $\dot X_R$ is
$16\delta_u$-hyperbolic, for some universal constant $\delta_u$.

Fix $R>0$. Then, by \cite[Proposition 7.7]{DGO}, for $n$ large enough,
$K_n$ is the group of a so-called separated very rotating family on
$\dot X_R$, and by \cite[Proposition 5.28]{DGO} the quotient $\overline{X_n}$ of the space  $\dot X_R$
by  the action of $K_n$ is $60000\delta$-hyperbolic, and the action of $G/K_n$  on  $\overline{X_n}$ is
metrically proper, co-bounded.  Recall that $\mathcal{A}$ is the set of apices of $\dot X_R$, and $\overline{\calA}$ is its image in the quotient $\overline{X_n}$. 

\begin{lemma} \label{lem;barfree} 
   For $n\geq 1$, the action of $G/K_n$ on $\overline{X_n} \setminus
  \overline{\calA}$ is free.     
\end{lemma}

\begin{proof}
  Assume that $\bar g$ fixes a point $\bar x \in \overline{X_n}$ that is
  not an apex, and choose a preimage $x$ of $\bar x$. Then, $\bar g$
  has a preimage $g$ that fixes $x \in \dot X_R$, hence also the
  (unique) projection of $x$ on $X$.  Since the action on
  $X$ is free, $g$  is trivial. 
\end{proof}

\begin{lemma} \label{lem;geom_Xbar} There exists an integer $n_0$ such
  that, for all ball $B \subset X\setminus \calH$, of radius
  $1000\delta$, and any orbit $Gx$, the set $Gx\cap B$ has at most
  $n_0$ elements. Moreover, for any $n\geq 1$, and for any ball $\bar
  B \subset (X\setminus \calH)/K_n \subset \overline{X_n}$ of radius
  $1000\delta$, and any orbit $(G/K_n)\bar x$, the set $(G/K_n)\bar x
  \cap \bar B$ has at most $n_0$ elements.
\end{lemma}

\begin{proof}
  The first claim follows from the fact that the action of $G$ on
  $X\setminus \calH$ is metrically proper.  For the second claim, one
  can lift $\bar B$ as a ball $B$ in $X\setminus \calH$, and the full
  preimage of the orbit $(G/K_n)\bar x$ is itself an orbit. The
  estimation follows.
\end{proof}

\begin{prop}\label{prop;barPi_own_N}
  Let $(G, \{[P_1], \dots, [P_k]\})$ be a    
  relatively hyperbolic group.  Assume that each $P_i$ is infinite,
  residually finite. For all $n$, let $K_n$ be the $n$-th Dehn kernel
  as above. Then, for $n$ large enough, the groups $\bar P_i =
  P_i/(K_n\cap P_i)$ are their own normalizers in $G/K_n$.  Moreover,
  there is an algorithm that, given $n$, will terminate if and only if
  $G/K_n$ is hyperbolic and the groups $\bar P_i$ are their own
  normalizers.
\end{prop}

\begin{proof}
  By definition of the parabolic subgroups, the fixator of the apex
  $a_i \in \dot{X}_R$ associated to the horoball preserved by $P_i$ is
  precisely $P_i$ in $G$. Therefore the fixator of $\bar a_i$ is
  exactly $\bar P_i = P_i/(K_n\cap P_i)$. The first step is to show
  the other inclusion. This next lemma implies that $\bar P_i$ is its own normalizer, since the
  normalizer preserves the set of fixed points.

  \begin{lemma} $\bar P_i$ fixes only one point in $\overline{X_n}$. 
  \end{lemma}
  \begin{proof}
    Let $\bar p$ be non trivial in $\bar P_i$, of maximal order.  We
    claim that \emph{if $n$ is large enough, $\bar p$ is of order greater
    than $n_0$, the constant given by Lemma \ref{lem;geom_Xbar}.}

    The proof of the claim is easier if $P_i$ has an infinite order
    element $p$: by residual finiteness of $P_i$, the $n_0$-th power
    of $p$ in $P_i$ survives in a certain finite quotient $\bar P_i$,
    forcing the maximal order of elements of $\bar P_i$ to be larger
    than $n_0$.  If $P_i$ has no infinite order element, since $P_i$
    is finitely generated, infinite, residually finite, the solution
    to the restricted Burnside problem states that there is no upper
    bound on the order of its elements.  A similar argument then
    concludes; \emph{the claim is established.}

    Assume that $\bar p$ fixes $\bar x \in \overline{X_n}$, different from
    $\bar a_i$. Recall (Lemma \ref{lem;barfree}) that $G/K_n$ acts
    freely on $\overline{X_n} \setminus \overline{\calA}$, and therefore, $\bar x$
    has to be the image of an apex.  We may choose $\bar x$ so that
    $d(\bar x, \bar a_i)$ does not exceed $\inf_{x'} d(\bar x' , \bar
    a_i)+R/100$, for $x'$ ranging over the set of points fixed by
    $\bar p$.

    Consider a geodesic segment $[\bar a_i, \bar x]$. Let us show that
    it contains no apices.  If $\bar z$ was such an apex, then $\bar p
    \bar z$ would be another image of an apex on the geodesic $\bar p[
    \bar a_i, \bar x]$, which is another geodesic joining $[\bar a_i,
    \bar x]$. Using that $\overline{X_n}$ is $200\delta$-hyperbolic, one
    easily obtains that $d_{\overline{X_n}} (\bar z, \bar p \bar z) \leq
    500\delta$, but image of apices are $R/2$-separated, which implies
    that $\bar z= \bar p \bar z$. Moreover, being on the geodesic,
    $\bar z$ is at least $R/2$-closer to $\bar a_i$ than $\bar x$.
    This contradicts the minimality of $d_{\bar X_n} (\bar a_i, \bar
    x)$ among the fixed points of $\bar p$. Thus, there is no other
    image of apex in $[\bar a_i, \bar x]$.

    Now we may lift $[\bar a_i, \bar x]$ as a connected path in
    $\dot{X}_R$ starting from $a_i$ (there is a choice of first edge,
    in an orbit under $K_n \cap P_i$, but once this first edge is
    chosen, the lift is determined, since one never pass through an
    apex, the only points whose stabilizers non-trivially intersect
    $K_n$.)  Let $x$ the end point of this path $[a_i, x]$, which is
    necessarily a geodesic, since any shorter path gives a shorter
    path in $\overline{X_n}$.

    By separation of the horoballs in $X$, there is, on the geodesic
    $[a_i,x]$, a point $y$ that is in the original space $X \setminus
    \calH$.  Let $\bar y$ be its image in $[\bar a_i, \bar x]$. Since
    for all $k$, $\bar p^k$ fixes both $\bar a_i$ and $\bar x$, it has
    to move $\bar y$ at distance at most $5 \bar \delta$ from itself,
    where $\bar \delta$ is the hyperbolicity constant of $\overline{X_n}$. However, $\bar\delta \leq 1000\delta$ where $\delta$ is the
    hyperbolicity constant of $X$. Therefore, the sequence $\bar p^k
    \bar y$ stays in the ball centred at $\bar y$ of radius
    $1000\delta$, which, by the local injectivity property of the
    quotient map, is isometric to $B(y, 1000\delta) \subset X$, and in
    which the injectivity radius of $G/K_n$ is the same as $G$. It
    follows that, by definition of $n_0$, there is some power $k \leq
    n_0$ such that $\bar p^k \bar y = \bar y$. However, $\bar y$ is
    not the image of an apex, and by Lemma \ref{lem;barfree}, $\bar
    p^k=1$, thus contradicting that the order of $\bar p$ was greater
    than $n_0$.

  \end{proof}
  The lemma proves the first assertion of the proposition.  For the
  final assertion of the proposition, let us invoke Papasoglu's
  algorithm \cite{Pap} that will terminate if $G/K_n$ is hyperbolic,
  and if so, \cite[Lemma 2.5]{DG_gafa} will compute generating sets of
  the normalizers of each $\bar P_i$ (which are easily checked to
  belong or not to belong to the finite group $\bar P_i$.)
\end{proof}

\begin{prop}\label{prop:pi_injective}

  Assume that $\Out(G,\calP_{uo})/\Out(G,\calP_{m})$ is finite, and
  that groups in $\calP_{uo}$ have congruences separating the
  torsion. Consider a Dehn filling of $G$ by $K=\langle\langle
  \bigcup_i N_i\rangle\rangle$.  Assume $\bar G= G/K$ is hyperbolic,
  and $P_i/N_i$ is its own normalizer in $G/K$.  If all $N_i$ are
  characteristic, of finite index, and in congruence subgroups of
  $P_i$ separating torsion, then the homomorphism $\pi:
  \Out(G,\calP_{uo})/\Out(G,\calP_{m}) \to \Out(\bar G,\bar
  \calP_{uo})/\Out(\bar G, \bar\calP_{m})$ induced by the quotient map
  is an injective.
\end{prop}

\begin{proof}
  By Lemma \ref{lem;pi_well_definied}, the homomorphism $\pi$ is well
  defined. By Lemma \ref{lem;surjective} it is surjective.  By
  Corollary \ref{coro;left_vertical}, the left vertical arrow of
  Diagram (\ref{diag;injectivity}), say $\rho$, is injective, and
  induced by restriction.  Since
  $\Out(G,\calP_{uo})/\Out(G,\calP_{m})$ is finite, and $\ker \oplus
  c_i$ is torsion free (by Lemma \ref{lem;c_tfker}), the composition
  $\oplus c_i \circ \rho$ is injective.  The right vertical arrow is
  well defined, by Proposition \ref{prop;barPi_own_N} and
  \ref{prop;def;r}. The diagram is clearly commutative, hence it
  follows that $\pi$ is injective.
\end{proof}

This proves our structural feature, Proposition \ref{prop;intro}.

\begin{coro}\label{coro;struct_feat}
  Assume that $\Out(G,\calP_{uo})/\Out(G,\calP_{m})$ is finite, and
  that groups in $\calP_{uo}$ are residually finite, and have
  congruences separating the torsion.  Then there is a Dehn filling by
  groups $N_i$ characteristic in $P_i$, of finite index, such that the
  map $\pi$ is an isomorphism.
\end{coro}

\begin{proof}
  Lemma \ref{lem;surjective} guarantees that deep enough Dehn fillings
  will make $\pi$ surjective. On the other hand, Proposition
  \ref{prop;barPi_own_N} ensures that for deep enough Dehn fillings,
  Proposition \ref{prop:pi_injective} can be applied, and this
  ensures the injectivity.
\end{proof}

\subsubsection{Proof of  Proposition \ref{prop;orbit_computation}}\label{sec:orb-comp}

We can now prove Proposition \ref{prop;orbit_computation}.

\begin{proof}[Proof of Proposition \ref{prop;orbit_computation}]

  On the one hand, the algorithm searches for classes of automorphisms
  in $\Out(G, \calP_{uo})$ modulo $\Out(G, \calP_{m})$, and enumerates
  them in a list $\calL$.  On the other hand, the algorithm computes a
  number $n_0$ so that congruences by characteristic subgroups of
  index larger than $n_0$ separate the torsion, in the parabolic
  subgroups (this is possible thanks to the assumption of effective
  separation of torsion by congruences.) Then, in parallel, for
  incrementing $n\geq n_0$, the algorithm tries to certify that $n$ is
  suitable for applying Proposition \ref{prop:pi_injective} (this is
  done by the last point of Proposition \ref{prop;barPi_own_N}.) It
  may happen that for some $n$, the certification procedure does not
  terminate, but we know by the first point of Proposition
  \ref{prop;barPi_own_N} that it will terminate if $n$ is large enough
  (hence the reason why we perform this certification in incrementing
  parallel.)  For each $n$ which is certified suitable for the
  conclusion of Proposition \ref{prop:pi_injective}, the algorithm
  computes $(G/K_n, \calP_{uo}/\calP(n))$ (which is thus certified
  hyperbolic) and also computes its outer automorphism group by
  \cite[Theorem 8.1]{DG_gafa}

  Proposition \ref{prop:pi_injective} guarantees that for all
  certified $n$, and at every state of the enumeration of the list
  $\calL$ (in particular when it is complete), $\calL$ naturally
  embeds (injectively) in
  \[\Out(G/K_n, \calP_{uo}/\calP(n))/\Out(G/K_n,
    \calP_{m}/\calP(n)). \] Lemma \ref{lem;surjective}, on the other
  hand, guarantees that, for some $n$, this embedding is a
  bijection. When that happens, we are done: the list covers
  representatives of the cosets
  \[\Out(G/, \calP_{uo})/\Out(G/, \calP_{m}).\]
\end{proof}

\section{Solving the isomorphism problem}\label{sec;Solving}

In this section we will provide a proof of Theorem \ref{theo;main}. We
will start with the one-ended case
\subsection{The proof of the one-ended case of Theorem
  \ref{theo;main}}\label{sec:proof-main-res}

\begin{thm}[Main result, one-ended case]\label{thm:main-one-ended}

  Let $\calC$ be a heriditarily algorithmically tractable class, with
  algorithmically bounded torsion, satisfying the following
  properties:
  \begin{itemize}
  \item all groups in $\calC$ are residually finite,
  \item the isomorphism problem is solvable in $\calC$, 
  \item in $\calC$, congruences effectively separate the
    torsion (see Definition \ref{def;CST}),
  \item the mixed Whitehead problem is  solvable in $\calC$
    (see Definition \ref{def:mwhp}.)
  \end{itemize}
  There is an algorithm which decides if two explicitly given
  virtually torsion-free relatively hyperbolic groups $(G_1,\calP_1)$,
  $(G_2,\calP_2)$  one-ended, and 
  whose peripheral subgroups belongs to $\calC$, are isomorphic as
  groups with unmarked peripheral structure. 

\end{thm}

\begin{proof}

Given two such 
one-ended relatively hyperbolic groups $(G_1,\calP_1),
(G_2,\calP_2)$,
with a finite generating set for each representative of maximal
parabolic subgroup, we need to check whether they are isomorphic as
groups with unmarked peripheral structure.

By Corollary \ref{cor:ctf-compute-can-splittings} we can compute the
canonical JSJ decompositions $\bbX_i$ for $(G_i,\calP_i)$, with
maximal elementary black vertex groups, and white vertex groups are
either rigid or $\QH$.

By criterion (\ref{carac_toc3}) of Lemma \ref{lem;carac_toc}, $G_i$ is
hyperbolic relative to the augmented peripheral structure
$(\calP_i)_{\bbX_i}$ as well, and $\bbX_i$ is essentially peripheral
for this structure.To avoid handling too many peripheral structures at
different places, it is convenient to turn this splitting into a
peripheral one, and we do that canonically as follows.  On each white
vertex $v$ of $\bbX_i$, we may compute the induced peripheral
structure on its group $\G_v$ as follows: we know the images of the
adjacent edge groups, and for each representative of group in
$\calP_i$, by enumeration, we may find a vertex containing one of its
conjugates; thus we may find the list of
peripheral subgroups that have a conjugate in $\G_v$. According to Lemma
\ref{lem;describe_induced}, this is the induced peripheral structure.

For each conjugacy class of peripheral subgroups $P<\G_v$ that are not
yet conjugated to an edge group, we may refine $\bbX_i$ by appending
an inessential black vertex, with an edge to this white vertex $v$, both whose
groups are $P$ with the natural attaching maps.  This produces
\emph{peripheral} splittings $\bbX'_i$ of $(\calP_i)_{\bbX_i}$. By
Proposition \ref{prop:jsj-charac-1} (\ref{it:auto-invariant}) this
splitting is canonical; therefore $(G_1,\calP_1)$ and $ (G_2,\calP_2)$
are isomorphic if and only if $\bbX'_1$ and $\bbX'_2$ are isomorphic
as graphs-of-groups.

We may list all the isomorphisms of the underlying graphs of $\bbX'_1$
and $\bbX'_2$, and the possible orderings of the unmarked structures
$\calP_1$ and $\calP_2$, and proceed in parallel for each choice of
them. This way, we reduce to the case of two structures of graphs of
groups on the same underlying graph.

According to Proposition \ref{prop;main_reduction}, in order to check
whether these decompositions are isomorphic it is enough to solve, in
black (elementary) vertex groups, the mixed Whitehead problem, and the
isomorphism problem, and, in the white vertex groups, the isomorphism
problem (of groups with unmarked ordered peripheral structures) and
the orbit problem for markings. The assumptions on the parabolic
groups readily guarantee that the two first problems (on black vertex
groups) are solvable.

For white vertex groups, let us differentiate whether the groups to
compare are rigid or $\QH$ non rigid (whether they belong to a class
or the other can be decided by Theorem \ref{theo;TouC}, but actually,
at this stage, we know which vertex groups are rigid.) Let us consider
first rigid groups. By Bowditch's Theorem \ref{thm;induced-struct},
since $\bbX'_i$ is peripheral, these rigid vertex groups are
hyperbolic relative to the peripheral structure induced by
$(\calP_i)_{\bbX_i}$. To show that (\ref{assumeIP}) in Proposition
\ref{prop;main_reduction} (the isomorphism problem) is satisfied, we
use one of the main result of \cite{DG_charac}.

\begin{thm} \cite[Theorem 6.1]{DG_charac} \label{thm;iso-pb-rigid}
  There is an algorithm which decides if two explicitly given finitely
  presented relatively hyperbolic groups $(G,\calP)$, $(H,\calQ)$
  (given by
  finite presentations, and finite generating sets of conjugacy
  representatives of the peripheral structure) are isomorphic (as
  groups with unmarked peripheral structure), provided $\calP$,
  $\calQ$ consist of infinite residually finite groups and that $(G,\calP)$,
  $(H,\calQ)$ are rigid, non elementary.
\end{thm}

In \cite{DG_charac}, the algorithm is given explicitly. As for the
requirement (\ref{orbit1b}) of Proposition \ref{prop;main_reduction}
(the orbit problem for markings), the assumption on the congruences
separating torsion allows us to use Proposition
\ref{prop;orbit_computation}, which solves this algorithmic problem.

Let now us consider $\QH$  white vertex groups. In the torsion-free
case, deciding isomorphism is
a matter of identifying the surface,
and the orbit problems for
markings is easy, since the peripheral groups are infinite cyclic; two
different markings are either equal, or opposite in each boundary
components (in which case, they are in the same orbit under the
automorphism group), or they differ on certain but not all boundary
components, in which case they are in the same orbit if and only if
the surface is not orientable. In the non torsion-free case, we may
use the main result of \cite{DG_gafa}, since these groups are
hyperbolic, with peripheral structure consisting of virtually abelian
groups (which have finite computable automorphism groups.) One can
make a list of possible markings of the peripheral subgroups (in a
given orbit under automorphisms) and check by \cite{DG_gafa} which
configurations are reached by an automorphism of the ambient $\QH$ 
group.  

Thus, one may apply Proposition \ref{prop;main_reduction}, and decide
whether $\bbX'_1$ and $\bbX'_2$ are isomorphic graphs of groups. As
already discussed, this allows to decide whether $(G_1,\calP_1)$ and
$(G_2,\calP_2)$ are isomorphic.
\end{proof}

\subsection{The many ended case of the proof of Theorem \ref{theo;main}}

Given two Dunwoody-Stallings decompositions with isomorphic underlying
graphs, we can use Theorem \ref{thm:main-one-ended} to decide if
corresponding vertex groups are isomorphic. We now want to use
Proposition \ref{prop;compatible_twist} to reduce the isomorphism
problem to finding a suitable extension adjustment. Although the
fundamental idea is in the same vein as the one-ended case, there is
one important distinction: unlike parabolic groups, finite subgroups
may have strictly larger normalizers. We therefore start by examining
the normalizers of non parabolic finite subgroups of relatively
hyperbolic groups.

\subsubsection{On normalizers and conjugators of finite subgroups}

Throughout this section assume that $(G,\calP)$ is relatively
hyperbolic where the groups in $\calP$ are finitely presented and in
which we can solve the word problem. Let $X$ be the Groves-Manning
space with combinatorial horoballs. By Lemma \ref{lem:make-balls} we
can algorithmically determine a Gromov hyperbolicity constant
$\delta$, construct arbitrarily large balls in $X$ and compute the
partial action of $G$ on these balls.

\begin{lem}\label{lem:parabolic-intersection}
  Let $F \leq G$ be the intersection of distinct parabolic
  subgroups. Then it is virtually self-normalized
  i.e. $[N_G(F):F]<\infty$ and we can compute a generating set of
  $N_G(F)$. Furthermore given two finite subgroups $F,F'$ that arise
  as intersections of parabolic subgroups, we can determine whether
  they are conjugate.
\end{lem}
\begin{proof}
  Let $F = \bigcap_{i=1}^m P_i$, where $P_i \in \calP$ (here assume
  $\calP$ is closed under conjugation.) Then there is a collection of
  horoballs $\calH_i \subset X$ such that the set of vertices
  $I=V\left(\bigcap_{i=1}^m \calH_i\right) \neq \emptyset$. The group
  $F$ acts on $I$, which is a finite set. The normalizer of $F$ is the
  set of elements of $g\in G$ such that $g\cdot I = I$, which can be
  computed by constructing a ball in $X$ that contains $I$ and the
  identity.

  To prove the second claim let $I$ be the set of vertices constructed
  above and let $I'$ be the analogous set of vertices constructed for
  $F'$. The finite subgroups are conjugate if and only if $I'$ can be
  translated to $I$ by an element of $G$. This can be seen by
  constructing ball in $X$ containing $I,I'$ and the identity. 
\end{proof}

%
%
To examine non-parabolic subgroups we need the following lemma which
is classical,  for instance a variation of \cite[Lemma
III-$\Gamma$-3.3]{Bridson_Haefliger}, or for the proposed formulation, see
the elementary argument in the proof of the main theorem of
\cite{bogopolskii1996finite}.

\begin{lem}\label{lem:bg-almost-fixed-point}
  Let $Z$ be a subset of a $\delta$-hyperbolic space that consists of
  the orbit of a single point under the action of some finite subgroup
  $F$. Let $\beta$ be a geodesic between two furthest points in $Z$
  and let $m_\beta$ denote its midpoint. Then the diameter of the
  orbit $F\cdot m_\beta$ is at most $3\delta$.
\end{lem}

\begin{lem}\label{lem:md-qc} 
  Let $F\leq G$ be a finite non-parabolic subgroup. Then the following
  hold.
  \begin{enumerate}
  \item\label{it:afq-qc} The set of \emph{$F$-almost fixed points} (see
    \cite{liang2013centralizers})
  \[ X_F = \{x \in X \mid \diam(F\cdot x) \leq 3\delta+1\}
  \] is $3\delta+1$-quasiconvex in $X$.
  \item\label{it:afq-horoball} $X_F$ is contained in the complement of
    the horoballs of depth $3\delta$.
  \end{enumerate}
\end{lem} 

\begin{proof}     
  Let $x,y$ be in $X_F$ and suppose they are at distance more than
  $3\delta+1$ from one another. Let $\alpha:[0,t]\to X$ be a geodesic
  connecting $\alpha(0)=x$ and $\alpha(t)=y$ and denote by $m_\alpha$
  the midpoint of $\alpha$. For any $g \in F$ we have by hypothesis
  \[d\left(\alpha(0),g\cdot\alpha(0)\right) ,
    d\left(\alpha(t),g\cdot\alpha(t)\right)\leq 3\delta+1.\] Now it is
  an easy consequence of $\delta$-thinness and the properties of
  geodesic quadrilaterals that
  $d(m_\alpha,g\cdot m_\alpha)\leq 5\delta+1$.  Let
  $Z = F\cdot m_\alpha$. There is some $g_m \in F$ such that the
  geodesic $\gamma_m=[m_\alpha,g_m\cdot m_\alpha]$ realizes the
  maximal distance between two points in $Z$. By Lemma
  \ref{lem:bg-almost-fixed-point} the midpoint of $\gamma_m$ is an
  $F$-almost fixed point; thus there is an almost fixed point a
  distance at most $\frac{5\delta+1}2\leq 3\delta$ from
  $m_\alpha$. The closest vertex is therefore a distance at most
  $3\delta+1$. This proves (\ref{it:afq-qc}.)

  To see (\ref{it:afq-horoball}), note that $G$ maps horoballs to
  horoballs of the same depth. If $x \in X_F$ was at depth more than
  $3\delta$ inside a horoball $H$ then every $fx, f \in F$ must
  still lie in $H$ since all other depth $3\delta$ horoballs must be
  at distance at least $6\delta > 3\delta+1$ from $x$. It follows that
  $F\cdot H = H$ which implies that $F$ is parabolic; contradiction. 
\end{proof}

\begin{lem}\label{lem:orbits-touch}
  Let $F$ be a finite non-parabolic group, let $X$ be the Groves-Manning space, and
  let $Z = F\cdot x$ be an orbit. Suppose there is some translate such
  that \[ k\cdot Z \cap Z \neq \emptyset.
  \]
  Then $k \in F$.
\end{lem}
\begin{proof}
  This follows immediately from the fact that the action on $X$ is
  free. Let $y \in Z \cap k\cdot Z$. Then $k^\mo\cdot y = h \cdot y$
  for some $h \in F$, since $Z$ is an $F$-orbit. It follows that $kh
  \cdot y = y \Rightarrow kh=1$ by freeness of the action and the
  result follows.
\end{proof}

\begin{lem}\label{lem:small-sets}
  Let $(G,\calP)$, $X$ and $\delta$ be as the statement of Lemma
  \ref{lem:make-balls}. Then it is possible to construct a
  representative with respect to translations by elements of $G$ of
  every set of vertices $Z$ with $N$ elements of diameter at most
  $3\delta+1$ lying in the complement of the depth $3\delta+1$
  horoballs.
\end{lem}
\begin{proof}
  We first note that any such set can be translated under the action
  of $G$ into a ball $B$ of radius $6\delta+2$ about the vertex $x_1$
  representing the identity in $X$. By Lemma \ref{lem:make-balls} we
  can construct this ball and take all sets of cardinality $N$ with
  diameter at most $3\delta+1$.
\end{proof}

\begin{defn}[$N$-bounded and congruent]\label{defn:bded-cong}
  A set such as appears in the statement of Lemma \ref{lem:small-sets}
  is called \define{$N$-bounded}. If $Z,Z'$ are subsets of $X$ and
  there is some $g\in G$ such that $g\cdot Z = Z'$, then the sets are
  called \define{$G$-congruent}. We denote by $K_N$ the number of
  $G$-congruence classes of $N$-bounded sets.
\end{defn}

\begin{defn}
  Let $F \leq G$ be a finite non-parabolic-subgroup and let $x \in X$. The
  $G$-congruence class of the orbit $F\cdot x$ is called the
  \define{$F$-type} of $x$ and is denoted $\tp F x$.
\end{defn}

\begin{lem}\label{lem:afp-orb}
  Let $x,y \in X_F$ be almost fixed points and suppose that
  $\tp F x = \tp F y$. Then there is $k \in N_G(F)$ such that
  $k\cdot x = y$.
\end{lem}
\begin{proof} 
  By hypothesis there is some $k \in G$ such that
  $k\cdot \left(F \cdot x\right) = F\cdot y$. It follows that
  $kFk^{-1}$ acts transitively on $F\cdot y$, which by Lemma
  \ref{lem:orbits-touch} implies that $kFk^\mo = F$.
\end{proof}

\begin{cor}\label{cor:normalizer-action}
  The normalizer $N_G(F)$ acts on $X_F$ the set of almost fixed points
  and the orbits of this action are in bijective correspondence with
  the $F$-types of points in $X_F$.
\end{cor}

\begin{lem}\label{lem:close-types}
  For any element $x_F$ of $X_F$ and any $F$-type realized by an
  element of $X_F$ there is an element $y \in X_F$ at distance less
  than $27K_{|F|}\delta$ from $x_F$ realizing that $F$-type.
\end{lem}
\begin{proof}
  Let $\alpha:[0,t] \to X$ be the geodesic from $\alpha(0)=x_F$ to
  $\alpha(t)=y$ which is closest to $x_F$ of the desired type. Suppose
  towards a contradiction that it was longer than
  $27K_{|F|}\delta$. Subdivide $\alpha$ into non-overlapping segments
  $\alpha_j$ of length $9\delta$. By Lemma \ref{lem:md-qc} for each
  $\alpha_j$ there is a point $z_j \in X_F$ at distance at most
  $3\delta+1$ from $\alpha_j$. By hypothesis there must be segments
  $\alpha_1,\alpha_2,\alpha_3$ with increasing distance from $x_F$,
  and the $3\delta+1$-close $F$-almost fixed points
  $z_{j_1}, z_{j_2}, z_{j_3}$ have the same $F$-type. It follows that
  there is some $h \in N_G(F)$ such that $h\cdot z_{j_3} =
  z_{j_1}$. Let $\beta$ be the geodesic from $z_{j_3}$ to $y$ and let
  $\gamma$ be the concatenation of a geodesic from $x_F$ to $z_{j_1}$
  and $h\cdot\beta$. The endpoints of $\beta$ are in $X_F$ and the
  translate $h\cdot y$ is by definition of the same $F$-type as
  $y$. Now $|\gamma| \leq |\alpha|-(9\delta-2(3\delta+1))<|\alpha|$
  contradicting the minimality of the length of $\alpha$. The result
  follows.
\end{proof}

\begin{prop}\label{prop:compute-normalizer} 
  Given a finite non-parabolic group $F\leq G$ we can find a
  generating set of its normalizer $N_G(F)$. In particular it is a set
  of element that move an almost fixed point $x_F$ of $F$ a distance
  at most $54K_{|F|}\delta$.
\end{prop}
\begin{proof}
  Fix $x_F$ an almost fixed point of $F$. Since $F$ is finite non-parabolic $x_F$
  is a depth at most $3\delta$ inside a horoball. 
  It now follows from Lemma \ref{lem:close-types} and the standard
  argument for quasiconvex subgroups of hyperbolic groups that
  $N_G(H)$ is generated by a collection of elements that translate
  $x_F$ a distance of at most $54K_{|F|}\delta$. The result now
  follows from Lemma \ref{lem:make-balls}. 
\end{proof}

This next result resembles work of Bumagin, however here we deal with
finite subgroups instead of individual elements.

\begin{lem}[{c.f. \cite[Theorem 1.1]{bumagin2015time}}]\label{lem:short-conj-length}
  Let $X$ be a Groves-Manning space for $(G,\calP)$, let it be
  $\delta$-hyperbolic. Let $K=K_{N}$ be as in Definition
  \ref{defn:bded-cong} where $N$ is the maximal cardinality of a
  non-parabolic finite subgroup. There are computable constants
  $A=A(K,\delta), B=B(K,\delta)$ that satisfy the following: if
  $F,H   
  $ are conjugate non-parabolic finite subgroups of $G$ given as: 
  \[ F = \{f_1,\ldots,f_r\}, H = \{h_1,\ldots,h_r\} \] and if
  $\phi_i = [1,f_i\cdot 1]$ and $\eta_j = [1,h_j\cdot 1]$ are
  geodesics in $X$ the Groves-Manning space, then we can find a
  conjugator $k$, with $kFk^\mo = H$, such that if
  $\kappa = [1,k\cdot 1]$ is a geodesic
  then \begin{equation}\label{eq:lin-conj-bound} |\kappa| \leq
    A\left(\max_i |\phi_i|+\max_j|\eta_j|\right) +B.
  \end{equation}
  In particular the constants $A$ and $B$ can be computed from a
  sufficiently large ball in $X$ around the identity.
\end{lem}
\begin{proof}
  By Lemma \ref{lem:md-qc} 
  item (\ref{it:afq-horoball}), any finite
  group $F$ has a conjugate $F_0$ that has an almost fixed point $x_0$
  which is a distance at most $3\delta$ directly above (in a horoball)
  the vertex $1$, representing the identity, in $X$.  

  Now let $x_F \in X_F$ be an $F$-almost fixed point in the same
  $G$-orbit as $x_0$ that is as close as possible to $1$ and let
  $\beta$ be the geodesic connecting $1$ to $x_F$ and let $\kappa$ be
  concatenation $\beta\cdot\beta^+$ where $\beta^+$ is the vertical path
  from $x_F$ to a vertex representing a group element $k$ directly
  below $x_F$.  It follows from Lemma \ref{lem:orbits-touch} and the
  freeness of the action of $G$ on $X$ that $k^\mo F k=F_0$.
  
  \emph{Claim: $\kappa$ is not much longer than $\max_{i}|\phi_i|$.}
  We have already seen that $|\kappa| \leq |\beta|+3\delta$. For each
  $f_i \in F$ take the geodesic $\alpha_i = [x_F,f_i\cdot x_F]$. By
  definition of $X_F$ we have $|\alpha_i| \leq 3\delta+1$. Consider
  the quadrilaterals, viewed as sets,
  \[Q_i=\beta \cup \alpha_i \cup \left(f_i\cdot\beta\right)\cup
    \phi_i.\] To prove our claim we need to show that $|\beta|$ is not
  much more than the longest $|\phi_i|$. Suppose towards a
  contradiction that this was not the case, i.e. that $|\beta|$ is
  much longer than $|\phi_i|$ for all $i$. Then in this case all
  quadrilaterals $Q_i$ must have a long thin $\beta$-pinching
  component. Specifically the approximating tree (see \cite[Chapitre 2
  \S2]{ghys1990groupes}) will be of the following form:
  
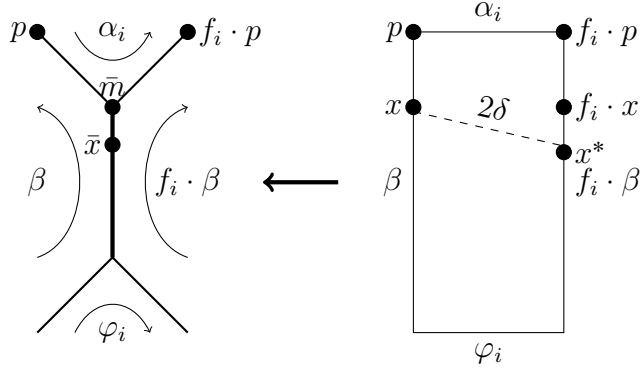
\begin{center}
    \begin{tikzpicture}
      \coordinate (ll) at (0,0);
      \coordinate (lc) at (1,1);
      \coordinate (lr) at (2,0);
      \coordinate (ul) at (0,4);
      \coordinate (uc) at (1,3);
      \coordinate (ur) at (2,4);
      \draw[thick] (ll) -- (lc) -- (lr);
      \draw[thick] (ul) -- (uc) -- (ur);
      \draw[ultra thick] (lc)--(uc);
      \draw[->] ($(ll)+(0,1)$) ..controls ++(0.75,0.25) and ++ (0.75,-0.25)
      .. ($(ul)+(0,-1)$);
      \draw[->] ($(lr)+(0,1)$) ..controls ++(-0.75,0.25) and ++
      (-0.75,-0.25) .. ($(ur)+(0,-1)$);
      \draw[->] ($(ul)+(0.5,0)$) ..controls ++(0.25,-0.5) and ++ (-0.25,-0.5)
      .. ($(ur)+(-0.5,0)$);
      \draw[->] ($(ll)+(0.5,0)$) ..controls ++(0.25,0.5) and ++ (-0.25,0.5)
      .. ($(lr)+(-0.5,0)$);
      \node (k) at ($(ll)!0.5!(lr)$) {$\varphi_i$};
      \node (b) at ($(ll)!0.5!(ul)$) {$\beta$};
      \node (fb) at ($(lr)!0.5!(ur)$) {$f_i\cdot\beta$};
      \node (ai) at ($(ul)!0.5!(ur)$) {$\alpha_i$};
      \draw[fill=black] ($(lc)!0.75!(uc)$) circle (0.1);
      \node[left] (xb) at($(lc)!0.75!(uc)$) {$\bar{x}$};

      \draw[ultra thick,->] (4,2)--(3,2);
      \draw[fill=black] (ul) circle (0.1);
      \draw (ul) node[left]{$p$};
      \draw[fill=black] (ur) circle (0.1);
      \draw (ur) node[right]{$f_i\cdot p$};
      \draw[fill=black] (uc) circle (0.1);
      \draw (uc) node[above]{$\bar m$};
      
      \begin{scope}[xshift=5cm]
        \coordinate (ll) at (0,0);
        \coordinate (lr) at (2,0);
        \coordinate (ul) at (0,4);
        \coordinate (ur) at (2,4);
        \draw
        (ll)--node[below]{$\phi_i$}(lr)--node[right]{$f_i\cdot\beta$}(ur)--node[above]{$\alpha_i$}(ul)--node[left]{$\beta$}(ll);
        \draw[fill=black] ($(ul)!0.25!(ll)$) circle (0.1);
        \node[left] (x) at ($(ul)!0.25!(ll)$) {$x$};
        \draw[fill=black] ($(ur)!0.4!(lr)$) circle (0.1);
        \node[right] (xs) at ($(ur)!0.4!(lr)$) {$x^*$};
        \draw[fill=black] ($(ur)!0.25!(lr)$) circle (0.1);
        \node[right] (fx) at ($(ur)!0.25!(lr)$) {$f_i\cdot x$};
        \draw[dashed] (x)--node[sloped,above]{$2\delta$}(xs);
        \draw[fill=black] (ul) circle (0.1);
        \draw (ul) node[left]{$p$};
        \draw[fill=black] (ur) circle (0.1);
        \draw (ur) node[right]{$f_i\cdot p$};
      \end{scope}
    \end{tikzpicture}
  \end{center}
  where $x,x^*$ are preimages of some point $\bar x$ in the approximating
  tree, and the segment containing $\bar x$ is the
  \emph{$\beta$-pinched segment}. Now $d(x,x^*) \leq 2\delta$.

  \emph{Sub claim: $d(f_i\cdot x,x^*) \leq |\alpha_i|\leq 3\delta+1$.}
  Consider the points $p,f_i\cdot p$ and $\bar m$ shown in the figure
  above. Along $\beta$ and $f_i\cdot\beta$, we have
  $d(x,p) = d(f_ix,f_i\cdot p)$. Consider the preimages $m$ of
  $\bar m$ on the arc $\beta$ and the preimage $m^*$ on
  $f_i\cdot \beta_i$ (both not shown) in the right quadrilateral. 
  Along the geodesic $f_i\cdot \beta$, we see that $f_i\cdot x$ is a
  distance $d(x,p) = d(\bar x, \bar m)+d(\bar m,p)$ from $f_i\cdot
  p$. Along this same geodesic we also see that $x^*$ is a distance
  $d(x^*,f_i\cdot p) = d(\bar x, \bar m)+d(\bar m,f_i(p))$ from
  $f_i\cdot p$. It follows that
  $d(f_i\cdot x,x^*) \leq | d(\bar m,f_i(p)) - d(\bar m,p)| \leq
  |\alpha_i|$; \emph{so the sub claim is proved.}
  
  It follows that $d(x,f_i\cdot x)\leq 5\delta+1$. Suppose now that
  $\beta$ is three times longer than every $\phi_i$ and $\alpha_i$,
  then the midpoint $m$ of $\beta$ is mapped to the $\beta$-pinched
  segment of the approximate tree.  In this case $m$ is moved at most
  $5\delta+1$ by $F$ so by Lemma \ref{lem:bg-almost-fixed-point} there
  is an $F$-almost fixed point a distance at most
  $\frac{5\delta+1}{2} \leq 3\delta$ from $m$. It follows that if
  additionally $|\beta|> 6\delta + 54K_{N}\delta$ (from Lemma
  \ref{lem:close-types}) then there is a shorter geodesic $\gamma$
  connecting $1$ to some $F-$almost fixed point $x_F'$ in the same
  $G$-orbit as $x_0$, contradicting the choice of $\beta$. \emph{The
    claim is proved. Specifically $|\beta|$ is bounded by a computable
    linear function of $\left|\max_{i}\varphi_i\right|$ with
    coefficients depending on $K_{N}$ and $\delta$.}

Now if we are given $F$ and $H$ as in the statement of the Theorem,
they are both conjugate to some $F_0$ that has almost fixed points
close to the identity and the length of the shortest conjugator from
$F$ to $H$ is at most the sum of the lengths of the conjugators from
$F_0$ to $F$ and $F_0$ to $H$. The result now follows.
\end{proof}

\begin{prop}\label{prop:conj-prob}
  Let $F,F'$ be finite non-parabolic subgroups. Then there is an
  algorithm that decides whether there is some $k \in G$ such that
  $kFk^{-1} = F'$.  Given a presentation $G =\bk{Z\mid R}$ the
  shortest conjugator is of word-length bounded in terms of the words
  defining the elements of $F$, $F'$, and of the hyperbolicity
  constant of the Grove-Manning space.
\end{prop}
\begin{proof}
  This result follows immediately from the previous lemma and
  following facts. Given a presentation $G=\bk{Z\mid R}$, 
  then the word length
  of an element $g$ gives an upper bound for the geodesic
  $[1,g\cdot 1]$ in $X$ the Groves-Manning space. Conversely $X$
  distorts the distance in the embedded Cayley graph in a controlled
  exponential fashion so that if $[1,g\cdot 1]=k$ then $g$ can be
  written as a word in $Z$ of length at most $2^k$.

  So given a set of generating words of $F$ and $F'$, by Lemma
  \ref{lem:make-balls} we can construct a large enough ball $B(r)$
  about the identity in $X$ to contain the orbit of $1$ by the
  elements $F$ and $F'$. We can then enlarge the radius of this ball
  by a computable amount, namely $2(A2r+B)$ as in
  (\ref{eq:lin-conj-bound}) from Lemma \ref{lem:short-conj-length} to
  find the shortest conjugator.

  Finally, we note that the $Z$-word length of the shortest conjugator
  $k$ is at most\[ 2^{2(A2m+B)},
  \] where $m$ is the longest word length from our list of words representing
  the elements of $F\cup F'$.

\end{proof}

A finite subgroup is called \define{strictly parabolic} if it is 
parabolic, but not contained in the intersection of two distinct
parabolic subgroups.

\begin{lem}\label{lem:mwhp-conj}
  Suppose that we can solve the Mixed Whitehead Problem in the
  subgroups of $\calP$. Then given two finite strictly parabolic
  subgroups $F,F'$ we can determine whether there is an element of $g$
  that conjugates $F$ into $F'$. Furthermore we can compute the action
  of the normalizer $N_G(F)$ on $F$.
\end{lem}
\begin{proof}
  The simultaneous conjugacy problem asks if given tuples
  $(s_1,\ldots,s_r)$ and $(t_1,\ldots,t_r)$, whether there is an
  element $g\in G$ such that $g^\mo s_i g = t_i$. It is easy to see,
  by requiring that a generating set be sent to a conjugate of a
  generating set, that the simultaneous conjugacy problem is a special
  case of the Mixed Whitehead Problem.  Given $F,F'$ as in the
  statement, we can use Bumagin's result \cite{bumagin2004conjugacy}
  to determine whether $F'$ is conjugate into the same parabolic
  subgroup as $F$. Note that strict parabolicity implies that there is
  a unique maximal parabolic overgroup $P$ that could contain them
  both, up to conjugacy. Since parabolic subgroups are
  self-normalized, the problem reduces to finitely many instances of
  the simultaneous conjugacy conjugacy problem where the tuples
  correspond to elements of $F,F'$.

  To compute the action of $N_G(F)$ on $F$, we take the tuple
  $(f_1,\ldots,f_r)$ of elements of $F$ and for each $\varphi \in
  \aut F$ we take the target tuple
  $(\varphi(f_1),\ldots,\varphi(f_r))$ and apply the simultaneous
  conjugacy problem to decide whether $\varphi$ can be induced by an
  inner automorphism of $G$.
\end{proof}

\subsubsection{The many-ended case}

We discuss now the case of Theorem \ref{theo;main} in which the
relatively hyperbolic groups have infinitely many ends. The case of
torsion-free groups is easier and follows from Proposition
\ref{prop;compute-Grushko} and Theorem \ref{thm:main-one-ended}.

Let us turn to the case of groups with torsion. The Dunwoody Stallings
decompositions  
plays 
the role of the Grushko decomposition, however, the overall argument
is significantly more involved.

A first difficulty is to compute a (any) Dunwoody Stallings
decomposition.  
 By Proposition
\ref{prop:compute-DS}, it suffices to have an algorithm that detects whether
a given relatively hyperbolic group $(G,\calP)$ in the class is one-ended or
not. Consider a torsion-free finite-index subgroup $(G_0, \calP_0)$, which exists by
assumption, and which is computable by enumeration (of morphisms to
finite groups, and of finite subgroups of $G_0$.)   
  Then
$(G, \calP)$ has one end if and only if $(G_0, \calP_0)$ has one end, if
and only if it is freely indecomposable. The later property is
decidable by Proposition \ref{prop;compute-Grushko}.

A second difficulty is the non-uniqueness of the Dunwoody-Stallings
decompositions. As in \cite{DG_gafa} we will treat this problem by
computing orbit representatives of Dunwoody-Stallings decompositions
under the automorphism groups.

If $\bbD$ and $\bbD'$ are two Dunwoody-Stallings splittings their
vertex groups must be conjugate, and in particular the dual Bass-Serre
trees lie in the same deformation space. Because all edge groups are
properly contained in vertex groups \cite[Proposition
7.1]{guirardel_deformation_2007} and \cite[Theorem
7.2]{guirardel_deformation_2007} immediately imply that $\bbD'$ can be
obtained from $\bbD$ by a sequence of slides. Specifically a slide of
an edge $e$ across an edge $f$, where $e\neq f$ and for which
$i_e(\Gamma_e)$ is conjugate to a subgroup of $i_f(\Gamma_f)$, is a
transformation of a graph of groups as depicted:

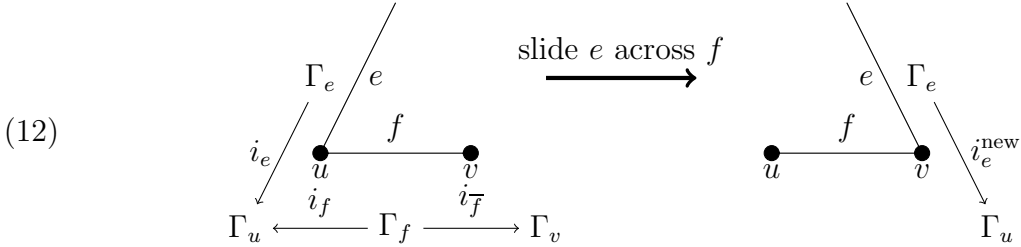
\begin{equation}
  \label{eq:slide}
  \begin{array}{c}
    \begin{tikzpicture}
      \node (gammae) at (0,1) {$\Gamma_e$};
      \node (gammau) at (-1,-1) {$\Gamma_u$};
      \node (gammaf) at (1,-1) {$\Gamma_f$};
      \node (gammav) at (3,-1) {$\Gamma_v$};
      \draw[fill=black] (1,2) --node[right]{$e$} 
      (0,0) node[below]{$u$} circle (0.1) --node[above]{$f$} 
      (2,0) node[below]{$v$} circle (0.1);
      \draw[->] (gammae)--node[left]{$i_e$}(gammau);
      \draw[->] (gammaf)--node[above]{$i_f$}(gammau);
      \draw[->] (gammaf)--node[above]{$i_{\overline f}$}(gammav);
      
      \draw[ultra thick,->] (3,1)--node[above]{slide $e$ across $f$}(5,1);
      
      \begin{scope}[xshift=6cm]
        \draw[fill=black] (1,2) --node[left]{$e$} 
        (2,0) node[below]{$v$} circle (0.1) --node[above]{$f$} 
        (0,0) node[below]{$u$} circle (0.1);
        \node (gammaen) at (2,1) {$\Gamma_e$};
        \node (gammaun) at (3,-1) {$\Gamma_u$};
        \draw[->] (gammaen) -- node[right]{$i^{\textrm{new}}_e$} (gammaun);
      \end{scope}
    \end{tikzpicture}
  \end{array}
\end{equation}

where possibly $u=v$. The new monomorphism
$i^{\textrm{new}}_e:\Gamma_e \to \Gamma_v$ is given by the well
defined restriction
\[ i^{\textrm{new}}_e= \left.\left(i_f\circ i_f^\mo\circ \ad_{g_{e,f}}\circ
    i_e\right)\right|_{\Gamma_e}
\]
where $g_{e,f} \in \Gamma$ conjugates $i_e(\Gamma_e)$ to a subgroup of
$i_f(\Gamma_f)$. Propositions \ref{prop:compute-normalizer},
\ref{prop:conj-prob} and Lemmas \ref{lem:parabolic-intersection},
\ref{lem:mwhp-conj} immediately imply:

\begin{cor}\label{cor:compute-slide}
  Let $(G,\calP)$ be relatively hyperbolic where it is possible to
  solve the Mixed Whitehead Problem in the parabolic subgroups. Given
  a Dunwoody-Stallings decomposition, and 
  distinct edges $e,f$ that meet at a vertex $u$, it is possible to
  find all possible slides of $e$ across $f$.
\end{cor}

\begin{prop}\label{prop:list-DS}
  Let $(G,\calP)$ be relatively hyperbolic where it is possible to
  solve the Mixed Whitehead Problem in the parabolic subgroups and let
  $\bbD$ be a given Dunwoody-Stallings decomposition. There is an algorithm
  that  finds, up to isomorphism (Definition \ref{def;iso_gog}), all
  possible graphs-of-groups that are obtainable from $\bbD$ via a
  sequence of slides.
\end{prop}
\begin{proof}
  Repeatedly applying Corollary \ref{cor:compute-slide}, we can
  enumerate all possible Dunwoody-Stallings splittings. We now must
  show that we can compute a complete finite list up to isomorphism of graphs
  of groups.

  There are only finitely many types of underlying graphs that can be
  obtained via sequences of slides. Furthermore if a sequence $\sigma$
  of slides induces an isomorphism of graphs such that corresponding
  edge groups and vertex groups are sent to isomorphic copies, then
  the data of the new graph of groups will differ only in that the
  boundary monomorphisms $i_e$ will be replaced by
  \[ \ad_{g_{\sigma,e}} \circ i_e\circ\varphi_{\sigma,e}\] where
  $\varphi_{\sigma,e} \in \aut{\Gamma_e}$ is called the
  \define{holonomy automorphism with respect to the slide sequence
    $\sigma$}, and $g_{\sigma,e}$ is the appropriate element of
  $\Gamma_{t(e)}$. Because the edge groups are finite it is possible
  to determine whether the holonomy automorphisms are conjugations
  and, if so, by Propositions \ref{prop:compute-normalizer},
  \ref{prop:conj-prob} and Lemmas \ref{lem:parabolic-intersection},
  \ref{lem:mwhp-conj} we can determine if some splitting $\bbD'$ is
  isomorphic as a graph of groups to another splitting already in our
  list. Because there are only finitely many possible holonomy
  automorphisms and only finitely many conjugacy classes of finite
  subgroups that arise as the images of edge group monomorphisms, we
  will eventually construct a complete finite list.
\end{proof}

We are now ready to prove the main result.

\begin{proof}[Proof of Theorem \ref{theo;main}.]
  Let $(G,\calP), (G', \calP')$ be two relatively hyperbolic groups as
  in the statement Theorem \ref{theo;main}. By Proposition
  \ref{prop:list-DS}, we may compare in parallel all pairs of
  Dunwoody-Stallings decompositions of $(G,\calP)$ and
  $(G',\calP')$. We cannot proceed exactly as in the proof of the
  first statement of \cite[Lemma 7.3]{DG_gafa}, because we cannot find
  generating sets of the automorphism groups, instead we must work
  with computation of orbits in Dehn fillings.

  Consider $\bbD, \bbD'$, respectively, two Dunwoody-Stallings
  decomposition of them. Denote by $\Gamma_v, \Gamma_{v'}$ their
  vertex groups (for $v\in V(D), v'\in V(D')$), and by
  $\Gamma_e, \Gamma_{e'}$ their edge groups
  ($e\in E(D), e'\in E(D')$.)

  Consider a marking for each edge group $\Gamma_e = \Gamma_{\bar e}$,
  for $e\in E(D)$, and similarly for $e'\in E(D') $ (in other words,
  the description of $\Gamma_e$ as an ordered tuple of elements.) For
  each vertex $v$, write $\calF_v$, a marked peripheral structure of
  $\Gamma_v$ of the conjugacy classes of the subgroups of adjacent
  edges to $v$ in $\bbD$ (and similarly for $\bbD'$) such that a marking
  of $[i_e(\Gamma_e)]$ in $\Gamma_{t(e)}$ is conjugate to the image by
  $i_e$ of the marking of $\Gamma_e$.
  For convenience let us call $F_e$ the chosen marking in
  $\Gamma_{t(e)}$ corresponding to the conjugacy class of
  $i_e(\Gamma_e)$.

Given    an isomorphism  of the underlying graphs $\psi : D \to D'$ (we assume, up to  renaming, that it sends $v$ to $v'$), and isomorphisms $\psi_v: \Gamma_v \to \Gamma_{v'}$,    by  Proposition \ref{prop;compatible_twist}
$(G,\calP)$ and $(G',\calP')$ are isomorphic by an isomorphism from
$\bbD$ to $\bbD'$ inducing $\psi$  if and only if 
for each $v$ 
there exists an automorphism $\alpha_{v'}$ of $\Gamma_{v'}$ (respecting
 the relative hyperbolic structure) that sends each marking
$F_{e'}  (= F_{\psi(e)}\in \calF_{v'})$ to a conjugate of $\psi_v(F_e)$. We call
this property $(\star_v)$. This property  amounts to the existence of suitable extension adjustments in the Dunwoody-Stallings decomposition.

The graphs $D,D'$ are finite, thus we may decide whether $D$ and $D'$
are isomorphic graphs. Using on each vertex group the algorithm of
Theorem \ref{thm:main-one-ended}, we may assume that we are given
isomorphisms between $\Gamma_v$ and $\Gamma_{v'}$ for all $v$.

We now assume that there exists $\psi :D\to D'$ a graph isomorphism,
and for all vertices $v\in D^{(0)}$, an isomorphism of relatively
hyperbolic groups
$\psi_{v}: \Gamma_{v} \to \Gamma_{\psi(v)} (= \Gamma_{v'})$ (for the
induced relatively hyperbolic structure), for if one of them does not
exist, the groups $(G,\calP), (G', \calP')$ are thus proven
non-isomorphic through these particular Dunwoody-Stallings
decompositions.

The graph isomorphism $\psi$ provides, for
each vertex, a bijection between the adjacent edges, therefore a
bijection between $\calF_{v}$ and $\calF_{\psi(v)} =\calF_{v'}$. For
each $F\in \calF_{v}$, we call its image in  $\calF_{v'}$ by this
bijection, its counterpart. At this stage, in principle,  it does not
come from an 
isomorphism between $\Gamma_v$ and $\Gamma_{v'}$. However, $\psi_v$ indeed sends $\calF_{v}$ to a marked peripheral structure of finite groups of $\Gamma_{v'}$.

Consider $\bbJ_v$ the canonical JSJ decomposition of $\Gamma_v$ (and
similarly for $\Gamma_{v'}$.) Note that $\psi_v$ induces an
isomorphism of graphs of groups between $\bbJ_v$ and $\bbJ_{v'}$ (they
are canonical.)  We will write $w'$ in $\bbJ_{v'}$ for the image of a
vertex $w$ in $\bbJ_v$. As in the beginning of the proof of Theorem
\ref{thm:main-one-ended}, we arrange so that $\bbJ_v$ and
$\bbJ_{v'}$ are peripheral splittings.
For a vertex $w$ of $\bbJ_v$, we denote by $\Upsilon_{v,w}$ the
corresponding vertex group. We use similar notation for edge
groups. 

Consider the groups in $\calF_v$. They are conjugate to edge groups of a
Dunwoody-Stallings decomposition, hence finite, and therefore, in
$\Gamma_v$ each of these groups has a conjugate in a certain vertex
group $\Upsilon_{v,w}$. Denote by $\calF_{v,w}$ the collection of those
that are conjugated into $\Upsilon_{v,w}$. This may not give a
partition of $\calF_v$ because some groups can be conjugated to an
edge group of the JSJ decomposition and appear in several distinct vertex
groups. 

For every $F_e\in \calF_v$ that is conjugate into $\Upsilon_{v,w}$, let
$F_e(w)$ be a specific conjugate of the marking in  $\Upsilon_{v,w}$.

The property  $(\star_v)$ asks for  certain automorphisms of
$\Gamma_v$, hence of $\bbJ_v$ since the later is canonical.  One can
determine which automorphisms of the underlying graph of $\bbJ_v$ are
realized as automorphism of the graph-of-groups, and collect, for each
of those, an automorphism realizing the underlying graph
automorphisms. Up to composing with one of these, we may assume
without loss of generality that we are searching for an automorphism that induces the identity on the underlying graph.

By the canonicity of the JSJ decomposition,     the property $(\star_v)$ 
 happens if and only if there exists a collection of
automorphisms $\alpha_{w'}$ of the vertex groups $\Upsilon_{v',w'}$ of
$\bbJ_{v'}$ that: 

\begin{enumerate}
\item[$(\star_{v}^1)$]\label{cond;one} for all $F'_{\psi(e)} (w') \in \calF_{v',w'}$ (whose counterpart
  in $\calF_{v,w}$ is denoted $F_e(w)$), send the marking
  $F'_{\psi(e)}(w)$ to a $\Upsilon_{v',w'}$-conjugate of
  $\psi_v(F_e(w))$

\item[$(\star_{v}^2)$]\label{cond;two}  the collection of automorphisms $\alpha_{w'}$ preserve
  $\bbJ_{v'}$ in the following sense: given a marking
  $M_{\epsilon'} = M_{\bar{\epsilon'}}$ for each edge group
  $\Upsilon_{v',\epsilon'}$, there exists another marking
  $N_{\epsilon'} = N_{\bar \epsilon'}$ of $ \Upsilon_{v',\epsilon'} $
  such that $\alpha_{t(\epsilon')}(i_{\epsilon'} (M_{\epsilon'}))$ is
  $ \Upsilon_{v',w'}$-conjugated to $i_{\epsilon'} (N_{e'}) $. 

\end{enumerate}

The second condition $(\star_{v}^2)$ is need to extend the collection
of vertex group automorphisms $\alpha_{w'}$ into a global automorphism of the graph of groups $\bbJ_{v'}$, see Definition \ref{def;iso_gog}.

Let us denote by $\calM_{w'}$ the marked peripheral structure of the
adjacent edges of $w'$, marked by the tuples
$i_{\epsilon'} (M_{\epsilon'})$.  Let us also gather
$\calF_{v',w'}^\calP$ the markings of $\calF_{v',w'}$ that are
contained in peripheral subgroups of $\Upsilon_{v',w'}$, and
$\calF_{v',w'}^\calT$ the markings that are not contained in any
peripheral subgroups ($\calT$ stands for thick.) We call
$ ( (\calM_{w'}, \calF_{v',w'}^\calP))_{ w'\in \bbJ_{v'} ^{(0)}} $ the
\define{collection of source markings} of the vertex groups. The orbit
under
$\prod_{ w'\in \bbJ_{v'} ^{(0)} } \Out(\Upsilon_{v',w'},
\calP_{\Upsilon_{v',w'}} )$, of the collection of source markings is
called the set of \define{collection of hit markings}.




Denote by
$ ( (\calH_{w'}, \calF\calH_{v',w'}^\calP))_{ w'\in \bbJ_{v'} ^{(0)}}
$ a collection of hit markings. We define \define{ eligible
  collections of hit markings} to be those collections of hit markings
that are such that, first, $\calF\calH_{v',w'}^\calP$ are tuple-wise
conjugate to $\psi_{v} (\calF_{v,w}^\calP)$ in $\Upsilon_{v',w'}$, and
second, for all edges $e'$, the pull-back by the attachment map
$i_{e'}$ of $\calH_{t(e')}$ in $\Upsilon_{v',e'}$ corresponds to the
pull-back by the attachment map $i_{\bar e'}$ of $\calH_{t(\bar e')}$
in $\Upsilon_{v',\bar e'}$ (which is \emph{equal} to
$ \Upsilon_{v',e'} $), up to conjugation in the edge group
$ \Upsilon_{v',e'} $.
 

Note that in white vertices, the orbit of the source markings by
$\Out(\Upsilon_{v',w'}, \calP_{\Upsilon_{v',w'}} )$ is finite and
computable, by (Proposition \ref{prop;orbit_computation}).  Once these
have been computed, the mixed Whitehead problem in the black vertices
can be used to determine whether a partial collection of markings in
the white vertex groups can be extended to an eligible collection of
hit markings.



The conditions  $(\star_{v}^1), (\star_{v}^2)$  are thus  equivalent to the existence of  automorphisms $\alpha_{w'}$
of $\Upsilon_{v',w'}$  sending the collection of source markings to a collection of eligible hit
markings and  sending the
finite marked groups $F'_{\psi(e)}(w')$   of  $\calF_{v',w'}^\calT$
on conjugates of   $\psi_v(  F_{e}(w) )$, for $F_{e}(w)$ the
counterpart of  $F'_{\psi(e)}(w')$.   We now proceed to decide the existence of these automorphisms.   

For $\QH$ groups $\Upsilon_{v',w'}$, which are hyperbolic relative to
two-ended groups and therefore hyperbolic, this can be achieved by
applying \cite[Corollary 8.2]{DG_gafa}. For parabolic groups
$\Upsilon_{v',w'}$, the solution to the mixed Whitehead problem
answers the question. The more difficult situation is for rigid
groups. In this last case, we follow the same method that allowed us to
compute the orbit under the automorphism group of markings of
parabolic subgroups.

We iteratively consider the canonical Dehn Fillings
$\bar \Upsilon_{v',w'} = \Upsilon_{v',w'}/K_n$ (where $K_n$ is the
$n$-th characteristic Dehn kernel as in section \ref{subsec;injpi}),
check whether it is hyperbolic by \cite{Pap}, check whether it has
finite outer automorphism group by \cite[\S 6.2]{DG_gafa}, and if so,
compute its outer automorphism group (by \cite[\S 8]{DG_gafa}), and
check whether there is an automorphism that, not only sends the image
of the source marking to the image of an eligible hit marking, but
also sends all the markings of $\overline{\calF_{v',w'}^\calT}$ to
conjugates of the markings of $\psi_v(\overline{ \calF_{v,w}^\calT})$.

If one finds such a Dehn filling in which there is no such automorphism, the game is over: property $(\star_v)$ cannot be
satisfied.

Observe  that by the Dehn filling theorem (and residual
finiteness of the peripheral subgroups), for $n$ large enough, all
$\Upsilon_{v',w'}/K_n$ will be hyperbolic, and by
\cite[Theorem. 5.2]{DG_charac}, for $n$ large enough, all have finite
outer automorphism groups. Therefore, the preliminary checks performed
above will
eventually be successful.

 Assume finally that,   for each Dehn Filling satisfying these preliminary checks, there exists an automorphism $\bar
\alpha_{v',w',n}$ as above. 
We need to check that
there is an automorphism of $\Upsilon_{v',w'}$ sending the source
marking to an eligible hit marking and sending all the markings of
$\calF_{v',w'}^\calT$ to conjugates of the  markings of
$\psi_v( \calF_{v,w}^\calT)$.

By \cite[Theorem 5.2]{DG_charac}, there exists an
automorphism $\alpha_{w'}$ of $\Upsilon_{v',w'}$ inducing, through the
quotients,  
infinitely many   automorphisms $\bar
\alpha_{v',w',n}$ in the Dehn Fillings, up to conjugation in  $\bar
\Upsilon_{v',w'}$. After taking
a subsequence, we thus assume it induces all of them. 

We modify $\bar \alpha_{v',w',n}$ by conjugation in
$\bar \Upsilon_{v',w'}$ so that now $\bar \alpha_{v',w',n}$ is induced
by $\alpha_{w'}$ through the quotient.  It follows that the image of
$\bar \alpha_{v',w',n}$ of each tuple in $\calF_{v',w'}^\calT$ remains
at bounded distance from the identity (uniformly in $n$) in the word
metric induced by a chosen one in $\Upsilon_{v',w'}$.

 By Proposition \ref{prop:conj-prob}  it follows that the conjugator in $\bar
\Upsilon_{v',w'} $ conjugating a tuple of $\overline{ \calF_{v',w'}^\calT}$ to its
counterpart in $\psi_v(\overline{ \calF_{v,w}^\calT})$ can be taken to be  bounded
in word length.   
After taking a subsequence again, we can assume that these conjugators
are the images of a word independent of $n$. It follows that the same
words conjugate the markings of
$\calF_{v',w'}^\calT$  to  the  markings of
$\psi_v( \calF_{v,w}^\calT)$.

Thus if there does not exist such an automorphism $\alpha_{w'}$, this
will be apparent in a certain Dehn Filling, and as we said, by our
procedure, we will algorithmically discover it. On the other hand if
there exists such an automorphism $\alpha_{w'}$, an enumeration
process will discover it. Doing so for all eligible hit markings
allows us to decide whether or not $(G,\calP)$ and $(G',\calP')$ are
isomorphic by an isomorphism from $\bbD$ to $\bbD'$ inducing
$\psi$. By the reduction performed before, this solves the isomorphism
problem between $(G,\calP)$ and $(G',\calP')$.

\end{proof}

\section{Algorithmic problems in nilpotent groups}\label{sec;nilp}
As promised, we now prove that finitely generated nilpotent groups
satisfy the  fourth and fifth 
assumptions on the peripheral subgroups in Theorem
\ref{theo;main}.

Let $N$ be a finitely generated nilpotent group, such groups will be
called \define{$\fgn$-groups.} Let $\outo(N)$ denote the finite order
elements of $\Out(N)$.  We write $H \charleq G$ to denote a
characteristic subgroup and $H \nsgp G$ to denote a normal subgroup.
Recall (Definition \ref{def;CST}) that a finite index characteristic
subgroup $P \charleq N$ such that the elements of $\outo(N)$ do not
vanish via the natural map $\Out(N) \to \Out(N/P)$, is said to
\define{separate torsion in $\out{N}$.}

The goal of this section is to prove the following two results:

\begin{thm}\label{thm:find-deep-enough}
  Let $N$ be an \fgn-group. Then given a finite presentation of $N$ we
  can algorithmically construct a generating set of a subgroup $P$
  which separates torsion in $\Out(N)$.
\end{thm}

\begin{thm}\label{thm:mwhp}
  There is a uniform algorithm to solve the mixed Whitehead problem
  for \fgn-groups.
\end{thm}

As we said in the introduction, these two results, with  Theorem \ref{theo;main} and Corollary
\ref{cor;canonicity}    give Theorem
\ref{theo;ip-nilpotent}, namely, that the isomorphism problem is solvable for
the class of virtually torsion-free relatively hyperbolic groups with finitely
generated nilpotent parabolics.

\subsection{Congruences effectively separate the torsion in finitely
  generated nilpotent groups}
The algorithm we shall present that finds a finite index
characteristic subgroup of $N$ that separates the torsion in $\out N$
will be by recursion on the upper central series length. It is almost
completely elementary. Let $\nu_1 N$ denote the centre of $N$ and
suppose we have found finite index subgroups in $\nu_1N$ and
$N/\nu_1N$ that separate torsion in $\out{\nu_1N}$ and
$\out{N/\nu_1N}$ respectively. In section \ref{sec:good-enough} we
will show how to construct a finite index subgroup $P_0$ of $N$ that
is ``good enough'' in the sense that it will separate the torsion in
$\out N$ which can be seen by its natural images in $\out{\nu_1N}$ and
$\out{N/\nu_1N}$. The finite order elements of $\out{N}$ which are not
separated by $P_0$ are called \emph{elusive}. In Section
\ref{sec:elusive-characterization} we give a complete description of
these elements and in Section \ref{sec:elusive-computation} we give a
procedure to find all of them. Finally in Section
\ref{sec:find-cong-sep}, with the help of Proposition \ref{prop:segal}
(due to Dan Segal), we assemble all these components to give a proof
of Theorem \ref{thm:find-deep-enough}

\subsubsection{A partial inductive step}\label{sec:good-enough}

Let $\nu_0 = \{1\} \nsgp \nu_1 N \nsgp \ldots \nsgp \nu_m N = N$
denote the upper central series. In particular $\nu_1 N = Z(N)$,
i.e. it is the centre, and $\nu_{i+1} N$ is the preimage of
$Z(N/\nu_iN)$. The proof of Theorem \ref{thm:find-deep-enough} will be
by induction on the length of the upper central series of $N$, denoted
$\lucs(N)$.

Consider the short exact sequence
\begin{equation}\label{eqn:ses}1 \to \nu_1N \to N \to
  N/\nu_1N \to 1,
\end{equation}
then in particular $\lucs(\nu_1N) = 1$ and $\lucs(N/\nu_1N) =
\lucs(N)-1$. Let $\beta \in \Aut(N)$ and denote by $[\beta]$ its
class in $\Out(N)$. We have two homomorphisms from $\Out(N)$. The
first one:
\begin{equation}\label{eqn:p}
  \begin{tikzpicture}
    \node (outN) at (0,0) {$\Out(N)$};
    \node (outNnu1N) at (3,0) {$\Out(N/\nu_1N)$};
    \node (beta) at (0,-1) {$[\beta]$};
    \node (betabar) at (3,-1) {$[\ol{\beta}]$};
    \draw[->] (outN) -- node[above]{$p$} (outNnu1N);
    \node at (1.5,-1) {$\mapsto$};
  \end{tikzpicture}
\end{equation}
follows from the fact that $\nu_1N \charleq N$ and the second one:
\begin{equation}\label{eqn:r}
\begin{tikzpicture}
  \node (AutN) at (0,0) {$\Aut(N)$};
  \node (Aut1N) at (3,0) {$\Aut(\nu_1 N)$};
  \node (OutN) at (0,-1) {$\Out(N)$};
  \node (Out1N) at (3,-1) {$\Out(\nu_1 N)$};
  \node (beta) at (0,-2) {$[\beta]$};
  \node (betar) at (3,-2) {$[\beta|]$};
  \node[rotate=90] at (3,-0.5) {$=$};
  \draw[->] (AutN) -- (Aut1N);
  \draw[->>] (AutN) -- (OutN);
  \draw[->] (OutN) -- node[above]{$r$} (Out1N);
  \node at (1.5,-2) {$\mapsto$};
\end{tikzpicture}
\end{equation}
follows from the fact that, on the one hand restriction gives a
homomorphism $\Aut(N) \to \Aut(\nu_1N)$ and on the other hand since
$\nu_1 N$ is free abelian, all its automorphisms are outer. The
following proposition is needed for our induction.

\begin{prop}[Good enough]\label{prop:good-enough}
  Let $P$ be a finitely presented and subgroup separable group and let
  $H \charleq P$ be a finitely presented characteristic
  subgroup. Suppose there are given finite index subgroups $H_0
  \charleq H$ and $K_0 \charleq P/H$. Then there is a finite index
  subgroup $P_0 \charleq P$ such that:
  \begin{itemize}
  \item[(a)] $P_0H/H \leq K_0$ and $P_0H/H \nsgp P/H$.
  \item[(b)] $P_0 \cap H \leq H_0$ and $P_0 \cap H \nsgp H$.
  \end{itemize}
  Moreover; given a finite presentation of $H$ and $P$ and an explicit
  homomorphism $H \hookrightarrow P$ given by the images of the generators of
  $H$, generating sets for $H_0 \leq H$ and $K_0 \leq P/H$; $P_0$ can
  be constructed algorithmically.
\end{prop}
\begin{proof}
  Because $[H:H_0]$ is finite we have the finite
  partition \begin{equation}\label{eqn:cosets}H = H_0 \cup g_1H_0 \cup
    \ldots \cup g_nH_0\end{equation} with $g_i \not\in H_0$. Since $P$
  is subgroup separable we can find a finite index $P^* \leq P$ such
  that $P^* \geq H_0$ and $P^* \cap \{g_1,\ldots,g_n\} =
  \emptyset$. We therefore have \[P^* \cap H = H_0.\] W.l.o.g. $P^*$
  can be characteristic by being the intersection of all finite index
  subgroups $H$ of minimal index in $P$ such that $H\cap H = H_0$. We
  now look at \[\pi: P \onto P/H.\] Since $K_0 \charleq P/H$ and is
  finite index we have $\pi^\mo(K_0) \charleq P$, which is also finite
  index. Finally we take \[P_0 = P^* \cap \pi^\mo(K_0).\] Since it is
  the intersection of two finite index characteristic subgroups of $P$
  we find that $P_0 \charleq P$ and $[P:P_0] < \infty$, and $P_0$ is
  easily seen to satisfy the requirements of the proposition.

  We are given finite presentations of $H$ and $P/H$. We then use the
  Todd-Coxeter algorithm for $H_0 \leq H$ to produce the list of
  cosets (\ref{eqn:cosets}). Since we have an explicit homomorphism $H
  \hookrightarrow P$ we have a generating set for $H_0$ in $P$ as well
  as the images of the coset representatives $g_1,\ldots,g_n$ in
  $P$. We identify these with their image in $P$.

  We now enumerate the finite index subgroups $P=P_1, P_2,\ldots$ of
  $P$ in a way such that if $i < j$ then $[P:P_i] \leq
  [P:P_j]$. Eventually we will find some subgroup $P_k \leq P$ of
  index $d$ such that:\begin{itemize}
    \item[(i)]$H_0 \leq P_k$, and 
    \item[(ii)] $\{g_1,\ldots,g_n\} \cap P_k = \emptyset$.
  \end{itemize}
  If we take the intersection $P^*$ of all subgroups of index $d$
  satisfying (i) and (ii), then we have a finite index characteristic
  subgroup of $P$ with the desired properties.

  Now given $K_0 \leq P/H$ we can easily find $\pi^\mo(K_0)$ as the
  only finite index subgroup of $P$ that contains $H$ and maps onto
  $K_0$ via $\pi$ (the subgroup membership problem for $K_0\leq P/H$
  is decidable because $[P/H:K_0] < \infty$ .) Finally computing the
  intersection of finite index subgroups of finitely presented groups
  is straightforward so the desired subgroup $P_0$ can be constructed
  algorithmically.
\end{proof}

\begin{defn}
  A subgroup $P_0\charleq P$ that satisfies (a) and (b) of Proposition
  \ref{prop:good-enough} is said to be \define{good enough w.r.t $K_0$
    and $H_0$.}
\end{defn}

Suppose that we were able to find subgroups $N_0 \charleq \nu_1N$ and
$K_0 \charleq N/\nu_1N$ that separated torsion in $\Out(\nu_1N)$ and
$\Out(N/\nu_1N)$ respectively and recall the homomorphisms $p,r$ given
in (\ref{eqn:p}) and (\ref{eqn:r}.)

\begin{defn}
  We call \define{elusive} any element $[\beta] \in \outo(N)$ that
  lies in $\ker(r) \cap \ker(p) \setminus \{1 \}$.

\end{defn}

\begin{lem}\label{lem:good-enough-elusive}
  Let $N_0 \charleq \nu_1N$ and $K_0 \charleq N/\nu_1N$ be deep enough
  and let $P_0 \charleq N$ be good enough w.r.t. $N_0$ and $K_0$, then
  every $[\beta] \in \outo(N)$ that isn't elusive survives in the
  homomorphism \[ \Out(N) \to \Out(N/P_0).\]
\end{lem}
\begin{proof}

  By construction we have the following commutative diagram with exact
  rows and columns:
  \begin{equation}
    \begin{tikzpicture}[yscale=2.3,xscale=2.5]
    \node (a) at (1,-0.5) {$1$};
    \node (b) at (2,-0.5) {$1$};
    \node (c) at (3,-0.5) {$1$};
    \node (d) at (4,-1) {$1$};
    \node (e) at (4,-1.5) {$1$};
    \node (eh) at (4,-2) {$1$};
    \node (f) at (3,-2.5) {$1$};
    \node (g) at (2,-2.5) {$1$};
    \node (h) at (1,-2.5) {$1$};
    \node (hh) at (0,-2) {$1$};
    \node (i) at (0,-1) {$1$};
    \node (j) at (0,-1.5) {$1$};
    \node (n1n) at (1,-1.5) {$\nu_1N$};
    \draw[->] (j) -- (n1n) ;
    \node (N) at (2,-1.5) {$N$};
    \draw[->] (n1n) -- (N);
    \node (nn1n) at (3,-1.5) {$N/\nu_1N$};
    \draw[->] (N) -- (nn1n);
    \draw[->] (nn1n) -- (e);
    \node (pon1n) at (1,-2) {$P_0 \cap \nu_1N$};
    \draw[->] (h) -- (pon1n);
    \draw[->] (pon1n) -- (n1n);
    \draw[->] (hh) -- (pon1n);
    \node (compl) at (1,-1) {$\nu_1N/ \big( P_0 \cap \nu_1N\big)$};
    \draw[->] (i) -- (compl);
    \draw[->] (compl) -- (a);
    \draw[->] (n1n) -- (compl);
    \node (p0) at (2,-2) {$P_0$};
    \node (p0n1nn1n) at (3,-2) {$P_0\nu_1N/\nu_1N$};
    \node (holycr) at (3,-1) {$\frac{N/\nu_1N}{P_0\nu_1N/\nu_1N}$};
    \draw[->] (holycr) -- (d);
    \draw[->] (holycr) -- (c);
    \node (np0) at (2,-1) {$N/P_0$};
    \draw[->] (pon1n) --(p0);
    \draw[->] (p0) --(p0n1nn1n);
    \draw[->] (p0n1nn1n) -- (eh);
    \draw[->] (g) -- (p0);
    \draw[->] (p0) -- (N);
    \draw[->] (N) -- (np0);
    \draw[->] (np0) -- (b);
    \draw[->] (f) -- (p0n1nn1n);
    \draw[->] (p0n1nn1n) -- (nn1n);
    \draw[->] (nn1n) -- (holycr);
    \draw[->] (compl) -- (np0);
    \draw[->] (np0) -- (holycr);
  \end{tikzpicture}.
\end{equation}
Now note that by (a),(b) of Proposition \ref{prop:good-enough} there
are epimorphisms \[\nu_1N/(P_0\cap \nu_1N) \onto \nu_1N/N_0
\textrm{~and~} \frac{N/\nu_1N}{P_0\nu_1N/\nu_1N} \onto
\frac{N/\nu_1N}{K_0}.\] Moreover, since $P_0 \charleq N$ there are
well defined natural maps\begin{eqnarray*}
 \ol{r} : \Out(N) &\to & \Out(\nu_1N/(P_0\cap \nu_1N))
 \\
 \ol{p} : \Out(N) &\to & \Out(\frac{N/\nu_1N}{P_0\nu_1N/\nu_1N}).
\end{eqnarray*}
It is now possible to check that for $[\beta] \in \outo(N)$ if either
$r([\beta])\neq 1$ or $p([\beta])\neq 1$, since $K_0,N_0$ separated
torsion in $\Out(N/\nu_1N),\Out(\nu_1N)$ respectively, then $[\beta]$
will not vanish in the natural homomorphism \[\Out(N) \to
\Out(N/P_0).\]
\end{proof}

\subsubsection{An algebraic characterization of elusive elements}\label{sec:elusive-characterization}

It turns out that elusive elements have a very natural algebraic
characterization. We will first describe some algebraic constructions
and then give a description of elusive elements in terms of these
constructions. Recall that $\nu_1N\leq\nu_2N \charleq N$ is the
maximal subgroup such that

\begin{equation}\label{eqn:nu2}\nu_2N/\nu_1N =
  Z(N/\nu_1N)
\end{equation}

\begin{lem}\label{lem:nu2-to-hom}
  Each $\xi \in \nu_2N$ induces a homomorphism $z_\xi:N \to \nu_1N$
  given by the mapping \[x \mapsto [x,\xi]. \] Moreover the
  mapping \[\Phi: \nu_2N \to \Hom(N,\nu_1N)\] is in fact a
  homomorphism where $\Hom(N,\nu_1N)$ is viewed as an abelian group
   equipped with the standard $\mathbb{Z}$-module addition.
\end{lem}
\begin{proof}
  By (\ref{eqn:nu2}) for all $x \in N, \xi \in \nu_2N$ we have
  $[x,\xi] = z_\xi(x) \in \nu_1N$. Let $x,y \in N$ then with the
  commutator convention $[x,y] = x^\mo y^\mo xy$ we can observe that
  on one hand $(xy)\xi = \xi(x,y) [(xy),\xi] =\xi (xy) z_\xi(xy)$ and
  on the other hand (recall that $[z,\xi]$ is always central):
  \[xy\xi = x \xi y [y,\xi] = \xi x [x,\xi] y [y,\xi] = \xi (x y)
  [x,\xi] [y,\xi] = \xi(xy)z_\xi(x)z_\xi(y)\] so the map $x \mapsto
  z_\xi(x)$ is a homomorphism.

  Let now $\xi,\zeta \in \nu_2N$ and let $x \in N$. On the one hand
  setting $[x,\xi\zeta] = z_{\xi\zeta}(x)$ we have $x(\xi\zeta) =
  (\xi\zeta)xz_{\xi\zeta}(x)$ and on the other hand we have: \[
  x\xi\zeta = \xi x z_\xi(x) \zeta = \xi x \zeta z_\xi(x) = (\xi \zeta)
  x z_\zeta(x)z_\xi(x) \] which gives the formula: \[ z_{\xi\zeta}(x) =
  z_\xi(x) + z_\zeta(x) \] so the map $\Phi$ is a homomorphism.
\end{proof} 

Because $\nu_1N$ is the center of $N$, it is clear that for any
$x \in \nu_1N, \xi \in \nu_2 N$ we have $\Phi(\xi)(x) =
[x,\xi]=1$. This next result can therefore be seen a converse of the
previous result.

\begin{lem}\label{lem:psi}
  Let \[\Hom^*(N,\nu_1N) =\{f \in \Hom(N,\nu_1N) \mid \nu_1N \leq
  \ker(f)\}.\] For each $f \in \Hom^*(N,\nu_1N)$ the map $x \mapsto
  xf(x)$ is an automorphism $\Psi(f) \in \Aut(N)$. Moreover this map
  $\Psi: \Hom^*(N,\nu_1N) \to \Aut(N)$ is a homomorphism.
\end{lem}
\begin{proof}
  We first check that $\Psi(f)$ is an endomorphism. Indeed
  $\Psi(f)(xy) = xyf(xy)$ and, since $f$ is a homomorphism and $f(x)
  \in \nu_1N$, we have \[xyf(xy) = xyf(x)f(y) = xf(x)yf(y) =
  \Psi(f)(x)\Psi(f)(y).\] We can also immediately check that by our
  hypothesis $\Psi(f) \circ \Psi(-f)=1$ so $\Psi(f)$ is an automorphism.

  Now note that because we are only interested in $f,g \in
  \Hom(N,\nu_1N)$ in which elements of $\nu_1N$ vanish we have the
  following\[\Psi(f) \circ \Psi(g)(x) = \Psi(f)(xg(x))
  = xg(x)f(x) = \Psi(g+f)(x)\] therefore $\Psi$ is a
  homomorphism.
\end{proof}

We leave the verification of this next lemma to the reader:

\begin{lem}\label{lem:psi-phi-composition}
For all $\xi$ in $\nu_2N$ we have the following equality of automorphisms:\[
\Psi\circ\Phi(\xi) = Ad_\xi. 
\]
\end{lem}

Having described the homomorphisms $\Psi:\Hom^*(N,\nu_1N) \to \Aut(N)$
and $\Phi:\nu_2N \to \Hom^*(N,\nu_1N)$ we can now give a precise
characterization of the elusive elements of $\outo(N)$.

\begin{prop}\label{prop:elusive-description}
  The set of elusive elements of $\outo(N)$ coincides exactly with the
  set \[ \big\{[\beta] \in \Out(N) \mid (\exists \beta \in [\beta])~
  \beta \in \Psi(\hat{S}\setminus S) \big\}\] where $S = \Phi(\nu_2N)$
  and $S \leq \hat{S}$ is the set consisting of $f \in
  \Hom^*(N,\nu_1N)$ such that \[d\cdot f = \underbrace{f+\cdots+f}_{d
    \textrm{~times}} \in S = \Phi(\nu_2N)\] for some $d \in
  \mathbb{Z}_{\geq 0}.$
\end{prop}

\begin{proof}
  \emph{We first show $\outo(N) \cap \ker(r) \cap \ker(p) \setminus \{1\} \subset
    \Phi(\hat{S}\setminus S)$, where $r,p$ are defined in
    (\ref{eqn:r}), (\ref{eqn:p}).}  Let $\alpha \in [\beta]$ be an
  automorphism. Consider the canonical projection:
  \begin{eqnarray*}
    N & \onto & N/\nu_1N\\
    w & \mapsto & \ol{w}.\\
  \end{eqnarray*}
  and denote by $\ol{\alpha}$ the natural image of $\alpha$ in
  $\Aut(N/\nu_1N)$. Since $[\beta] \in \ker(p)$, $\ol{\alpha}$ must in
  fact be an inner automorphism of $N/\nu_1N$ , so there exists some
  $u_\alpha \in N$ such that for all $w \in N$ \[ \ol{\alpha}(\ol{w})
  = \ol{u_\alpha}\ol{w}\ol{u_\alpha^\mo}.\] We therefore take \[\beta
  = Ad_{u_\alpha} \circ \alpha \in [\beta]\] where $Ad_x$ denotes
  conjugation by $x$ and we get that $\ol{\beta}=1$. This means that
  for all $w \in N$ we have \[\beta(w) = w z_\beta(w)\] where
  $z_\beta(w) \in \nu_1N$. Now since $\beta:N \to N$ is a homomorphism
  we have for all $u,w \in N$ \[w u z_\beta(w) z_\beta(u) =
  \beta(w)\beta(u) = \beta(wu) = wuz_\beta(wu)\] which means that
  $z_\beta:N \to \nu_1N \in \Hom(N,\nu_1N) $ and since $[\beta] \in
  \ker(r)$, we have that $\beta(\xi) = \xi$ for all $\xi \in \nu_1N$
  and therefore $z_\beta(\xi)=1$ so $z_\beta \in \Hom^*(N,\nu_1N)$
  and \begin{equation}\label{eqn:beta-phi-z}\beta =
    \Psi(z_\beta).\end{equation} Now, since $[\beta] \in \Out(N)
  \setminus \{1\}$, $z_\beta \not\in S$.

  Since $[\beta] \in \ker(r)$ we have that for all $\xi\in \nu_1N$
  $\beta(\xi)=\xi$, which
  gives \begin{equation}\label{eqn:powers}\beta^n(w) = w
    z_\beta(w)^n \end{equation} for all $w \in N, n \in \mathbb{Z}$.

  Finally since $[\beta] \in \outo(N)$, and say of order $d$, $\beta^d$
  must be an inner automorphism so there exists some $\xi_\beta$ such
  that for all $w \in N$
  \begin{equation}\label{eqn:order-d}
    \beta^d(w) =  Ad_\xi(w) =  w z_\beta(w)^d
  \end{equation}
  which implies $[w,\xi_\beta] = w^\mo \xi_\beta^\mo w \xi_\beta =
  z_\beta(w)^d$. This in particular implies that $\ol{\xi_\beta} \in
  Z(N/\nu_1N)$ so $\xi_\beta \in \nu_2N$ and by (\ref{eqn:order-d}) we
  have $\Phi(\xi_\beta) = d\cdot z_\beta$. This means that $z_\beta
  \in \hat{S} \setminus S$ and we have the first desired inclusion.

  \emph{We now prove that every element of $\Psi(\hat{S} \setminus S)$
    is elusive.} We first note that every automorphism in the image of
  $\Psi$ (see the statement of Lemma \ref{lem:psi}) fixes $\nu_1N$
  pointwise and therefore every corresponding outer automorphisms lies
  in $\ker(r)$. Similarly every such outer automorphism lies in
  $\ker(p)$. Now suppose that for some $s \in \hat{S}\setminus S$,
  $\Phi(s) = Ad_\xi$ for some $\xi \in N$. Then for all $w \in N$ we
  have \[\xi^\mo w \xi = w s(w) \Rightarrow [w,\xi] = s(w) \in
  \nu_1N \] which means that $\ol{\xi} \in Z(N/\nu_1N)$, i.e. $\xi \in
  \nu_2N$ so $s \in \Phi(\xi) = S$ -- contradiction. Therefore for all
  $f \in \hat{S}\setminus S$, $[\Psi(f)] \in \Out(N)\setminus \{1\}.$

  Finally since $\Psi$ is a homomorphism and since for every $s \in
  \hat{S} \setminus S$ there is some minimal $d \in \mathbb{Z}$ such
  that $d\cdot s \in S$ we have \[ \Psi(d\cdot s) = \Psi(s)^d =
  Ad_\xi\] where $\xi \in \nu_2N$ so the element $[\Psi(s)] \in
  \Out(N)\setminus \{1\}$ has order $d$. We have therefore shown the
  other inclusion.
\end{proof}

The following is a basic exercise in $\Z$-modules that we leave to the
reader.

\begin{lem}\label{lem:hatS-dir-sum}
  The submodule $\hat S \leq \Hom^*(N,\nu_1N)$ given in Proposition
  \ref{prop:elusive-description} is the smallest direct summand of
  $\Hom^*(N,\nu_1N)$ containing $S$ and the torsion subgroup.
\end{lem}

\subsubsection{Computing the list of conjugacy classes of elusive
  elements}\label{sec:elusive-computation}

For this section as in the statement of Proposition
\ref{prop:elusive-description}, let $S$ denote the image of
$\Phi(\nu_2N) \leq \Hom^*(N,\nu_1N)$ and let $\hat{S}$ denote smallest
direct summand of $\Hom^*(N,\nu_1N)$ containing $S$ and the torsion
elements. Since we are working over finitely generated $\bbZ$ modules
it follows that $[\hat{S}:S] < \infty.$

\begin{prop}\label{prop:coset-list}
  There is a surjection from the set of cosets $\hat{S}/S$ onto the
  set of elusive elements in $\Out(N)$. In particular there are
  finitely many elusive elements.
\end{prop}
\begin{proof}
  Let $\xi \in \nu_2N$ and let $t = \Phi(\xi) \in S$ and let $x\in
  \hat{S}$ be arbitrary. Then we need only verify the
  formula \[\Psi(t+x) = Ad_\xi \circ \Psi(x).\] Indeed we have for all
  $w \in N$
  \begin{eqnarray*}
    \Psi(t+x)(w) &= & wt(w)x(w)\\
    & = & \xi^\mo w \xi x(w)\\
    & = & \xi^\mo w x(w) \xi \\
    & = & Ad_\xi \circ \Psi(x)(w)
  \end{eqnarray*}
  It follows that if in $\hat{S}$, $s \equiv s' \mod S$ then
  $[\Psi(s)]=[\Psi(s')]$. The result now follows from Proposition
  \ref{prop:elusive-description}.
\end{proof}

This next result is an immediate consequence of Theorems 4.1 and 6.5
of \cite{BCRS-PF}. 

\begin{prop}\label{prop:find-gens}
  Given a presentation $\bk{X\mid R}$ of a $\fgn$-group $N$, we can
  find finite generating sets $\{u_i(X)\}$ and $\{v_j(X)\}$ for
  $\nu_1N$ and $\nu_2N$ respectively. Moreover we can construct an
  explicit isomorphism of the finitely generated abelian group \[\nu_1N
  \stackrel{\sim}{\rightarrow} \Z^s\oplus\left(\oplus_{i=1}^t\Z/n_i\Z
  \right).\]
\end{prop}

\begin{cor}\label{cor:construct-elusive}
  We can construct a finite set $\{\beta_1,\ldots,\beta_v\} \in
  \Aut(N)$ such that $\{[\beta_1],\ldots,[\beta_v]\} \subset
  \outo(N)\setminus \{1\}$ is the complete set of elusive
  elements.
\end{cor}
\begin{proof}
  We first note that since $\nu_1N\approx
  \mathbb{Z}^s\oplus\left(\oplus_{i=1}^t\Z/n_i\Z \right)$ is finitely
  generated abelian we have a natural isomorphism $\Hom(N,\nu_1N)
  \approx \Hom(N/[N,N],\nu_1N)$. Seeing as $N/[N,N] =
  \Z^r\oplus\left(\oplus_{i=1}^s\Z/m_i\Z \right)$ is also a finitely
  generated $\bbZ$-module it follows that $\Hom(N/[N,N],\nu_1N)$ is
  also a finitely generated $\bbZ$-module. 

  The isomorphism \[ \rho: \Hom(N,\nu_1N) \stackrel{\sim}{\to}
  \Hom\left(\Z^r\oplus\left(\oplus_{i=1}^s\Z/m_i\Z
    \right),\mathbb{Z}^s\oplus\left(\oplus_{i=1}^t\Z/n_i\Z
    \right)\right)\] is computable in the following sense: given a map
  $f^*$ from the generators of $N$ to the generators of $\nu_1N$
  defining a homomorphism $f$, we can find the corresponding element
  $\rho(f)$ and conversely given some $x \in
  \Hom\left(\Z^r\oplus\left(\oplus_{i=1}^s\Z/m_i\Z
    \right),\mathbb{Z}^s\oplus\left(\oplus_{i=1}^t\Z/n_i\Z
    \right)\right)$ we can find a map $\rho^\mo(x)^*$ from the
  generators of $N$ to $\nu_1N$.

  By Proposition \ref{prop:find-gens} we can find a generating set for
  $\nu_1 N$ and therefore find the submodule
  $\rho(\Hom^*(N,\nu_1N))=T$. Similarly we can  find the
  submodule $\rho(\Phi(\nu_2N)) = \rho(S) \leq T$. We can then compute
  $U \geq \rho(S)$, the minimal direct summand of $T$ containing
  $\rho(S)$ and the torsion elements of $T$. We must have $U =
  \rho(\hat S)$. We can now finally find a complete finite list of
  representatives $\{b_1,\ldots,b_n\} \subset
  \Hom\left(\Z^r\oplus\left(\oplus_{i=1}^s\Z/m_i\Z
    \right),\mathbb{Z}^s\oplus\left(\oplus_{i=1}^t\Z/n_i\Z
    \right)\right)$ of cosets of $U/\rho(S)$. The result now follows
  from Proposition \ref{prop:coset-list}.
\end{proof}

\subsubsection{The proof of Theorem \ref{thm:find-deep-enough}}\label{sec:find-cong-sep}

The final ingredient we need is an argument graciously communicated to us by Dan Segal.

\begin{prop}\label{prop:segal}
  Let $P$ be a virtually polycyclic group and denote by $P_n\charleq
  P$ be a sequence of finite index subgroups which eventually lie
  inside any fixed finite index subgroup. For every finite order
  $[\alpha] \in \out P$ there exists some $j$ such that for every
  $k\geq j$ the image $\overline{[\alpha]}_k \in \out{P/P_k}$ is
  non-trivial.
\end{prop}

\begin{proof}
  Let $q$ be the order of $[\alpha]$ in $ \out{P}$ and chose a
  representative $\alpha\in \Aut P$. We assume that $q > 1$. Let $E=
  \bk{\alpha} \ltimes P$ be the natural semidirect product,
  i.e. endowed with the following
  multiplication \begin{equation}\label{eqn:semidirect-conv}
    (\alpha^{n_1};p_1)\cdot(\alpha^{n_2};p_2) = (\alpha^{n_1+n_2};
    \alpha^{n_2}(p_1)p_2),\end{equation} where $p_1,p_2 \in P$. $E$ is
  obviously virtually polycyclic.

  Suppose towards a contradiction that $\overline{[\alpha]}_n$ (which
  is in $\out{P/P_n}$) is trivial for infinitely many $n$. Then there
  exists $\bar x_n \in P/P_n$ such that the image of $\alpha$ in
  $\aut{P/P_n}$ satisfies $\overline{\alpha}_n = \ad_{{\bar x}_n}$.
  This means that $\alpha(g) \in g ^{x_n^{-1}}P_n$ for all $g\in P$.
  Thus in $E$, given any $(1;g)\in 1 \ltimes P$, the commutator with
  the element $(\alpha,x_n)$ lies in $1\ltimes P_n$. 
  Indeed for each
  $g \in P$ we verify by successive applications of
  (\ref{eqn:semidirect-conv}):
  \begin{eqnarray*}
    [(1;g),(\alpha;x_n)] & = &
    (1;g^{-1}) \cdot (\alpha^{-1};\alpha^{-1}(x_n^{-1})) \cdot (1;g)
    \cdot (\alpha;x_n)\\
    & = & (1;g^{-1})\cdot (\alpha^{-1};\alpha^{-1}(x_n^{-1})) \cdot
    (\alpha;\alpha(g)x_n) \\
    & = &  (1;g^{-1})\cdot (1;x_n^{-1}\alpha(g)x_n)\\
    & = & (1;g^\mo x_n^{-1} \alpha(g) x_n ) \in \{1\} \ltimes P_n
  \end{eqnarray*}

  Go to the projective limit of finite quotients (i.e. the profinite
  completion), up to diagonal extraction, there exists $x \in \bar P$
  such that $x \mapsto \bar x_n \in P/P_n $ for infinitely many $n$.

  By definition, $(\alpha;x) \in \bar E$ centralizes $\{1\} \ltimes P$
  in $\bar E$, i.e. $(\alpha;x) \in C_{\bar E}(\{1\} \ltimes P)$. We now apply a
  deep result, \cite[Prop 3.3a]{RSZ}, to infer that in fact
  $(\alpha;x) \in \overline{ C_E(\{1\} \ltimes P)} $.

  Moreover,
  $C_E(\{1\} \ltimes P) < \langle \alpha^q \rangle \ltimes P$. Indeed,
  for conjugation we have the
  formula\[ (\alpha^{-m};\alpha^{-m}(y^{-1}))(1;g) (\alpha^m;y) =
    (1;\ad_y\circ\alpha^m(g) )\] so if $(\alpha^m;y)$ centralizes
  $\{1\} \ltimes P$, then the image $[\alpha^m]$ is the identity in
  $\out P$, which implies that $q|m$. This now gives us that on the
  one hand
  $(\alpha;x) \in \overline{\langle \alpha^q \rangle \ltimes P}$.

  On the other hand,
  $(1;x) \in \overline{\langle \alpha^q \rangle \ltimes P}$. Since
  $\overline{\langle \alpha^q \rangle \ltimes P}$ is a subgroup,
  $(\alpha;1)$ also lies in
  $\overline{\langle \alpha^q \rangle \ltimes P}$.  Since
  $(\alpha;1) \in \langle \alpha \rangle \ltimes P=E$, we have
  \[ (\alpha;1) \in \overline{\langle \alpha^q \rangle \ltimes P} \cap
    E = \bk{\alpha^q}\ltimes P.
  \] This last equality holds because $\bk{\alpha^q}\ltimes P$ a
  finite index subgroup of $E$, and
  therefore a closed in the profinite topology. It follows that $q=1$
  -- contradiction.
\end{proof}

\begin{proof}[Proof of Theorem \ref{thm:find-deep-enough}]
  We proceed by induction on $\lucs(N)$. If $\lucs(N)=1$ then $N$ is
  finitely generated abelian and by the remark \ref{rem;GLnZ3}
  we can find $P \charleq N$ such that the map $\Out(N) =
  GL(N) \onto GL(N/P)$ doesn't kill any finite order elements.

  Suppose now that $\lucs(N) = m$ and that the result held for all
  $\fgn$-groups $M$ such that $\lucs(M) < m$, i.e. that we can
  compute a finite index subgroup $L\charleq M$ that is deep
  enough. We have a short exact sequence\[ 1 \to \nu_1N \to N \to
  N/\nu_1N \to 1
  \] with $\lucs(\nu_1N)=1$ and $\lucs(N/\nu_1N) = m-1$. By the
  induction hypothesis we can construct subgroups $N_0 \charleq
  \nu_1N$ and $K_0 \charleq N/\nu_1N$ that separate torsion in
  $\out{\nu_1N}$ and $\Out(N/\nu_1N)$ respectively.

  The elements of $\outo(N)$ are either elusive or they aren't. By
  Proposition \ref{prop:good-enough} we can compute $P_0 \charleq N$
  that is good enough w.r.t. $N_0,K_0$ so by Lemma
  \ref{lem:good-enough-elusive} all the non elusive elements $[\beta]
  \in \outo(N)$ survive in the natural map $\Out(N) \to
  \Out(N/P_0)$. By Corollary \ref{cor:construct-elusive} we can
  construct a finite set $\{\beta_1,\ldots,\beta_n\}$ of automorphisms
  such that $\{[\beta_1],\ldots,[\beta_n]\} \subset \outo(N) \setminus
  \{1\}$ is the complete set elusive elements. 
  
  We now start iteratively constructing a chain of characteristic
  finite index subgroups
  \[ N \chargeq P_0 \chargeq N_1 \chargeq N_2 \chargeq \ldots
  \] such that every finite index subgroup of $N$ contains some
  tail of the sequence. We can now apply Proposition \ref{prop:segal}
  to get that for all $j$ sufficiently large, the set
  $\{\beta_1,\ldots,\beta_n\}$ is mapped monomorphically via the map
  $\out{N} \to \out{N/N_j}$. Since $P_0 \chargeq N_j$ we have that
  non-elusive elements of $\outo(N)$ survive in
  $\Out(N) \to \Out(N/N_j)$, therefore for each new $N_i$ we
  construct, it is enough to check that each element of our finite set
  $\{[\beta_1],\ldots,[\beta_n]\}$ survives in
  $\Out(N) \to \Out(N/N_i)$. Eventually this will be the case and we
  will have found $P \charleq N$ that separates the torsion in
  $\out N$.
\end{proof}

\subsection{The mixed Whitehead problem for finitely generated nilpotent groups}\label{sec:mwhp-fgn}

Let $G$ be a finitely generated nilpotent group. We shall denote a
tuple $S=(s_1,\ldots,s_r)$. If $S$ is a tuple of elements of a group
$H$ we shall abuse notation and write $S \in H$. If
$S=(s_1,\ldots,s_r), T=(t_1,\ldots,t_r)$ are tuples in $H$ and $h \in
H$ we will write $S^h = T$ if\[ h^\mo s_i h = t_i
\] for $i=1,\ldots,r$. If $\sigma \in \aut H$ and
$S=(s_1,\ldots,s_r) \in H$ then we denote\[ \sigma(\tuple s) =
(\sigma(s_1),\ldots,\sigma(s_r)).
\]

\begin{defn}[The mixed Whitehead problem]\label{def:mwhp}
  Let $(S_1,\ldots,S_k), (T_1,\ldots,T_k)$ be tuples of tuples of elements in
  $G$. \define{The mixed Whitehead problem} consists in deciding
  whether there exists $\sigma \in \aut G$ and elements
  $g_1,\ldots,g_k \in G$ such that\[ \sigma(S_i) = T_i^{g_i}
  \] for $i = 1,\ldots,k$. If such is the case we say the tuples of
  tuples $(S_1,\ldots,S_k)$ and
  $(T_1,\ldots,T_k)$ are \define{Whitehead equivalent}.
\end{defn}

The terminology mixed Whitehead problem first occurs in
\cite{Bogopolski-Ventura} where the mixed Whitehead problem is solved
for torsion-free hyperbolic groups. The goal of this section is to
give an algorithm that solves the mixed Whitehead problem for finitely
generated nilpotent-groups.

It should be noted that the instances of the mixed Whitehead problem
for a single pair of tuples, i.e. find if there is some $\sigma \in
\Gamma$ such that $\sigma(\tuple s) = \tuple t$, are solvable by the
results in \cite{GS2} and applying Algorithm A of \cite{GS1}. The
situation here is more complicated because it is a \emph{two
  quantifier problem.} We will remedy this by re-expressing the mixed
Whitehead problem as an orbit problem for a certain \emph{semidirect
  product}. As in \cite{GS1,GS2} we must enter the realm of matrix
groups.

\subsubsection{Algebraic groups, arithmetic groups, $\exp, \log$, etc.}

We denote by $\tffgn$ the class of \define{finitely generated
  torsion-free nilpotent groups} and call such groups
\define{$\tffgn$-groups.}

We will always assume that the inclusions $GL(n,\Z) < GL(n,\Q) <
GL(n,\C)$ hold. There is a natural identification of the set of $n
\times n$ matrices with $\C$-coefficients with $\C^{n^2}$, which we
may regard as affine space. $GL(n,\C)$ can therefore be viewed as an
open algebraic subvariety of $\C^{n^2}$, i.e. the set of matrices whose
determinant is non-zero.

If $S$ is some subset of $ \C[x_1,\ldots,x_n]$ we denote by $V(S)$ the
corresponding algebraic variety and if $V$ is an algebraic variety we
denote by $I(V)$ the corresponding ideal.

\begin{defn}[Algebraic group]\label{defn:alg-group}
  A subgroup $\alg H \leq GL(n,\C)$ is called \define{algebraic} if it
  is a closed algebraic subvariety $V(S)$, i.e. it consists of
  matrices whose entries satisfy some set $S$ of polynomial equations.
  \begin{itemize}
  \item If $S$ is a set of polynomials that have $\Q$-coefficients then
    $\alg H$ is a \define{$\Q$-defined algebraic group.}
    \item If $S$ is explicitly given then $\alg H$ is an
      \define{explicitly given $\Q$-defined algebraic group.}
  \end{itemize}
\end{defn}

\begin{defn}[Arithmetic group]\label{defn:arith-group}
  Write $\alg H_\Z = \alg H \cap GL(n,\Z)$. A subgroup $\Delta \leq
  GL(n,\C)$ that is commensurable with $\alg H_\Z$ for some
  $\Q$-defined algebraic group $\alg H$ is called an
  \define{arithmetic group.} We say that $\Delta \leq \alg{H}_\Z$ is
  \define{an explicitly given arithmetic subgroup} if we are given the
  following:\begin{enumerate}
  \item We are given the system of polynomials $S$ so that $\alg H = V(S)$.
  \item We are given a finite upper bound for the index $[\alg H_\Z:\Delta] \leq k$.
  \item There is an effective procedure so that for each $g \in \alg H_\Z$
    we can decide if $g \in \Delta$.
  \end{enumerate}
\end{defn}

Let $\tron{R}$ denote the set of upper triangular $n\times n$ matrices
with coefficients in the ring $R$ with 1s on the diagonal. We observer
that $\tron\Z$ and $\tron\C$ are arithmetic and algebraic groups
respectively. The connection to \tffgn-groups starts with the
following result which is a rewording and weakening of Algorithm E and
the supplement to Algorithm E in \cite{GS2}.

\begin{thm}[\cite{GS2}]\label{thm:theta-functor}
  Given a finite presentation $\bk{X\mid R}$ of a $\tffgn$-group $G$
  we can effectively find a suitable $n$ and an embedding \[
  \Theta_G:G \hookrightarrow \tron\Z 
  \] such that the natural map \[
  N_{GL(n,\Z)}(\Theta_G(G)) \to \aut G
  \] is surjective.
\end{thm}

\begin{defn}
  A subgroup $H \leq \tron\Z$ which occurs as the image $\Theta_G(G)$
  of some $\tffgn$-group is called a $\Theta$-subgroup.
\end{defn}

The discussion that will now follow is essentially contained in
Chapter 6 of \cite{Segal-polycyclic-groups}. The reader may refer to
this chapter for more details and proofs.

We denote by $\trzn R$ the set of upper triangular $n\times n$ matrices
with coefficients in the ring $R$ with 0s on the diagonal. $\trzn R$
equipped with matrix addition and the Lie bracket\[
(u,v) = uv-vu
\]
is an $R$-Lie algebra. There are celebrated maps $\exp:\trzn\C \to
\tron\C$ and $\log:\tron\C \to \trzn\C$ which are in fact
\emph{mutually inverse isomorphisms in the category of algebraic
  varieties.}  Moreover these restrict to bijections between $\tron\Q$
and $\trzn\Q$.

\begin{defn}
  A subgroup $H \leq \tron\Q$ is called \define{radicable} if for
  every $h \in H, n \in \Z_{\geq 1}$ there is some element $\sqrt[n]h$
  such that $(\sqrt[n]h)^n = h$.
\end{defn}

It is a result of Mal'cev \cite{Malcev-tfn} that a \tffgn-group $H$
embeds into a unique up to isomorphism \define{radicable hull}, or
\define{Mal'cev completion}. We denote this by $\malcev H$. A proof of
this next result follows from the discussion in Chapter 6 of
\cite{Segal-polycyclic-groups} leading up to and including Theorem 2.

\begin{thm}[\cite{Segal-polycyclic-groups}]\label{thm:readicable-corresp}
  $\log$ sends radicable subgroups of $\tron\Q$ to Lie
  subalgebras, i.e. linear subspaces closed under the Lie bracket, of
  $\trzn\Q$ bijectively. $\exp$ does the inverse.
\end{thm}

This bijective correspondence is far deeper than merely
set-theoretic. Let $H \leq \tron\Q$ and let $\Q\log(H)$ be the
smallest $\Q$-Lie subalgebra of $\trzn\Q$ containing $\log(H)$. Then
(see \cite[Chapter 6, Theorem 2]{Segal-polycyclic-groups}) \[
\exp(\Q\log(H)) = \malcev H.
\] In particular $\Q\log(H)$ is a $\Q$-linear subspace of $\trzn\Q$
and it follows that:

\begin{prop}\label{prop:Malcev-algebraic}
  For any group $H \leq \tron\Q$ there is a $\Q$-defined algebraic
  group $\alg H$ such that:\[
  \alg H \cap GL(n,\Q) = \alg{H}_\Q = \malcev H.
  \]
\end{prop}

We now have the following result which immediately follows from the
results in \cite[\S 9]{GS2} up to Lemma 9.1.3.

\begin{lem}[\cite{GS2}]\label{lem:arith-crit}
  There is an algorithm which takes as input $\bk{X \mid R}$ a finite
  presentation of a $\tffgn$-group $G$ and outputs a positive integer
  $m_H$ such that if $H = \Theta_G(G) \leq \tron\Z$ is the isomorphic
  image constructed by Theorem \ref{thm:theta-functor} then the
  following holds: $g \in H$ if and only if $g \in \alg{H}_\Z$ and
  $\pi_{m_H}(g) \in \pi_{m_H}(H)$ where $\pi_{m_H}$ is the canonical epimorphism\[
  \pi_{m_H}:GL(n,\Z) \twoheadrightarrow GL(n,\Z/m_{H}\Z)
  \]
\end{lem}

\begin{cor}\label{cor:tffgn-explicit-arithmetic}
  There is an algorithm which takes as input $\bk{X \mid R}$ a finite
  presentation of a $\tffgn$-group $G$ and represents $H=\Theta_G(G)$
  as an explicitly given arithmetic group.
\end{cor}
\begin{proof}
  We abuse notation and consider $G=H$. We can find a generating set
  $X' = \bk{x_1,\ldots,x_h}$ of $H$ which witnesses the fact that $H$
  is polycyclic (here $h$ is the Hirsh length.) Equipped with $X'$,
  the maps $\Theta_G$ and $\log$ we can find elements
  $\log(x_1),\ldots,\log(x_n)$ which generate $\Q\log(H)$ (see
  \cite[Lemma 9.1.1]{GS2}) as a $\Q$-Lie subalgebra so we can
  effectively find a system of linear equations $S$ which define the
  algebraic variety $\C\log(H)$. Seeing as $\exp:\trzn\C \to \tron\C$
  is an isomorphism of algebraic varieties with inverse $\log$ (both
  explicitly given) we can effectively find a system of polynomial
  equations $\exp_*(S)$ which define the algebraic group
  $\exp(\C\log(H)) = \alg H \leq GL(n,\C)$. This gives item 1. of
  Definition \ref{defn:arith-group}.

  Items 2. and 3. of Definition \ref{defn:arith-group} follow
  immediately from the criterion of Lemma \ref{lem:arith-crit}.
\end{proof}

Another Corollary (this is essentially 
\cite[Lemma 9.1.4]{GS2}) of Lemma \ref{lem:arith-crit} is:
\begin{cor}\label{cor:normalizer-crit}
  Let $H\leq \tron\Z$ be a $\tffgn$-group and let $m_H$ be as in Lemma
  \ref{lem:arith-crit}. For any $g \in GL(n,\Z)$ the following are
  equivalent:\begin{itemize}
  \item $g^\mo H g = H$.
  \item $g^\mo \alg{H}_\Z g = \alg{H}_\Z$ and
    $\pi_{m_H}(g^\mo)\pi_{m_H}(H)\pi_{m_H}(g) = \pi_{m_H}(H)$.
  \end{itemize}
\end{cor}

From which we deduce:

\begin{prop}\label{prop:arithmetic-automorphisms}
  There is an algorithm which takes as input $\bk{X \mid R}$ a finite
  presentation of a $\tffgn$-group $G$ and represents the
  $GL(n,\Z)$-normalizer $N_{GL(n,\Z)}(H)=\hat\Gamma$ of $H=\Theta_G(G)\leq
  GL(n,\Z)$ as an explicitly given arithmetic group.
\end{prop}
\begin{proof}
  There is an explicitly given algebraic group $\alg A(H)\leq GL(m,\C)$
  given in (7) on page 613 of \cite{GS2}, which
  satisfies: \[ \alg A(H)_\Z = N_{GL(n,\Z)}(\alg{H}_\Z).
  \] By Corollary \ref{cor:normalizer-crit} the explicitly given
  arithmetic subgroup\[ \hat\Gamma = \{g \in \alg A(H)_\Z \mid
  \pi_{m_H}(g^\mo)\pi_{m_H}(H)\pi_{m_H}(g) = \pi_{m_H}(H)\} \] is
  exactly the normalizer $N_{GL(n,\Z)}(H)$.
\end{proof}

\subsubsection{The mixed Whitehead problem as an orbit problem}\label{sec;MWP_orbit_pb}

We first fix some notation. Let $H,K$ be groups and
let \begin{eqnarray*} \phi:K &\to&
  \aut H\\
  k & \mapsto & \phi_k
\end{eqnarray*}
 be a homomorphism. We denote the associated right semidirect
product $K \semidirect_\phi H$ with the multiplication rule\[
(k_1;h_1)\cdot(k_2;h_2) = (k_1k_2;\phi^\mo_{k_2}(h_1)h_2)
\] It is a standard result that if some group $G$ has subgroups $H,K$
with $H \nsgp G$, $H \cap K = \{1\}$, and $G=KH$ then $G \approx K
\semidirect_\phi H$ where $\phi$ is induced by conjugation.

Let $G$ be some group and let $\Gamma = \aut G$. There is a well
defined right action of the right semidirect product $\Gamma
\semidirect G^r$ on the set of $r$-tuples of tuples given
by\begin{equation}\label{eqn:semidirect-action} 
  (S_1,\ldots,S_r)\cdot \big(\sigma;(g_1,\ldots,g_r) \big) =
  (\sigma^\mo(S_1)^{g_1},\ldots,\sigma^\mo(S_r)^{g_r}).\end{equation}
It immediately follows that

\begin{prop}\label{prop:orbit-reduction}
  Let $(S_1,\ldots,S_k), (T_1,\ldots,T_k)$ be tuples of elements in
  $G$. They are Whitehead equivalent if and only if they lie in the
  same orbit under the $\Gamma \semidirect G^r$-action given in
  (\ref{eqn:semidirect-action}.)
\end{prop}

Our goal is to use the Grunewald-Segal orbit algorithm (we slightly
changed the terminology to make it consistent with this paper):

\begin{thm}[{\cite[Algorithm A]{GS1}}]\label{thm:orbit-algorithm}
  There exists and algorithm which takes as input:\begin{itemize}
  \item an explicitly
    given $\Q$-defined algebraic group $\alg G$, 
  \item$\rho$, an explicitly  given rational action on some subset $W
    \subset \C^n$,
  \item $\Gamma$, an explicitly given arithmetic subgroup of $\alg G$,
    and 
  \item two points $a,b \in W \cap \Q^n$
\end{itemize}
The algorithm decides whether there is some $\gamma \in \Gamma$ such
that $\rho(\gamma)\cdot a = b$ and if so produces such a matrix
$\gamma$.
\end{thm}

\begin{rmk}
  We will not give the full definition of a rational action but the
  only example we will consider is the action of $GL(n,\C)$ on
  $GL(n,\C) \subset \C^{n^2}$ given by conjugation.
\end{rmk}

\begin{defn}
  Let $g_0,\ldots,g_m$ be $n\times n$ matrices. The $mn \times mn$
  matrix $\diag{g_0,\ldots,g_m}$ is the corresponding block-diagonal
  matrix.
\end{defn}

Let $H \leq GL(n,\C)$ and let $K \leq N_{GL(n,\C)}(H)$. Then we have a
natural map $\phi: K \to \aut H$ given by $\phi_k(h) = khk^\mo$.

\begin{lem}\label{lem:linear-semidirect}
  Let $H,K,\phi$ be as above. Then the set of matrices\[ T =
  \{\diag{k,kh_1,\ldots,kh_r} \mid k \in K; h_1,\ldots,h_n \in H\}
  \] equipped with the induced multiplication is isomorphic to $K
  \semidirect_{\phi^r} H^r$ via \begin{equation}\label{eqn:semidirect-to-linear}
  (k;(h_1,\ldots,h_r)) \mapsto \diag{k,kh_1,\ldots,kh_r}.
\end{equation}
\end{lem}
\begin{proof}
  We first note that
  \begin{eqnarray*}
    \diag{k,kh_1,\ldots,kh_r}\cdot\diag{k',k'h'_1,\ldots,k'h'_r}
    & = & \diag{kk',kh_1k'h'_1,'\ldots,kh_rk'h'_r}\\
    & = & \diag{kk', kk'\phi^\mo_{k'}(h_1)h_1',\ldots,kk'\phi^\mo_{k'}(h_r)h'_r}.
  \end{eqnarray*}
  The map (\ref{eqn:semidirect-to-linear}) is therefore a
  homomorphism, which is clearly bijective.
\end{proof}

\begin{conv}
  If, $H,K$ are as above we will abuse notation and denote the linear
  group $T$ given in Lemma \ref{lem:linear-semidirect} simply as
  $K\semidirect H^r$.
\end{conv}

The following is obvious
\begin{lem}\label{lem:explicit-action}
There is a well defined right rational action of $K
\semidirect H^r$ on $\bigoplus_{i=1}^r H^{r_i}$ given by\begin{eqnarray*}
\left(\oplus_{i=1}^r (h_{i,1},\ldots,h_{i,n_i})\right) \cdot \diag{k,kh_1,\ldots,kh_r} 
&=& \oplus_{i=1}^r (h_i^\mo k^\mo h_{i,1} kh_i,\ldots,h_i^\mo k^\mo
h_{i,n_i}kh_i)\\
& =& \oplus_{i=1}^r (\phi^\mo_k(h_{i,1})^{h_i},\ldots,\phi^\mo_k(h_{i,n_i})^{h_i}).
\end{eqnarray*} 
\end{lem}

The final missing ingredient is an explicit arithmeticity result

\begin{lem}\label{lem:fi-semidirect}
  Let $G=K\semidirect_\phi H = KH$ and let $H'<H,K'<K$ be such that
  $[K:K']=d, [H:H']=e$. Suppose moreover that $H'$ is
  $\phi(K')$-invariant. Then $K'H' = G'$ is a subgroup that satisfies
  $[G:G'] \leq ed$.
\end{lem}
\begin{proof}
  Since $G/H \approx K$ and $[K:K'] = d$ we immediately get
  $[KH:K'H]=d$. Now we have cosets\[
  H = H'h_1 \sqcup \cdots \sqcup H'h_e
  \] which means that \[K'H = K'H'h_1 \cup \cdots \cup K'H'h_e\] so
  $[K'H:K'H'] \leq e$ and the desired inequality follows.
\end{proof}

\begin{prop}\label{prop:explicit-semidirect}
  Let $H\leq \tron\Z$ be a $\Theta$-group, obtained from a finite
  presentation $\bk{X \mid R}$, and let $\hat\Gamma =
  N_{GL(n,\Z)}(H)$. Then the semidirect product $\hat\Gamma
  \semidirect H^r$ can be realized as an explicitly arithmetic
  subgroup of $GL((r+1)n,\C)$.
\end{prop}

\begin{proof}
  Let $S_{\alg H}$ and $S_{\alg A(H)}$ be the explicit systems of
  polynomial equations defining the $\Q$-algebraic groups $\alg H$ and
  $\alg A(H)$ containing $H$ and $\hat\Gamma$ as arithmetic subgroups
  respectively. Let $\alg S$ be the set of matrices
  $\diag{g_0,\ldots,g_r}$ such that the entries of $g_0$ satisfy
  $S_{\alg A(H)}$ and for $i=1,\ldots,r$ the entries of $g_0^\mo g_i$
  satisfy $S_{\alg H}$. It immediately follows from the definition
  that \[\alg S_\Z = \alg A(H)_\Z \semidirect \alg H_\Z^r.\] By Lemma
  \ref{lem:fi-semidirect} we have the upper bound $[\alg
  S_\Z:\hat\Gamma \semidirect H^r] \leq e^rd$ where
  $[A(H)_\Z:\hat\Gamma]=d$ and $[\alg H_\Z,H]=e$. Finally given some
  $\diag{\alpha,g_1,\ldots,g_n} \in \alg S_\Z$ we can decide if $\alpha \in
  \hat\Gamma$ using the criterion of Corollary
  \ref{cor:normalizer-crit} and we can decide whether each $\alpha^\mo
  g_1 \in H$ with the criterion given by Lemma \ref{lem:arith-crit}.
\end{proof}

We now give our solution to the mixed Whitehead problem for
$\tffgn$-groups.

\begin{prop}\label{prop:mwhp-tffgn}
There is a uniform algorithm to solve the mixed Whitehead problem for
$\tffgn$-groups.
\end{prop}

\begin{proof}
  We are given a finite presentation $\bk{X \mid R}$ of a
  $\tffgn$-group $G$ and a pair of tuples of tuples
  $(S_1,\ldots,S_k),
  (T_1,\ldots,T_k)$.\\

  \noindent \emph{Step 1:} Construct the embedding $\Theta_G: G \to GL(n,\Z)$
  given by Theorem \ref{thm:theta-functor} and represent $H =
  \Theta_G(G)$ and $\hat\Gamma = N_{GL(n,\Z)}(H)$ explicitly as
  arithmetic groups.\\

  \noindent \emph{Step 2:} Via $\Theta_G$ represent the tuples of
  tuples $(\tuple{s}_1,\ldots,\tuple{s}_k),
  (T_1,\ldots,T_k)$ as points $s,t$ in
  $\bigoplus_{i=1}^k GL(n,\Z)^{n_i}$.\\

  \noindent \emph{Step 3:} Represent $\hat\Gamma \semidirect H^k$
  explicitly as an arithmetic subgroup of $GL((k+1)n,\C)$.\\

  By Theorem \ref{thm:theta-functor} since the natural map $\hat\Gamma
  \to \aut H$ is \emph{surjective}, Proposition
  \ref{prop:orbit-reduction} implies that
  $(S_1,\ldots,S_k)$ and $(T_1,\ldots,T_k)$ are
  Whitehead equivalent \emph{if and only if} the points $s,t$ are in
  the same $\hat\Gamma \semidirect H^r$ orbit via the action given in
  Lemma
  \ref{lem:explicit-action}.\\

  \noindent \emph{Step 4:} Use \cite[Algorithm A]{GS1}  to
  decide if $s,t \in \bigoplus_{i=1}^k GL(n,\Z)^{n_i}$ are in the same
  $\hat\Gamma \semidirect H^r$-orbit via the action given in  Lemma
  \ref{lem:explicit-action}.
\end{proof} 

\subsubsection{Finitely generated nilpotent groups with torsion}

If $G$ is a finitely generated nilpotent group, i.e. an
$\fgn$-group. Then the approach is essentially the same. We will
construct an embedding $G$ in to some $GL(m,\bbC)$ as an arithmetic
group and this embedding will be such that conjugation gives a
surjection $N_{GL(m,\bbC)}(G) \onto \aut G$, the latter also being
explicitly given as an arithmetic group. The proof will then go
through immediately as in the proof of Proposition \ref{prop:mwhp-tffgn}.

We now follow \cite[\S 7]{GS2}. Let $G$ be a $\fgn$-group. It is
classical that $\tau(G)$, the set of finite order elements of $G$, in
fact form a finite characteristic subgroup of $G$ and that $\bar G =
G/\tau(G)$ is a $\tffgn$-group.

\begin{lem}[{see \cite[\S 7]{GS2}}]\label{lem:find-tau-and-m}
  There is an algorithm which, given a finite presentation of an
  $\fgn$-group $G$, will find the following:
  \begin{itemize}
  \item a finite set of words representing the set $\tau(G)$,
  \item an integer $m$ such that $G^m$ is a characteristic subgroup
    and $G^m \cap \tau(G) = \{1\}$ (here $G^m$ denotes set of $m$-th
    powers of $G$.)
  \end{itemize}
\end{lem}

Let $\bar G = G /\tau(G)$ and let $\tilde G = G/G^m$ where $m$ is
given as Lemma \ref{lem:find-tau-and-m}. By choice of $m$, we have an
embedding $\iota: G \into \bar G \times \tilde G$.

\begin{cor}[{see \cite[\S 7]{GS2}}]\label{cor:construct-embedding}
  There is an algorithm which, given a finite presentation of an
  $\fgn$-group $G$, will
  \begin{itemize}
  \item find a finite presentation for $\bar G$,
  \item write out a multiplication table for $\tilde G$,
  \item construct an explicit embedding $\iota: G \into \bar G \times \tilde G$.
  \end{itemize}
\end{cor}

Denote $\Gamma = \aut{\bar G}$ and $\Delta = \aut{\tilde G}$. Let
$\Gamma \times \Delta$ act on $\bar G \times \tilde G$ componentwise.

\begin{lem}[{see \cite[\S 7]{GS2}}]\label{lem:membership-criterion}
  Let $\pi:\bar G \times \tilde G \onto \bar G / (\bar G)^m \times
  \tilde G$ be the canonical epimorphism.
  \begin{itemize}
  \item An element $(\bar g, \tilde h) \in \bar G \times \tilde G$
    lies in $\iota(G)$  if and only if $\pi((\bar g, \tilde h)) \in \pi \circ
    \iota(G)$.
  \item The canonical map $\alpha: \aut G \to \Gamma \times \Delta$ is an
    embedding and an element $(\gamma,\delta) \in \Gamma \times
    \Delta$ lies in $\alpha(\aut G)$ if and only if its canonical
    image $(\bar \gamma,\delta) \in \aut{\bar G/(\bar G)^m}
    \times \aut{\tilde G}$ maps $\pi\circ\iota(G)$ to itself.
  \end{itemize}
\end{lem}
From which we immediately conclude
\begin{cor}\label{cor:membership-criterion}
  The subgroups $\iota(G) \leq \bar G \times \tilde G$ and
  $\alpha(\aut G) \leq \Gamma \times \Delta$ are finite index, and
  this index is computable. Moreover there is an effective procedure
  to decide the membership problem for these subgroups
\end{cor}

The next step is to explicitly present $\Gamma \times \Delta$ and $\bar G
\times \tilde G$ as arithmetic groups. We start with a sensible linear
representation of finite groups.

Let $F$ be a finite group and denote by $S_F$ the set of permutations
of the set $F$. Let $e_i$ denote the $i$-th standard basis vector of
$\bbC^{|F|}$ and let $f \mapsto e_{i(f)}$ be some (set-theoretic)
embedding $F \into \bbC^{|F|}$. This gives a natural faithful
representation $\rho: S_F \into GL(|F|,\bbC)$ where $\sigma \in S_F$
is sent to the matrix that permutes the set of basis vectors
$\{e_1,\ldots,e_{|F|}\}$ accordingly.

We can embed $F \into S_F$ via the map $g \mapsto l_g$ where $l_g$ is
the permutation of $F$ induced by left multiplication. Each element of
$\aut F$ is already a permutation of $F$ so we put $\aut F \leq
S_F$. Note that for all $g,k \in F$ and $\delta \in \aut F$ we have
\begin{eqnarray*}
  \delta\circ l_g \circ \delta^\mo (k) &=& \delta \circ l_g
(\delta^\mo(k))\\  &=& \delta(g\delta^\mo(k))\\ &=&\delta(g)k \\ & =
& l_{\delta^\mo(g)}(k)
\end{eqnarray*} So in this representation the elements of $\aut F$ act
naturally on the image of $F$ via conjugation.

We can now embed $\bar G \times \tilde G$ into a group of matrices as
follows. Let $\Theta_{\bar G}: \bar G \into GL(n,\bbC)$ be the
embedding given in Theorem \ref{thm:theta-functor} and denote $H =
\Theta_{\bar G}(\bar G)$. Let $\rho: \tilde G \into GL(|\tilde
G|,\bbC)$ be the permutation representation described above. Then the
map \begin{eqnarray*} \psi: \bar
  G \times \tilde G & \into & GL(n+|\tilde G|,\bbC)\\
  (\bar g, \tilde h) & \mapsto & \diag{\Theta_{\bar G}(\bar
    g),\rho(\tilde h)}
\end{eqnarray*} has explicitly arithmetic image by Corollary
\ref{cor:tffgn-explicit-arithmetic} and the fact that $\tilde G$ is a
finite group.

Similarly if $\hat \Gamma$ is as given in Proposition
\ref{prop:arithmetic-automorphisms} then we also have an embedding:
\begin{eqnarray*}
  \mu: \hat\Gamma \times \Delta & \into &  GL(n+|\tilde G|,\bbC)\\
  (k,\delta) & \mapsto & \diag{k,\rho(\delta)}
\end{eqnarray*}
where $\rho: \Delta \into GL(|\tilde G|,\bbC)$ is the permutation
representation described above. $\mu$ also has an explicitly
arithmetic image. 

\begin{lem}\label{lem:fgn-explicitly-arith}
  The image $\psi\circ\iota(G) = K$ and the subgroup $\hat\Sigma$ of
  $\mu(\hat \Gamma \times \Delta)$ that normalizes $K$ and naturally
  surjects onto $\aut G$ via conjugation on $K$ are explicitly given
  arithmetic groups.
\end{lem}
\begin{proof}
  By Corollary \ref{cor:membership-criterion}, to show that $K$ is an
  explicitly given arithmetic group, it remains to provide an
  effective procedure to decide the membership problem for $K$ in
  $\psi(\bar G \times \tilde G)$. Let
  $g \in \psi(\bar G \times \tilde G)$.  By construction we have a
  finite presentation $\bk{Z \mid T}$ of $\bar G \times \tilde G$ so
  by exhaustive searching we will find some word $g(Z)$ such that
  $\psi(g(Z))=g$, and by Corollary \ref{cor:membership-criterion} we
  can check whether $g(Z) \in \iota(G)$ so we are done.

  Obviously $\mu(\hat\Gamma \times \Delta)$ normalizes
  $\psi(\bar G \times \tilde G)$ and by Theorem
  \ref{thm:theta-functor} and Proposition
  \ref{prop:arithmetic-automorphisms} there is a surjection
  $\mu(\hat \Gamma \times \Delta) \onto \Gamma \times \Delta$ induced
  by conjugation on $\psi(\bar G \times \tilde G)$. Now $\aut G$ sits
  inside $\Gamma \times \Delta$ as a finite index subgroup whose index
  is algorithmically bounded by Corollary
  \ref{cor:membership-criterion} thus the index of $\hat \Sigma$ in
  $\mu(\hat\Gamma \times \Delta)$ also has this bound. So again we
  only need a procedure to decide membership.

  Let $\sigma \in \mu(\hat\Gamma \times \Delta)$ and let
  $\{g_1,\ldots,g_s\}$ be the image in $\psi(\bar G \times \tilde G)$
  of a generating set of $\bar G \times \tilde G$. Then again by
  exhaustive searching we can find preimages of
  $\{g_1^\sigma,\ldots,g_s^\sigma\}$ in $\bar G \times \tilde G$ and
  thus by Corollary \ref{cor:membership-criterion} decide if $\sigma$
  normalizes $K$ and hence lies in $\hat\Sigma$.
\end{proof}

Finally we have:
\begin{proof}[Proof of Theorem \ref{thm:mwhp}]
  By Lemma \ref{lem:fgn-explicitly-arith} after substituting $\hat
  \Sigma$ in place of $\hat\Gamma$ and $K$ in place of $H$ all the
  arguments from Lemma \ref{lem:linear-semidirect} to Proposition
  \ref{prop:mwhp-tffgn} go through.
\end{proof}

\bibliographystyle{alpha} \bibliography{IP_NIL}

\end{document}